\documentclass[12pt]{amsart}
\usepackage{a4wide}
\usepackage[open, depth=2]{bookmark}
\usepackage{amsmath}
\usepackage{amsfonts}
\usepackage{amssymb}
\usepackage{array}
\usepackage{nicematrix}
\usepackage{mathrsfs}
\usepackage{pb-diagram}
\usepackage{color}
\usepackage{dynkin-diagrams}
\usepackage{epstopdf}
\usepackage{enumerate}
\usepackage{comment}
\usepackage{amsthm}
\usepackage{mathdots}
\usepackage{bbm}
\usepackage{tikz}
\usepackage{tkz-euclide}
\usepackage{cleveref}
\usepackage{stmaryrd}
\usepackage{epstopdf}
\newtheorem{lemma}{Lemma}[subsection]
\newtheorem{remark}[lemma]{Remark}
\newtheorem{theorem}[lemma]{Theorem}
\newtheorem{corollary}[lemma]{Corollary}

\newtheorem{definition}[lemma]{Definition}
\newtheorem{proposition}[lemma]{Proposition}

\newtheorem{assumption}[lemma]{Assumption}
\newtheorem{example}[lemma]{Example}
\usepackage{tikz, tikz-cd}
\usetikzlibrary{matrix}
\usetikzlibrary{shapes}
\usetikzlibrary{arrows}
\usetikzlibrary{calc,3d}
\usetikzlibrary{decorations,decorations.pathmorphing}
\usetikzlibrary{through}

\newtheorem{theoremintro}{Theorem}

\usepackage[style=alphabetic,maxnames=99,maxcitenames=99, maxalpha names=99, sorting=nyt,doi=true,url=false]{biblatex}

\newcommand{\CC}{\mathbb{C}}

\newcommand{\gotM}{\mathfrak{m}}
\newcommand{\gotN}{\mathfrak{n}}

\DeclareMathOperator{\irr}{Irr}

\newcommand{\supp}{\operatorname{supp}}
\newcommand{\C}{\mathbb{C}}

\newcommand{\m}{\mathfrak{m}}

\renewcommand{\subset}{\subseteq}

\newcommand{\id}{\operatorname{Id}}

\newcommand{\sgn}{\mathrm{sgn}}

\newcommand{\Hom}{\mathrm{Hom}}
\newcommand{\Rep}{\mathrm{Rep}}
\newcommand{\Irr}{\mathrm{Irr}}
\newcommand{\triv}{\mathrm{triv}}
\newcommand{\Ad}{\mathrm{Ad}}

\newcommand{\Lie}{\mathrm{Lie}}
\newcommand{\gkd}{\mathrm{GKdim}}
\newcommand{\spci}{\mathrm{spc}}
\newcommand{\spc}{\mathrm{Spc}}

\newcommand{\ind}{\mathrm{ind}}
\newcommand{\res}{\mathrm{res}}

\newcommand{\WF}{\mathrm{WF}}
\newcommand{\frb}{\mathrm{Fr}}

\newcommand{\ZZ}{Z_{\mathbf{G^\vee}}}

\usepackage[foot]{amsaddr}

\begin{document}

	\title[Weak Arthur packets are weakly spherical]{Ramification of weak Arthur packets for $p$-adic groups}
	
	
	\begin{abstract}

 Weak Arthur packets play a fundamental role in the study of the unitary dual and the automorphic spectrum of reductive Lie groups, and were recently introduced in the $p$-adic setting by Ciubotaru--Mason-Brown--Okada.
 
   
For split odd orthogonal and symplectic $p$-adic groups, we explicitly determine the decomposition of weak Arthur packets into Arthur packets that arise from endoscopic transfer. We establish a characterization of the Arthur packets that partake in such decompositions by means of ramification properties of their constituents.
  
A notion of weak sphericity for an irreducible representation is introduced: The property of containing fixed vectors with respect to a (not necessarily hyperspecial) maximal compact subgroup. We show that this property determines the weak Arthur packets in a precise sense.

    

As steps towards this description, we explore alignments between Langlands-type reciprocities for finite and $p$-adic groups, and their dependence on the geometry of the unipotent locus of the dual Langlands group.
Weak sphericity is shown to match with Lusztig's canonical quotient spaces that feature in the geometric theory for Weyl group representations, while the fine composition of weak Arthur packets is found to be governed by the partition of the unipotent locus into special pieces.

	\end{abstract}
	
	\author{Maxim Gurevich}
	\address{Department of Mathematics, Technion -- Israel Institute of Technology, Haifa, Israel.}
	\email{maxg@technion.ac.il}
	
	\author{Emile Okada}
	\address{Department of Mathematics, National University of Singapore, Singapore.}
	\email{okada@nus.edu.sg}
	
	\date{\today}
	
	\maketitle
 
\tableofcontents

\section{Introduction}

Weak Arthur packets, as defined via microlocal invariants as opposed to endoscopic transfer, were first studied in \cite{barbaschvogan} under the name of special unipotent representations.
Later introduced to the representation theory of real reductive Lie groups by Adams--Barbasch--Vogan \cite[Section 27]{abv}, weak Arthur packets decompose into a union of an important class of Arthur packets called \emph{unipotent Arthur packets}. Unipotent Arthur packets form the building blocks of Arthur type representations for both complex and real reductive groups by the work of Moeglin--Renard \cite{mr2,mr1}.



Given a reductive group $G$, selecting a unipotent conjugacy class $\mathcal{O}^\vee$ in its complex Langlands dual group  determines a corresponding unipotent infinitesimal character $\chi_{\mathcal{O}^\vee}$ for $G$-representations. In brief, the weak Arthur packet consists of the irreducible $G$-representations whose Gelfand--Kirillov dimension is minimal among those admitting $\chi_{\mathcal{O}^\vee}$ as their infinitesimal character.

Our work deals with the composition of the analogous notion for reductive $p$-adic groups of classical type. We show (Theorem \ref{thm:intro1}) that in these  instances, weak Arthur packets can be \emph{characterized} by examining restrictions of constituents of (endoscopic) Arthur packets to maximal compact subgroups.
We also show (Theorem \ref{thm:B}) that the (endoscopic) Arthur packets contained in the weak Arthur packet admit a geometric parameterization in terms of the unipotent locus of the Langlands dual group.

Let $F$ be a non-Archimedean local field of characteristic $0$, with a large enough residue characteristic (see Section \ref{sec:maxcpcts}). Let $G$ be an $F$-split symplectic or special odd orthogonal group. Its Langlands dual group is the complex reductive group 
\[
\mathbf{G^\vee} = \left\{  \begin{array}{ll}  \mathbf{Sp}_{2n} &\;,\mbox{ if } G = \mathrm{SO}_{2n+1}(F) \\  
\mathbf{SO}_{2n+1} &\;,\mbox{ if } G = \mathrm{Sp}_{2n}(F)
\end{array}   \right.\;.
\]

Let $\mathcal{U}^\vee$ denote the finite set of unipotent conjugacy classes in $\mathbf{G^\vee}(\mathbb{C})$.
We fix $s_G =1$ for the case of symplectic $G$, and $s_G=-1$ for an odd orthogonal $G$.

\subsection{What are weak Arthur packets for $G$ ?}

The theory of local Arthur packets \cite{Abook} attaches to each \textit{$A$-parameter} $\psi\in \Psi(G)$ a finite subset $\Pi_{\psi}^{A}$ of $\irr(G)$, the collection of isomorphism classes of irreducible smooth complex $G$-representations.
We recall that $A$-parameters for the group $G$ may be viewed as $\mathbf{G^\vee}(\mathbb{C})$-conjugation classes of continuous group homomorphisms
\[
\psi:W_F\times \mathrm{SL}_2(\mathbb C) \times \mathrm{SL}_2(\mathbb C) \to  \mathbf{G^\vee}(\mathbb{C})\;,
\]
whose restriction to $\mathrm{SL}_2(\mathbb C)\times \mathrm{SL}_2(\mathbb C)$ is algebraic, while $\psi(W_F)$ is bounded and consists of semisimple elements. Here, $W_F$ is the Weil group of the local field $F$.

The local Langlands reciprocity attaches an \textit{infinitesimal character} $\chi_{\pi}$ to each representation $\pi\in \irr(G)$. This is a $\mathbf{G^\vee}(\mathbb{C})$-conjugation class of a homomorphism
\[
\chi_{\pi}: W_F \to \mathbf{G^\vee}(\mathbb{C})\;.
\]
We write $\irr_{\chi}(G)\subset \irr(G)$ for irreducible representations $\pi$ admitting $\chi_{\pi}= \chi$.

A known property of local Arthur packets is that representations within a single packet share a common infinitesimal character. In other words, for each $\psi\in \Psi(G)$ a map $\chi_{\psi}$ exists, so that $\Pi^{A}_{\psi}\subset \irr_{\chi_{\psi}}(G)$. 
One distinguished family in $\Psi(G)$ are the basic unipotent $A$-parameters, defined to be those whose restriction to the Weil-Deligne group $W_F\times \mathrm{SL}_2(\mathbb C)$ is trivial. 
Basic unipotent $A$-parameters are visibly in bijection with algebraic homomorphisms $\mathrm{SL}_2\to \mathbf{G^\vee}$, and thus, by the Jacobson-Morozov theorem, are indexed by the classes in $\mathcal{U}^\vee$: $\{\psi_{\mathcal{O}^\vee}\}_{\mathcal{O^\vee}\in \mathcal{U}^\vee}\subset \Psi(G)$.
To each class $\mathcal{O^\vee}\in \mathcal{U}^\vee$ we write the associated \textit{basic unipotent} Arthur packet 
\[
\Pi_{\mathcal{O^\vee}}:= \Pi^{A}_{\psi_{\mathcal{O^\vee}}}\subset \irr(G)\;,
\]
and its infinitesimal character as $\chi_{\mathcal{O}^\vee} := \chi_{\psi_{\mathcal{O}^\vee}}$.

A basic unipotent Arthur packet $\Pi_{\mathcal{O}^\vee}$ consists of all \textit{anti-tempered} (that is, Aubert-dual to tempered) representations in $\irr_{\chi_{\mathcal{O}^\vee}}(G)$. 
Since special odd orthogonal groups are not simply connected and admit a non-trivial complex character, a curious complexity appears in our study of that case. Considering that quadratic unramified character $\kappa_0$ of $G$ that arises from the \textit{spinor norm}, we write, in the orthogonal case, $\kappa_0\otimes \Pi\subset \Irr(G)$ for the set of representations obtained by $\kappa_0$-twisting of a given set $\Pi\subset \Irr(G)$.
We say that the resulting Arthur packet 
\[
\Pi_{-1,\mathcal{O}^\vee}: = \kappa_0 \otimes\Pi_{\mathcal{O}^\vee}
\]
is \textit{quasi-basic unipotent}, and record its infinitesimal character as $\chi_{-1,\mathcal{O}^\vee}$.
We also mark $\Pi_{1,\mathcal{O}^\vee} = \Pi_{\mathcal{O}^\vee}$ and $\chi_{1,\mathcal{O}^\vee}=\chi_{\mathcal{O}^\vee}$ in all cases.

%
%



The second ingredient needed for the definition of weak Arthur packets is the Gelfand--Kirillov dimension of an irreducible representation. As commonly done in the $p$-adic setting, we extract this invariant out of the finer concept of the \textit{algebraic wavefront set}.

Let $\mathcal{N}_F$ be the set of nilpotent $\Ad(G)$-orbits in the Lie algebra $\Lie(G)$.
Briefly, the logarithm of the Harish-Chandra--Howe character \cite{hcbook} of a representation $\pi\in \irr(G)$ gives a distribution around $0\in \Lie(G)$, which is expanded as a linear combination $\sum_{\mathcal{O}\in \mathcal{N}_F} c_{\mathcal{O}}(\pi) \widehat{\mu}_{\mathcal{O}}$ of Fourier transforms of nilpotent orbital integrals $\{\widehat{\mu}_{\mathcal{O}}\}_{\mathcal{O}\in \mathcal{N}_F}$. 
The \textit{Gelfand-Kirillov dimension} of $\pi\in \irr(G)$ is then set to be 
\[
\gkd(\pi) = \max\left\{\frac12 \dim(\mathcal{O})\;:\; \mathcal{O}\in \mathcal{N}_F,\; c_{\mathcal{O}}(\pi)\neq 0  \right\}\;,
\]
where $\dim(\mathcal{O})$ is the Zariski dimension of the orbit as an algebraic variety.

\begin{definition}[Based on \cite{cmbo-arthur}]\label{defi:intro}
For a unipotent conjugacy class $\mathcal{O}^\vee\in \mathcal{U}^\vee$ and a sign $z\in \{\pm1\}$ ($z=1$, when $G$ is symplectic), its associated weak Arthur packet is constructed as 
\[
\Pi^{w}_{z,\mathcal{O}^\vee} = \left\{\pi\in \irr_{\chi_{z,\mathcal{O}^\vee}}(G)\;: \; \gkd(\pi) \leq \gkd(\sigma),\; 
\mbox{for all } \sigma \in \irr_{\chi_{z,\mathcal{O}^\vee}}(G) \right\}\;.
\]
\end{definition}

It was proved in \cite{cmbo-arthur} that an inclusion $\Pi_{1,\mathcal{O}^\vee}\subset \Pi_{1,\mathcal{O}^\vee}^{w}$ holds, for all $\mathcal{O}^\vee\in \mathcal{U}^\vee$. 
A follow-up conjecture \cite[Conjecture 3.1.2]{cmbo-arthur} was raised, whose claim is that in a natural analogy with the Lie groups case, each weak Arthur packet $\Pi^{w}_{1,\mathcal{O}^\vee}$ is a union of Arthur packets.
This conjecture, for the case of the classical groups in hand, is resolved as a consequence of our Theorem \ref{thm:B}, and also through a parallel work of \cite{liulo-weak}.







\subsection{The `Arthur-closure' of the weakly spherical spectrum}

Our main result phrases a new characterization of weak Arthur packets in terms of ramification properties of irreducible representations.

Let us note that basic unipotent infinitesimal characters $\{\chi_{z,\mathcal{O}^\vee}\}_{\mathcal{O}^\vee\in \mathcal{U}^\vee,z=\pm1}$ are unramified homomorphisms, that is, they are trivial on the inertia subgroup $I_F< W_F$. In particular, that designates the representations in $\irr_{\chi_{z,\mathcal{O}^\vee}}(G)$ as unipotent in the sense of Lusztig \cite{Lu-unip1}.
For each $\mathcal{O}^\vee\in\mathcal{U}^\vee$ and a possible sign $z=\pm1$, that also means that a choice of a hyperspecial open compact subgroup $K_0< G$ gives rise to a unique \textit{spherical} representation $\delta_{z,\mathcal{O}^\vee}\in \irr_{\chi_{z,\mathcal{O}^\vee}}(G)$, that is, a representation possessing a non-zero $K_0$-invariant vector. Indeed, the semisimple conjugacy class of the element $s_{z,\mathcal{O}^\vee} = \chi_{z,\mathcal{O}^\vee}(\frb)\in \mathbf{G^\vee}(\mathbb{C})$, $\frb\in W_F$ being a choice of a Frobenius element, is the Satake parameter that determines $\delta_{z,\mathcal{O}^\vee}$. 
It is known \cite[Proposition 6.4]{moeg09} that $\Pi_{z,\mathcal{O}^\vee}$ is the unique (spherical) Arthur packet containing the irreducible spherical representation $\delta_{z,\mathcal{O}^\vee}$.



We now formulate a characterization of weak Arthur packets further extending this phenomenon. 

\begin{definition}\label{def:weak-sphe}
    A representation $\pi\in \irr(G)$ is \textit{weakly spherical}, when a (any) maximal open compact subgroup $K<G$ exists, for which $\pi$ possesses a non-zero $K$-invariant vector.

When $G$ is odd orthogonal, a representation $\pi\in \irr(G)$ is \textit{$-1$-weakly spherical}, when $\kappa_0\otimes \pi$ is weakly spherical.
\end{definition}

Although literature typically emphasizes the hyperspecial case, due to its role in the automorphic context, the groups $G$ under consideration admit a finite sequence of non-conjugate maximal compact subgroups that are conveniently classified using the Iwahori-Matsumoto theory \cite{IM}.

Let $\Pi^{\mathrm{sph}}_{z,\mathcal{O}^\vee}\subset \Pi_{z, \mathcal{O}^\vee}$ be the subset of weakly spherical representations in the spherical Arthur packet. Hence, we are given a chain of containments
\begin{equation}\label{eq:intro-chain}
\delta_{z,\mathcal{O}^\vee}\in \Pi^{\mathrm{sph}}_{z,\mathcal{O}^\vee}\subset \Pi_{z,\mathcal{O}^\vee}\subset \Pi^{w}_{z,\mathcal{O}^\vee}\;.
\end{equation}
Recalling the overlapping nature of Arthur packets, the following definition is meaningful.
\begin{definition}
Let $\mathcal{O}^\vee\in \mathcal{U}^\vee$ be a unipotent conjugacy class, and $z\in\{\pm1\}$ be a sign ($z=1$, when $G$ is symplectic).

An $A$-parameter $\psi\in \Psi(G)$ with infinitesimal character $\chi_{\psi}= \chi_{z,\mathcal{O}^\vee}$, and its associated Arthur packet $\Pi_{\psi}\subset \irr_{\chi_{z,\mathcal{O}^\vee}}(G)$, are said to be \textit{weakly spherical} (or, $1$-weakly spherical), whenever
\[
\Pi_{\psi}\cap \Pi^{\mathrm{sph}}_{z,\mathcal{O}^\vee}\neq \emptyset
\]
holds. 
In the orthogonal case, the $A$-parameter $\psi$ and Arthur packet $\Pi_{\psi}$ are said to be \textit{$-1$-weakly spherical}, whenever the Arthur packet $\kappa_0 \otimes \Pi_{\psi}\subset \irr_{\chi_{-z,\mathcal{O}^\vee}}(G)$ is weakly spherical.
\end{definition}

In essence, we say that an Arthur packet is weakly spherical when it admits a quasi-basic unipotent infinitesimal character, and contains an anti-tempered weakly spherical representation.

We are now ready to state the main result.

\begin{theoremintro}\label{thm:intro1}
\begin{enumerate}
    \item  Let $\psi\in \Psi(G)$ be an $A$-parameter with infinitesimal character $\chi_{\psi}= \chi_{z,\mathcal{O}^\vee}$, for a conjugacy class $\mathcal{O}^\vee\in \mathcal{U}^\vee$, and a sign $z\in\{\pm1\}$ ($z=1$, when $G$ is symplectic).
Then, $\psi$ is $zs_G$-weakly spherical, if and only if, an inclusion of packets \[
\Pi^{A}_{\psi}\subset \Pi^{w}_{z,\mathcal{O}^\vee}
\]
holds.

    \item For each conjugacy class $\mathcal{O}^\vee\in \mathcal{U}^\vee$, the weak Arthur packet $\Pi^{w}_{z,\mathcal{O}^\vee}$ consists of the union of all $zs_G$-weakly spherical Arthur packets admitting the infinitesimal character $\chi_{z,\mathcal{O}^\vee}$.
    
\end{enumerate}

\end{theoremintro}

In particular, it follows that all constituents of weak Arthur packets are unitarizable representations.

The result of Theorem \ref{thm:intro1} naturally prompts the question of whether a comparable definition of weak Arthur packets for broader families of infinitesimal characters could be characterized through similar ramification properties. Progress in this area should be coupled with further advancement in connecting wavefront invariants with the Langlands reciprocity.


Our path towards the proof of Theorem \ref{thm:intro1} passes through two preparatory tasks of standalone interest.
One is an explicit description of representations inside weak Arthur packets in terms of $A$-parameters and the Langlands reciprocity.
The other is an explicit description of the anti-tempered weakly spherical representations of $\Pi^{\mathrm{sph}}_{s_G,\mathcal{O}^\vee}$ in terms of an \textit{enhanced} Langlands parameterization.

\subsection{Constituents of weak Arthur packets}

An explicit decomposition of weak Arthur packets will be given in terms of the geometry of the unipotent locus of $\mathbf{G^\vee}(\mathbb{C})$. 




%
%





For a class $\mathcal{O}^\vee\in \mathcal{U}^\vee$, let us denote as $\overline{\mathcal{O}^\vee} \subset \mathcal{U}^\vee$ the set of conjugacy classes contained in the topological closure of $\mathcal{O}^\vee$.
A prominent subset of conjugacy classes $\mathcal{U}^\vee_{\spci} \subset \mathcal{U}^\vee$ are known as the \textit{special} unipotent classes, as introduced in \cite[\S 13.1.1]{lusztigbook}
Pivoting around this concept, for any $\mathcal{O}^\vee\in \mathcal{U}^\vee$ we define its \textit{relative special piece} as


\[
\spc(\mathcal{O}^\vee) = \overline{\mathcal{O}^\vee} - \bigcup_{ \mathcal{O'}^\vee\in \mathcal U_{\spci}\,:\,\mathcal{O}^\vee\neq \mathcal{O'}^\vee\in \overline{\mathcal{O}^\vee}} \overline{\mathcal{O'}^\vee} \subset \mathcal{U}^\vee\;.
\]
Those sets, in our case of groups of classical type, are always of power-of-$2$ cardinality, while their inherited topological partial orders are of a hypercube lattice form. 
The resulting disjoint division
\[
\mathcal{U}^\vee = \bigsqcup_{\mathcal{O}^\vee\in \mathcal{U}^\vee_{\spci}} \spc(\mathcal{O}^\vee)
\]
of the unipotent locus into its \textit{special pieces} has been thoroughly studied in the context of perverse sheaves and singularity theory by \cite{lusztiggreenpolys,spal-book,,kp-aster,sommersspecialpieces}. In particular, explicit combinatorial descriptions of the special pieces for classical groups in terms of the standard parameterization of $\mathcal{U}^\vee$ by partitions trace back to \cite{kp-aster}.







For an effective description of weak Arthur packets we now need to recall further details of the local Langlands reciprocity. Indeed, we take this reciprocity as an established consequence of the Arthur theory for classical groups.

The reciprocity provides a map $\pi \mapsto \phi(\pi)$ from $\irr(G)$ onto $\Phi(G)$, the collection of $L$-parameters for the group $G$. Each fiber of the reciprocity map over a given $L$-parameter $\phi\in \Phi(G)$ is known as an $L$-packet 
\[
\Pi_{\phi}\subset \irr_{\chi_{\phi}}(G)\;.
\]
Here, in similarity with Arthur packets, $\chi_{\phi}$ is the infinitesimal character shared by all constituents of $\Pi_{\phi}$.
An $L$-parameter may be viewed as the $\mathbf{G^\vee}(\mathbb{C})$-conjugation class of a continuous group homomorphism
$$\phi:W_F\times \mathrm{SL}_2(\mathbb C) \to  \mathbf{G^\vee}(\mathbb{C})\;,$$
whose restriction to $\mathrm{SL}_2(\mathbb C)$ is algebraic, while $\phi(W_F)$ consists of semisimple elements.
For $\phi\in \Phi(G)$, we take note of the unipotent element
\[
u_\phi = \phi\left(1,\begin{pmatrix}
    1 & 1 \\ 0 & 1
\end{pmatrix}\right)\in \mathbf{G^\vee}(\mathbb{C})\;,
\]
and write $\mathcal{O}_{\phi}^\vee\in \mathcal{U}^\vee$ for the conjugacy class containing $u_{\phi}$.

This invariant is of high relevance for our discussion as, for example, exhibited by the main result of \cite{cmowavefront}. The latter states (Theorem \ref{thm:cmbo}) that for certain $L$-parameters $\phi\in \Phi(G)$ (i.e. having an unramified and real $\chi_{\phi}$),
the algebraic wavefront sets of all representations Aubert-dual to constituents of $\Pi_{\phi}$ are determined by $\mathcal{O}^\vee_{\phi}$ through means of Barbasch--Vogan--Lusztig--Spaltenstein duality.

Finally, we recall that the collection of $A$-parameters is naturally embedded $\Psi(G)\hookrightarrow \Phi(G)$, $\psi \mapsto \phi_{\psi}$, into the collection of $L$-parameters, as a sub-collection that is often regarded as $L$-parameters of \textit{Arthur-type}. In this setup, an inclusion $\Pi_{\phi_\psi}\subset \Pi^{A}_{\psi}$ of packets holds, for all $\psi\in \Psi(G)$.

\begin{proposition}\label{prop:intro}

For any $\mathcal{O}^\vee_1\in \mathcal{U}^\vee$ and $\mathcal{O}^{\vee}_2\in \spc(\mathcal{O}^\vee_1)$, and a possible sign $z\in\{\pm1\}$ there is a unique $L$-parameter $\phi=\phi_{z,\mathcal{O}^\vee_1, \mathcal{O}^{\vee}_2}\in \Phi(G)$ which satisfies $\chi_{\phi} = \chi_{z,\mathcal{O}^\vee_1}$ and $\mathcal{O}^\vee_{\phi} = \mathcal{O}^{\vee}_2$, and it is of Arthur-type.
\end{proposition}

While $L$-packets may not be well-behaved with respect to the Aubert involution $\pi\mapsto \pi^t$ on $\irr(G)$, Arthur packets are better situated. This is arithmetically manifested in the involution $\psi\mapsto \psi^t$ on $\Psi(G)$ given by a transposition of the pair of $SL_2(\mathbb{C})$ components of the parameter (see \cite[Theorem 1.6]{atobe-crelle1}).

Considering any $L$-parameter $\phi=\phi_{z,\mathcal{O}^\vee_1, \mathcal{O}^{\vee}_2}=\phi_{\psi}$ of the form stated in Proposition \ref{prop:intro}, with corresponding $\psi=\psi_{z,\mathcal{O}^\vee_1, \mathcal{O}^{\vee}_2}\in \Psi(G),$ we denote the Arthur packet associated with its transposed parameter as
\begin{equation}\label{eq:introApack}
\Pi_{z,\mathcal{O}^\vee_1,\mathcal{O}^{\vee}_2} := \Pi^{A}_{\psi^t}\subset \irr_{\chi_{z,\mathcal{O}^\vee}}(G)\;.
\end{equation}
In particular, we see that for any $\mathcal{O}^\vee\in \mathcal{U}^\vee$,  $\psi_{z,\mathcal{O}^\vee,\mathcal{O}^\vee}^t= \psi_{z,\mathcal{O}^\vee}$ is the quasi-basic unipotent parameter, while $\Pi_{z,\mathcal{O}^\vee,\mathcal{O}^\vee}= \Pi_{z,\mathcal{O}^\vee}$ is the associated anti-tempered, or spherical, Arthur packet.

\begin{theoremintro}\label{thm:B}
For any unipotent conjugacy class $\mathcal{O}^\vee\in \mathcal{U}^\vee$ and a sign $z\in\{\pm1\}$ ($z=1$, when $G$ is symplectic), we have 

\begin{enumerate}
    \item  a decomposition into a (disjoint) union of $L$-packets
\[
(\Pi^{w}_{z,\mathcal{O}^\vee})^t = \bigsqcup_{\mathcal{O}^{1\vee} \in \spc(\mathcal{O}^\vee)} \Pi_{\phi_{z,\mathcal{O}^\vee,\mathcal{O}^{1\vee}}}\;,
\]
where $(\Pi^{w}_{z,\mathcal{O}^\vee})^t = \{\pi^t\;:\;\pi\in\Pi^{w}_{z,\mathcal{O}^\vee}\}$ is the Aubert duality image of the weak Arthur packet.

\item a (non-disjoint) decomposition into a union of local Arthur packets
\[
\Pi^{w}_{z,\mathcal{O}^\vee} = \bigcup_{\mathcal{O}^{1\vee} \in \spc(\mathcal{O}^\vee)} \Pi_{z,\mathcal{O}^\vee,\mathcal{O}^{1\vee}}\;.
\]

\end{enumerate}

\end{theoremintro}

As mentioned, Theorem \ref{thm:B} should be compared with the parallel work in \cite{liulo-weak} and could be viewed as its explication.
We also refer to \cite{bmsz} which is a study of weak Arthur packets for classical groups in the Archimedean case which is similar in spirit to Theorem \ref{thm:B}.
Finally we note that the Arthur parameters that appear in the union are of a very particular type, as the Aubert-dual of their constituents all have $L$-parameters of Arthur type.



\subsection{Weakly spherical constituents of the spherical Arthur packet}

Another aspect of this study describes the location of weakly spherical representations within the spherical Arthur packet $\Pi_{z,\mathcal{O}^\vee}$ and the characterization of all additional Arthur packets that may contain these representations. 
While essential for the proof of Theorem \ref{thm:intro1}, it is also notable that our findings (Section \ref{sect:redtospringer}) relate this problem with the Springer correspondence for the (finite) Weyl group $W_G$ of $\mathbf{G^\vee}$.

 
Each conjugacy class $\mathcal{O}^\vee\in \mathcal{U}^\vee$ comes equipped with a component group $A(\mathcal{O}^\vee) = Z(u)/Z(u)^{\circ}$,
that is, the group of connected components of the centralizer subgroup $Z(u)< \mathbf{G^\vee}(\mathbb{C})$ of a representative $u\in \mathcal{O}^\vee$. 
It is a finite $2$-group, whose character group $\widehat{A(\mathcal{O}^\vee)}$ parameterizes constituents of the Arthur packet $\Pi_{\mathcal{O}^\vee}$. That is the scope of the enhanced Langlands reciprocity, that may be constructed through two distinct, yet interplaying, approaches.

The first is inherent in the design of Arthur's theory for classical groups, which produces local Arthur packets by means of endoscopic transfer. Pending a choice of a Whittaker datum for $G$, a parameterization
\begin{equation}\label{eq-intro}
\Pi_{z,\mathcal{O}^\vee} = \left\{\delta(z,\mathcal{O}^\vee,\epsilon)\;:\; \epsilon\in \widehat{A(\mathcal{O}^\vee)}_0\right\}
\end{equation}

for the anti-tempered Arthur packet is canonically assigned, where $\widehat{A(\mathcal{O}^\vee)}_0 < \widehat{A(\mathcal{O}^\vee)}$ is a specified subgroup (of index $1$ or $2$).
Here, the trivial character $\epsilon=\mathrm{trv}$ would give the aforementioned  spherical representation $\delta(z,\mathcal{O}^\vee,\mathrm{trv}) = \delta_{z,\mathcal{O}^\vee}\in \Pi_{z,\mathcal{O}^\vee}$.



A second approach traces back to the Kazhdan-Lusztig \cite{KL} geometric construction of the irreducible spectrum of affine Hecke algebras, which established Langlands reciprocity for $G$-representations in the principal Bernstein block. 
Lusztig in \cite{Lu-unip1}, using further geometric methods, later gave a parameterization for all unipotent representations in $\irr(G)$, providing further means for reinterpreting the parameterization in \eqref{eq-intro}.

The issue of matching between Lusztig's and Arthur's parameterizations of the enhanced (unipotent) Langlands reciprocity was treated in \cite{waldspurgerendoscopy} for the case of odd orthogonal groups, via pinning of endoscopic identities.
In our analysis we rely on Assumption \ref{assumpt:param} for a similar expected match in the symplectic case, which remains to be clarified due to an apparent gap in existing literature.

\subsubsection{The Springer leap}

The geometric point of view brings this discussion nearer to the role of the group $\widehat{A(\mathcal{O}^\vee)}$ as irreducible local systems on the variety $\mathcal{O}^\vee$.
This is the perspective of the Springer correspondence. It attaches to each irreducible (complex) $W_G$-representation, a pair $(\mathcal{O}_{\sigma}^\vee, \epsilon_{\sigma})$, consisting of a conjugacy class $\mathcal{O}_{\sigma}^\vee\in \mathcal{U}^\vee$ and an irreducible local system $\epsilon_{\sigma}\in \widehat{A(\mathcal{O}_{\sigma}^\vee)}$ on it.

From a separate angle Lusztig provides in \cite[Section 4.2]{lusztigbook} a division of the set of isomorphism classes of irreducible $W_G$-representations into \textit{families}
\[
\irr(W_G) = \bigsqcup_{\mathfrak{c}} \irr_{\mathfrak{c}} (W_G)\;,
\]
according to the two-sided Kazhdan-Lusztig cell $\mathfrak{c}$ on which the representation is supported.
A part of Lusztig's theory sets up a bijection between those cells and the set of special orbits of $\mathcal{U}^\vee$ via the Springer correspondence. Accordingly, for $\mathcal{O}^\vee\in \mathcal{U}^\vee$, we write $\mathfrak{c}(\mathcal{O}^\vee)$ for the cell associated with the special piece to which $\mathcal{O}^\vee$ belongs.

Yet, detection of familial affiliations of representations in $\irr(W_G)$ in terms of their Springer parameters is a subtle issue.
To that aim the subsets 
\begin{equation}\label{eq:ach-sg}
A^{\dagger}(\mathcal{O}^\vee)= \left\{ \epsilon\in \widehat{A(\mathcal{O}^\vee)}\;:\; \exists \sigma \in \irr_{\mathfrak{c}( \mathcal{O}^\vee)}(W_G)\mbox{ s.t. }(\mathcal{O}_{\sigma}^\vee, \epsilon_{\sigma})= (\mathcal{O}^\vee,\epsilon)\right\}
\end{equation}
of the character groups $\widehat{A(\mathcal{O}^\vee)}$ are defined.
It was shown by Achar-Sage \cite{ach-sage}, that $A^{\dagger}(\mathcal{O}^\vee)$ is in fact an explicitly described subgroup, dual to what is known as \textit{Lusztig's canonical quotient} of the component group $A(\mathcal{O}^\vee)$.

\subsubsection{Characterization of weak sphericity}

Returning to our original question, we now present an answer of similar nature in terms of the enhanced Langlands parameterization.

\begin{theoremintro}\label{thm:C}

For any unipotent conjugacy class $\mathcal{O}^\vee\in \mathcal{U}^\vee$ and a sign $z\in\{\pm1\}$, the set of anti-tempered $zs_G$-weakly spherical representations in $\irr_{\chi_{z, \mathcal{O}^\vee}}(G)$ is given as
\[
 \left\{  \delta(z, \mathcal{O}^\vee,\epsilon)\;:\;
\epsilon \in A^{\dagger}(\mathcal{O}^\vee)
\right\}\subseteq\Pi_{z, \mathcal{O}^\vee}\;,
\]
in terms of the parameterization of \eqref{eq-intro}.
\end{theoremintro}

It follows from the analysis of Achar-Sage that canonical group embeddings 
\[
\iota_{\mathcal{O}^\vee_1,\mathcal{O}^\vee}: A^{\dagger}(\mathcal{O}^\vee_1)\hookrightarrow A^{\dagger}(\mathcal{O}^\vee)
\]
exist, for any $\mathcal{O}^\vee\in \mathcal{U}^\vee$ and any $\mathcal{O}^\vee_1 \in \spc(\mathcal{O}^\vee)$.
They are compatible with the partial order on $\mathcal{U}^\vee$, in the sense that 
\[
\iota_{\mathcal{O}^\vee_1,\mathcal{O}^\vee}\circ\iota_{\mathcal{O}^\vee_2, \mathcal{O}^\vee_1} = \iota_{\mathcal{O}^\vee_2,\mathcal{O}^\vee}
\]
holds, whenever $\mathcal{O}^\vee_2\in \spc(\mathcal{O}^\vee_1)$ and $\mathcal{O}^\vee_1\in \spc(\mathcal{O}^\vee)$.

Note that the relation $\mathcal O_1^\vee\in \spc(\mathcal O^\vee)$ determines a partial order on $\mathcal U^\vee$ obtained from the topological partial order by removing from the Hasse diagram edges between orbits in different special pieces, and so in particular the relation is transitive. 

\begin{definition}\label{defi:prim}
Given $\mathcal{O}^\vee\in \mathcal{U}^\vee$ and $\mathcal{O}^\vee_1 \in \spc(\mathcal{O}^\vee)$, we say that a character $\epsilon \in A^{\dagger}(\mathcal{O}^\vee)$ is \textit{$\mathcal{O}^\vee_1$-primitive}, when there is no class $\mathcal{O}^\vee_1 \leq \mathcal{O}^\vee_2\lneq \mathcal{O}^\vee$ (in the topological partial order on $\mathcal{U}^\vee$) with $\epsilon\in \mathrm{Im}(\iota_{\mathcal{O}^\vee_2,\mathcal{O}^\vee})$.
\end{definition}


The following theorem fully characterizes the set of Arthur packets that contain a given weakly spherical anti-tempered representation that admits a quasi-basic infinitesimal character.

\begin{theoremintro}\label{thm:D}
Let $\mathcal{O}^\vee\in \mathcal{U}^\vee$ be a unipotent conjugacy class, $z\in \{\pm1\}$ a possible sign, and 
\[
\delta = \delta(z,\mathcal{O}^\vee,\epsilon)\in \Pi_{z,\mathcal{O}^\vee}
\]
an anti-tempered $zs_G$-weakly spherical representation (according to Theorem \ref{thm:C}), parameterized by a character $\epsilon\in  A^{\dagger}(\mathcal{O}^\vee)$.
Let $\Pi = \Pi^{A}_{\psi}\subset \irr(G)$ be an Arthur packet, associated to an $A$-parameter $\psi\in \Psi(G)$. 

Then, an inclusion $\delta\in \Pi$ is valid, if and only if, $\Pi = \Pi_{z,\mathcal{O}^\vee, \mathcal{O}^\vee_1}$, for a class $\mathcal{O}^\vee_1\in \spc(\mathcal{O}^\vee)$, for which $\epsilon$ is $\mathcal{O}^\vee_1$-primitive.








\end{theoremintro}

\subsection{Example: Triangular partitions}

One appealing family of examples for our analysis appears when considering the dual unipotent conjugacy class $\mathcal{O}^\vee$ in $\mathbf{SO}_{4k(k+1)+1}(\mathbb{C})$ that is indexed by the triangular partition $(4k+1, 4k-1, 4k-3\ldots,1)$, for $k\geq1$.

The spherical Arthur packet $\Pi_{\mathcal{O}^\vee}\subset \irr(\mathrm{Sp}_{4k(k+1)}(F))$ in this case contains $\left|\widehat{A(\mathcal{O}^\vee)}_0\right|= 4^k$ representations. 
A total of $\left|\Pi_{\mathcal{O}^\vee}^{s}\right|= \left|A^{\dagger}(\mathcal{O}^\vee)\right|= 2^k$ out of them are weakly spherical. Those weakly spherical representations can be found as constituents of a total of $|\spc(\mathcal{O}^\vee)| =2^k$ distinct weakly spherical Arthur packets. 
Counting the union of those Arthur packets brings us to $5^k$ distinct constituents of the weak Arthur packet $\Pi^{w}_{\mathcal{O}^\vee}$.

Let us also consider the case of $G= \mathrm{Sp}_8(F)$ in greater precision. Here, $\mathcal{O}_{135}^\vee$ is the unipotent conjugacy class indexed by the partition $(135)$ in the Langlands dual group $\mathbf{SO}_9(\mathbb{C})$. 
Viewed as a representation of $W_F\times \mathrm{SL}_2(\mathbb{C}) \times \mathrm{SL}_2(\mathbb{C})$, the associated basic unipotent $A$-parameter  is given as
\[
\psi_{\mathcal{O}_{135}^\vee} = (1\otimes \nu_1\otimes \nu_1) \oplus (1\otimes \nu_1\otimes \nu_3) \oplus (1\otimes \nu_1\otimes \nu_5)\;,
\]
where $\nu_k$ is the $k$-dimensional irreducible $\mathrm{SL}_2(\mathbb{C})$-representation, while $1$ denotes the trivial representation of $W_F$.
The tempered Arthur packet $\Pi_{\phi_{\mathcal{O}_{135}^\vee,\mathcal{O}_{135}^\vee}}$ is known to contain a supercuspidal representation $\pi_{sc}$. Indeed, in \cite[Section 3.4]{atobe-crelle2} a total of $9$ distinct Arthur packets were exhibited to contain $\pi_{sc}$. 

Since supercuspidal irreducible representations are self-Aubert-dual, one of those Arthur packets is the anti-tempered packet $\Pi_{\mathcal{O}_{135}^\vee}$. In further detail, the packet has $4$ constituents
\[
\Pi_{\mathcal{O}_{135}^\vee} = \left\{ \begin{array}{rl} \delta_{\mathcal{O}_{135}^\vee}=\delta(\mathcal{O}_{135}^\vee,+++), & \pi_1=\delta(\mathcal{O}_{135}^\vee,--+), \\ \pi_2= \delta(\mathcal{O}_{135}^\vee,+--), & \pi_{sc}=\delta(\mathcal{O}_{135}^\vee,-+-)\end{array}\right\}\;,
\]
including the spherical representation $\delta_{\mathcal{O}_{135}^\vee}$. 
The group of characters $\widehat{A(\mathcal{O}_{135}^\vee)}_0$ is identified with a subgroup of $(\mathbb{Z}/2\mathbb{Z})^3$, read as signs attached to each part of the partition. The subgroup $A^{\dagger}(\mathcal{O}_{135}^\vee)$ is of order $2$, marking, by Theorem \ref{thm:C},
\[
\Pi_{\mathcal{O}_{135}^\vee}^{\mathrm{sph}} = \{ \delta_{\mathcal{O}_{135}^\vee}, \pi_1\}
\]
as the set of weakly spherical representations in $\Pi_{\mathcal{O}_{135}^\vee}$. 

Now, the relative special piece is given as $\spc(\mathcal{O}_{135}^\vee) = \{\mathcal{O}_{135}^\vee, \mathcal{O}_{144}^\vee\}$. The additional Arthur-type $L$-parameter 
\[
\phi_{\mathcal{O}_{135}^\vee, \mathcal{O}_{144}^\vee} = (1\otimes \nu_1) \oplus (q^{1/2} \otimes \nu_4) \oplus (q^{-1/2} \otimes \nu_4)\in \Phi(G)
\]
produces a singleton $L$-packet $\Pi_{\phi_{\mathcal{O}_{135}^\vee, \mathcal{O}_{144}^\vee}}= \{\tau\}$, while the resulting weak Arthur packet then consists of the $5$ representations
\[
\Pi_{\mathcal{O}_{135}^\vee}^{w}= \Pi_{\mathcal{O}_{135}^\vee} \cup \Pi_{ \mathcal{O}_{135}^\vee, \mathcal{O}_{144}^\vee}= \{\delta_{\mathcal{O}_{135}^\vee} ,\pi_1, \pi_2, \pi_{sc}, \tau^t\}\;.
\]
The additional weakly spherical Arthur packet $\Pi_{\mathcal{O}_{135}^\vee, \mathcal{O}_{144}^\vee}= \{\pi_1, \pi_{sc}, \tau^t\}$ is given by the $A$-parameter
\[
\psi^t_{\mathcal{O}_{135}^\vee, \mathcal{O}_{144}^\vee} = (1\otimes \nu_1\otimes \nu_1) \oplus (1\otimes \nu_2 \otimes \nu_4)\in \Psi(G)\;.\]
Indeed, the inclusion $\pi_1 \in \Pi_{\mathcal{O}_{135}^\vee, \mathcal{O}_{144}^\vee}$ follows, by Theorem \ref{thm:D}, from $\mathcal{O}_{144}^\vee$-primitivity of the element $(--+)\in A^{\dagger}(\mathcal{O}_{135}^\vee)$, since $\widehat{A(\mathcal{O}_{144}^\vee)}_0$ is a trivial group.

\subsection{Methods}

Theorem \ref{thm:intro1} is in fact an immediate corollary of the combination of Theorems \ref{thm:B},\ref{thm:C}, and \ref{thm:D}.
Each of these latter three results necessitates the application of distinct toolkits of recently developed techniques, which we will now outline.

For Theorem \ref{thm:B}, Gelfand--Kirillov dimensions of representations are extracted out of the nilpotent orbits that comprise the algebraic wavefront invariant. Here, we utilize the recent advancements in \cite{cmowavefront} that provide an explicit knowledge of these invariants for certain cases of unipotent
representations.


The proof of Theorem \ref{thm:D} hinges on techniques within the combinatorial theory of Arthur packet intersections, as developed in recent years by Xu \cite{xu21} and Atobe \cite{atobe-crelle1,atobe-crelle2}. For given tempered representations, whose parameters are associated with the Lusztig canonical quotient, we apply this theory to exhaust all possible Arthur packets that may contain a specified representation.
To that aim we recall the essentials of the theory of Moeglin parameters for Arthur packets in Sections \ref{sect:Aexplicit} and \ref{sect:moeg-param}, narrowing our focus to what we term \textit{near-tempered} Arthur packets in Section \ref{sect:near-temp}. This analysis suffices to complete the proofs of Theorems \ref{thm:B} and \ref{thm:D} in Sections \ref{sect:wArthur} and \ref{sect:weak-sph}, respectively.

Theorem \ref{thm:C} is tackled through categorical equivalences and deformation techniques, facilitating a direct reduction to the representation theory of finite Weyl groups. This is the concern of the last three sections.
Section \ref{sect:redtospringer} presents a full proof scheme for Theorem \ref{thm:C}.
The main tools are the Borel-Casselman equivalence between the category of Iwahori-spherical representations and the module category of the Iwahori-Hecke algebra, and Iwahori and Matsumoto's classification of maximal compact subgroups.
Section \ref{sec:hecke} elaborates on how the finite-dimensional module categories of Hecke algebras, through the Kazhdan--Lusztig construction, provide the proof for the pivotal reduction, Theorem \ref{thm:mainSpringer}.
A critical aspect that arises is the fact that the convolution algebras attached to maximal compact subgroups may lie in the class of \textit{extended} (finite) Hecke algebras. To harness the full strength of deformation techniques for computations of invariants, we reproduce, in Section \ref{sec:finiteheckealgebras}, the theory of Lusztig's asymptotic algebras in the extended case.
Lastly, Section \ref{sect:walds-la} delves into the representation theory of finite signed permutation groups, that is, the Weyl groups of classical Lie type.  Its goal is to prove Theorem \ref{prop:combinatoricsprop}, the final ingredient in the proof of Theorem \ref{thm:C}, which relates to the decomposition of (full) Springer fibre representations into irreducible constituents. 
Indeed, such decompositions, referred to as Green theory for their links with classical Green functions, are typically challenging to access. Yet, recent advancements by Waldspurger \cite{waldspurger} and La \cite{la} supply a novel algorithmic approach. We employ their results to exhibit a link (Proposition \ref{prop:dagger}) between those Springer representations whose parameterization is associated with Lusztig's canonical quotient and a Weyl group analogue of weak-sphericity.

\subsection{Acknowledgements}

The authors wish to express their gratitude to Dan Ciubotaru. 
Early investigations into weakly spherical representations including bounds on the $L$-parameters and a closed form of Proposition \ref{prop:firstreduction} were worked out in a private correspondence with him. We appreciate his permission to include the latter result in this work.


Special acknowledgments are owed to several individuals: Nadya Gurevich for her insights on minimizers of the Gelfand--Kirillov dimension, Lei Zhang for elucidating emerging tools connecting Arthur theory with wavefront invariants, Shilin Yu and Daniel Wong for highlighting the embedding of special pieces into Lusztig's canonical quotient, and Lucas Mason-Brown for sharing expertise on Arthur packets for real groups. Additional appreciation goes to Bin Xu, Chengbo Zhu, Jia Jun Ma, and Wee Teck Gan for engaging discussions that contributed to this work.

Investigations into the themes that led up to this work began at the January 2023 "Representation Theory, Combinatorics, and Geometry" program hosted by the Institute for Mathematical Sciences in Singapore, which provided valuable opportunities for the exchange of ideas. We also thank the Mathematics departments at the National University of Singapore and the Technion Institute for their subsequent hospitality and support, during reciprocal visits.

This research is supported by the Israel Science Foundation (Grant Number: 2542/25).

\section{Background}
\subsection{Classical $p$-adic groups}\label{sec:maxcpcts}

Repeating some notions, we fix $F$, a non-Archimedean local field of characteristic $0$. We write $\mathfrak{p}_F < \mathfrak{O}_F < F$ for its ring of integers, and the maximal ideal of the ring. The residue field $\mathfrak{O}_F/ \mathfrak{p}_F$ is finite of size $q=p^h$.
Let $N_G\geq 3$ be a fixed integer.

Throughout this work we will be concerned with the totally disconnected locally compact group $G = G_{N_G}$, defined as the $F$-split special orthogonal group $G = \mathrm{SO}_{N_G+1}(F)$, when $N_G$ is even, or ($F$-split) symplectic group $G = \mathrm{Sp}_{N_G-1}(F)$, when $N_G$ is odd.
We set 
\[
s_G = \left\{ \begin{array}{ll} -1 & N_G\mbox{ is even}  \\ 
 1 & N_G\mbox{ is odd} \end{array} \right.\;,
\]
and write $n_G$ for the rank of the simple group $G$, that is, either $N_G-s_G = 2n_G$ or $N_G-s_G = 2n_G+1$.
We assume that $p> 6n_G$ for applications of results from \cite{cmbo-arthur} (Theorem \ref{thm:cmbo})\footnote{ The second author can confirm that not much effort was put into this bound and it is expected that $p>2$ will suffice.}.

The group $G$ may be concretely realized as follows. 
Consider a $(N_G-s_G)$-dimensional $F$-vector space $V$, with a basis 
\[
e_1,\ldots,e_{n_G},v, f_{n_G},\ldots,f_1\in V\;,
\]
when $N_G$ is even, or  
\[
e_1,\ldots,e_{n_G}, f_{n_G},\ldots,f_1\in V\;,
\]
when $N_G$ is odd.
Let $B$ be a bilinear form on $V$ given as 
%
\[
B(e_i,e_j) = B(f_i,f_j) = 0, \; B(e_i,f_j) = \delta_{i,j},\; B(f_j,e_i) = -s_G\delta_{i,j} \quad 1\le i,j\le n_G\;,
\]
(here, $\delta_{i,j}$ is the Kronecker delta function), and
\[
B(f_j,v)= B(e_i,v)= B(v,e_i) = B(v,f_i) = 0, \quad B(v,v) = 1, \quad 1\le i,j\le n_G\;,
\]
if defined.
The group $G$ is then realized as the identity connected component of the isometry group of the form $B$.
%
%

\subsubsection{Weyl groups and compact subgroups}\label{sect:iwa-decomp}

The above-designated Witt basis for the form $B$ gives rise to an integral $\mathfrak{O}_F$-structure for $V$. The corresponding isometry group $G_{\mathfrak{O}_F}< G$ of the $\mathfrak{O}_F$-lattice constitutes a hyperspecial maximal compact subgroup of $G$.
It also gives rise to an (open compact) Iwahori subgroup $I_G< G_{\mathfrak{O}_F}$, that is defined as the pullback of the subgroup of upper-triangular matrices through the resulting projection $G_{\mathfrak{O}_F} \to G_{\mathfrak{O}_F/ \mathfrak{p}_F}$.

Let us consider the maximal $F$-torus $T< G$ of transformations that are diagonal with respect to the Witt basis, and its integral form $T_{\mathfrak{O}_F} = T\cap G_{\mathfrak{O}_F} = T\cap I_G$.
The (finite) Weyl group 
\[
W_G = N(G,T) / T 
\]
of $G$ arises when taking a quotient of the normalizer subgroup of $T$ in $G$. 
The Iwahori-Weyl group of $G$ is defined as the quotient
\[
\widetilde{W_G} = N(G,T) / T_{\mathfrak{O}_F} \;,
\]
clearly equipped with a projection $p: \widetilde{W_G}\to W_G$.
The Iwahori--Bruhat decomposition \cite[Theorem 2.16]{IM} describes the double cosets of $I_G$ in $G$ as
 \begin{equation}\label{eq:iwah-decomp}
 G=\bigsqcup_{w\in \widetilde{W_G}} I_G  w I_G\;.
 \end{equation}
Moreover, Iwahori-Matsumoto give a correspondence between  the set of compact subgroups $I_G < K< G$ and the set of finite subgroups $W_K < \widetilde{W_G}$. 

A particular focus of this work is on the set of \textit{maximal} open compact subgroups of $G$. It is known \cite[Section 14.7]{garrett}, that up to $G$-conjugation, a maximal compact open subgroup of $G$ must contain $I_G$. Thus, the classification of such groups reduces to the classification of finite subgroups of $\widetilde{W_G}$.
Indeed, as we recall in greater detail in Section \ref{sect:IM}, each maximal finite subgroup of $\widetilde{W_G}$ is conjugate to a member of the sequence of subgroups 
\[
W_{K_0}, W_{K_1},\ldots, W_{K_{n_G}}<\widetilde{W_G}\;,
\]
whose corresponding compact groups $I_G < K_i < G$ we now describe explicitly.

Embedding $G$ in a matrix form using the Witt basis for $B$, we write
\begin{equation}
    K_i = \begin{pmatrix}
         \mathfrak O_F & \mathfrak O_F & \mathfrak p_F^{-1} \\
         \mathfrak p_F & G^{i}_{\mathfrak O_F} & \mathfrak O_F \\
         \mathfrak p_F  & \mathfrak p_F & \mathfrak O_F
    \end{pmatrix} \cap G, \quad 0\le i \le n_G\;,
\end{equation}
where $G^{i}_{\mathfrak O_F}< G_{N_G-2i}$ stands for the integral form of the lower-rank group of same type as $G$. 
Note, that taking $i = 0$ recovers the hyperspecial maximal compact subgroup $K_0= G_{\mathfrak O_F}$ featuring at the outset of our analysis.
Let us also note that, when $N_G$ is odd, the symplectic group $G$ is simply connected. This fact causes the maximal compact subgroups of $G$ to coincide with maximal parahoric subgroups from the standard Bruhat-Tits theory. Indeed, there are precisely $n_G+1$ conjugacy classes of maximal parahoric subgroups.
%
%
%
%
Yet, in the odd orthogonal case of even $N_G$, when $G$ is no longer simply connected, the classification of its maximal compact subgroups slightly diverges from the parahoric case. For $0< i\leq n_G$,  each of the groups $K_i$ contains a maximal parahoric as a subgroup of index $2$.


\subsection{Unipotent locus of the dual group}

\subsubsection{Symplectic and orthogonal partitions}\label{sec:partitions}

The topological structure of $\mathcal{U}^\vee$ is described through combinatorics of integer partitions, which we explicate here.

We treat partitions as the set $\mathcal{P}$ of tuples $\lambda= (0 < \lambda_1\leq \lambda_2\leq\ldots \leq \lambda_{\ell(\lambda)})$ of integers. For convenience, we allow an empty partition $\lambda_{\emptyset}\in \mathcal{P}$ with $\ell(\lambda_{\emptyset})=0$.

For an integer $N\geq1$, its partitions $\mathcal{P}(N)\subset \mathcal{P}$ consist of $\lambda\in \mathcal{P}$ with $|\lambda|:= \sum_{i=1}^{\ell(\lambda)} \lambda_i = N$.
The multiplicity of a part $m(c,\lambda) = \#\{ i\,:\, \lambda_i = c\}$
is defined for each integer $c\in \mathbb{Z}_{>0}$ and a partition $\lambda\in \mathcal{P}$.
Given integers $a\leq b$, or possibly $b=\infty$, and a partition $\lambda\in \mathcal{P}$ as above, we set the \textit{interval} $\lambda_{a \leftrightarrow b}\in \mathcal{P}$ to be the partition consisting of all parts $\lambda'_i$ of $\lambda$ with $a\leq \lambda'_i \leq b$. In other words, 
\[
\lambda_{a \leftrightarrow b} = (\lambda_i\leq \lambda_{i+1}\leq\ldots \leq \lambda_j)\in \mathcal{P}\;,
\]
so that $i$ is the minimal index with $a\leq \lambda_i$ and $j$ is the maximal index with $\lambda_j \leq b$. 
For $\lambda\in \mathcal{P}$, we set 
\[
\supp(\lambda)= \{c\in\mathbb{Z}_{>0} \,:\, m(c,\lambda)>0\}
\]
to be its support. We also write $\supp(\lambda) = \{c_1 < \ldots < c_t\}$ and describe partitions in the common notation of
\[
\lambda = (c_1^{m(c_1,\lambda)}\,c_2^{m(c_2,\lambda)}\ldots c_t^{m(c_t,\lambda)})\in \mathcal{P}\;.
\]
For $\lambda^1,\lambda^2\in \mathcal{P}$, we write $\lambda^1\cup\lambda^2\in \mathcal{P}$ for the partition that is given by multiplicities $m(c,\lambda^1\cup \lambda^2) = m(c,\lambda^1) + m(c,\lambda^2)$, for all parts $c\in \mathbb{Z}_{>0}$. 
Similarly, we write $\lambda^1\setminus\lambda^2$ for the construction given by $m(c,\lambda^1\setminus \lambda^2) = m(c,\lambda^1)-m(c,\lambda^2)$, whenever those are all non-negative numbers.
Let us denote $\mathcal{P}^1_0$ (respectively, $\mathcal{P}^{-1}_0$) the subset of $\mathcal{P}$ of partitions whose support consists of odd (respectively, even) integers.
For each choice of a sign $s\in \{\pm1\}$, we define the sets of partitions
\[
\mathcal{P}^{s} = \{\lambda \in \mathcal{P}\;:\; \exists \mu\in \mathcal{P}^{s}_0,\; \exists \nu\in \mathcal{P},\; \lambda = \mu\cup \nu\cup \nu \} \;,
\]
and write $\mathcal{P}^{s}(N) = \mathcal{P}(N) \cap \mathcal{P}^{s}$, for every integer $N\geq1$.

There is a standard \textit{dominance} partial order on each set of partitions $\mathcal{P}^s(N)$. Namely, $\lambda\leq \mu$ holds, for $\lambda,\mu\in \mathcal{P}^s(N)$, when
\[
\sum_{i=0}^{k} \lambda_{\ell(\lambda)-i} \leq \sum_{i=0}^{k} \mu_{\ell(\mu)-i}
\]
hold, for all $0\leq k < \ell(\mu)$.

\begin{proposition}\label{prop:part-orbit}(e.g. \cite[Theorem 2.2]{kp82})
The conjugacy classes in $\mathcal{U}^\vee$ are in bijection with $\mathcal{P}^{s_G}(N_G)$.

A class $\mathcal{O}^\vee_{\lambda}\in \mathcal{U}^\vee$ that corresponds to a partition $\lambda$ consists of unipotent matrices whose multiset of lengths of Jordan blocks is given by the multiset of parts of $\lambda$.

Under this bijection, the topological closure order on conjugacy classes coincides with the dominance order on partitions. Namely, $\lambda\leq \mu$ holds, for $\lambda,\mu\in \mathcal{P}^{s_G}(N_G)$, if and only if, $\mathcal{O}_{\lambda}^\vee\subset \overline{\mathcal{O}_{\mu}^\vee}$.

\end{proposition}



Let us write 
\[
\mathcal{P}^{mf}= \{\lambda\in \mathcal{P}\;:\; m(c,\lambda)\leq1, \forall c\in \mathbb{Z}_{>0}\}
\]
for the set of multiplicity-free partitions.
For a partition $\lambda\in \mathcal{P}$, a unique decomposition $\lambda = \lambda^{mf} \cup \lambda^{m} \cup \lambda^{m}$ with $\lambda^{mf}\in \mathcal{P}^{mf}$ and $\lambda^{m}\in \mathcal{P}$ exists.
Clearly, for any partition $\lambda\in \mathcal{P}$,  the condition $\lambda\in \mathcal{P}^{\pm1}$ is equivalent to $\lambda^{mf}\in \mathcal{P}^{\pm1}_0$.

Let us now fix a sign $s\in \{\pm1\}$ and a partition $\lambda\in \mathcal{P}^s$.
There are unique partitions $\lambda^{gp}\in \mathcal{P}_{0}^{s}$ and $\lambda^{bp}\in \mathcal{P}_{0}^{-s}$, so that $\lambda = \lambda^{gp} \cup \lambda^{bp} \cup \lambda^{bp}$ holds.
In particular, $\lambda^{mf} = (\lambda^{gp})^{mf}$ holds.
Let us denote the set 
\[
S(\lambda) := \supp (\lambda^{gp})\subset \mathbb{Z}_{>0}\;,
\]
that consists of the \textit{good parity} parts in $\lambda$. We consider its power set 
\[
P(\lambda) = \{A\subset S(\lambda)\}\;,
\]
which we view as a $\mathbb{F}_2$-vector space, and in particular, a group of cardinality $2^{|S(\lambda)|}$. 
Alternatively, elements of $P(\lambda)$ may be taken as boolean functions on $S(\lambda)$.
As a vector space, $P(\lambda)$ is equipped with a natural pairing given by $\langle A, B\rangle = (-1)^{|A\cap B|}$, for $A,B\subset S(\lambda)$. The pairing provides a canonical identification of its dual group of complex characters $\widehat{P(\lambda)}$ with $P(\lambda)$ itself. 
Recalling that 
\[
S_0(\lambda):=\supp(\lambda^{mf})
\]
is a subset of $S(\lambda)$, that is, the set of good parity parts of $\lambda$ that appear with an odd multiplicity, we let 
\[
P(\lambda)^0 = P(\lambda) / \{\emptyset, S_0(\lambda)\}
\]
be the quotient group of $P(\lambda)$.
We also take note of the subgroups
\[
P(\lambda)_0 = \{A\in P(\lambda)\;:\; |A\cap S_0(\lambda)|\mbox{ is even}\}\;,\;
P(\lambda)' = \{A\in P(\lambda)\;:\; |A|\mbox{ is even}\}< P(\lambda)\;.
\]
Clearly, the previous identification $\widehat{P(\lambda)}\cong P(\lambda)$ factors through $\widehat{P(\lambda)^0}\cong P(\lambda)_0$.
Moreover, when $s=1$ and $|\lambda|$ is odd, the set $S_0(\lambda)$ must be of odd cardinality. Thus, the pairing $\langle\,,\rangle$ restricts to a perfect pairing between the subgroups $P(\lambda)'$ and $P(\lambda)_0$, giving a natural meaning to an identity $\widehat{P(\lambda)'}\cong P(\lambda)_0$ in this case.

\begin{proposition}\cite[Theorem 5.1.6,Corollary 6.1.6]{collingwoodmcgovern}
\label{prop:comp-comb}
\begin{enumerate}
    \item 
For a partition $\lambda\in \mathcal{P}^{1}(2n+1)$, the associated unipotent conjugacy class $\mathcal{O}^\vee_{\lambda}$ of the group $\mathbf{SO}_{2n+1}(\mathbb{C})$ has a natural identification of its component group $A(\mathcal{O}^\vee_{\lambda})$ with the $2$-group $P(\lambda)'$.
In particular, the character group $\widehat{A(\mathcal{O}^\vee_{\lambda})}$ is thus identified with $P(\lambda)_0$.

\item 
For a partition $\lambda\in \mathcal{P}^{-1}(2n)$, the associated unipotent conjugacy class $\mathcal{O}^\vee_{\lambda}$ of the group $\mathbf{Sp}_{2n}(\mathbb{C})$ has a natural identification of its component group $A(\mathcal{O}^\vee_{\lambda})$ with the $2$-group $P(\lambda)$.
In particular, the character group $\widehat{A(\mathcal{O}^\vee_{\lambda})}$ is thus identified with $P(\lambda)$ as well.

Let $z\in A(\mathcal{O}^\vee_{\lambda})$ be the representative of the central element $-1\in \mathbf{Sp}_{2n}(\mathbb{C})$ in the component group. 
Then, $z$ corresponds to the element $S_0(\lambda)\in P(\lambda)$ under the above identification, and the character subgroup
\[
\widehat{   A(\mathcal{O}^\vee_{\lambda}) / (z)  } < 
\widehat{   A(\mathcal{O}^\vee_{\lambda}) }
\]
corresponds to the subgroup $P(\lambda)_0< P(\lambda)$.

\end{enumerate}

\end{proposition}

We denote by $\widehat{A(\mathcal{O}^\vee_{\lambda})}_0<\widehat{A(\mathcal{O}^\vee_{\lambda})}$ the subgroup that corresponds to $P(\lambda)_0$ under the identifications of Proposition \ref{prop:comp-comb}.

\subsubsection{Lusztig's canonical quotient}\label{sect:quot}

Still holding to a fixed partition $\lambda\in \mathcal{P}^s$, we now write 
\[
S(\lambda) = \{\mu_{1}< \ldots < \mu_{k}\}
\]
and obtain 
\[
S_0(\lambda) = \{\mu_{\alpha_1},\ldots, \mu_{\alpha_r}\}\;,
\]
for indices $1\leq \alpha_1< \ldots < \alpha_r \leq k$. 
We now define an equivalence relation $\sim$ on $S(\lambda)$.

When $s=-1$, we impose the relations
\[
\mu_{\alpha_{r-2j-1}}\,\sim \,\mu_{\alpha_{r-2j-1}+1}\,\sim\,\mu_{\alpha_{r-2j-1}+2}\,\sim\,\cdots\,\sim\, \mu_{\alpha_{r-2j}}\;,
\]
for each $0\leq j < \lfloor r/2 \rfloor$, together with the additional relation
\[
\mu_{1}\,\sim\,\mu_{2}\,\sim\,\cdots\,\sim\, \mu_{\alpha_{1}}\;,
\]
in case $r$ is odd.

When $s=1$, we impose the relations
\[
\mu_{\alpha_{r-2j}}\,\sim \,\mu_{\alpha_{r-2j}+1}\,\sim\,\mu_{\alpha_{r-2j}+2}\,\sim\,\cdots\,\sim\, \mu_{\alpha_{r-2j+1}}\;,
\]
for each $0< j <  r/2 $, and the additional relation
\[
\mu_{\alpha_r}\,\sim\,\mu_{\alpha_r+1}\,\sim\,\mu_{\alpha_r+2}\,\sim\,\cdots\,\sim\, \mu_{k}\;.
\]

Let $S^{\dagger}(\lambda):= S(\lambda) / \sim$ denote the resulting set of equivalence classes, and $p: S(\lambda) \to S^{\dagger}(\lambda)$ the natural projection.
Let $P^{\dagger}(\lambda) = \{B\subset S^{\dagger}(\lambda)\}$ be the power set of $S^{\dagger}(\lambda)$, viewed again as a $2$-group.
Considering $P(\lambda)$ and $P^{\dagger}(\lambda)$ as spaces of boolean functions on $S(\lambda)$ and $S^{\dagger}(\lambda)$, we can pullback through the projection $p$ to obtain an embedding $p^\ast:P^{\dagger}(\lambda) \hookrightarrow P(\lambda)$ of groups.
More concretely, we identify
\[
P^{\dagger}(\lambda) \cong p^\ast(P^{\dagger}(\lambda)) =\{ A\subset S(\lambda)\;:\; \exists B \subseteq S^{\dagger} (\lambda),\; A= p^{-1}(B)\}
\]
as the subgroup of $P(\lambda)$ consisting of functions that are constant on $\sim$-classes.
We also set $P^{\dagger}(\lambda)_0 = P^{\dagger}(\lambda) \cap P(\lambda)_0 $.

\begin{proposition}\label{prop:quotient}
    Under the realizations of the character groups $\widehat{A(\mathcal{O}^\vee_{\lambda})}$ of Proposition \ref{prop:comp-comb} as $P(\lambda)$ (in case $\mathbf{G^\vee} = \mathbf{Sp}_{2n}$) or as $P(\lambda)_0$ (in case $\mathbf{G^\vee} = \mathbf{SO}_{2n+1}$), the subset 
    \[
    A^{\dagger}(\mathcal{O}^\vee_{\lambda}) < \widehat{A(\mathcal{O}^\vee_{\lambda})}
\]
that was defined in equation \eqref{eq:ach-sg}, is mapped to the subgroup $P^{\dagger}(\lambda)_0$. 



\end{proposition}

\begin{proof}
    For example, this is content of the combinatorial description in \cite[Section 3.4]{achar-quotient}, after identifying the canonical quotient with our definition in \eqref{eq:ach-sg} using \cite[Theorem 2.1]{ach-sage}.
\end{proof}

\begin{remark}

A more common approach in literature arrives at the character subgroup $A^{\dagger}(\mathcal{O}^\vee_{\lambda})$ in terms of dualizing a quotient space construction, rather than a subspace construction. 
Indeed, when dualizing the embedding map and obtaining a quotient $A(\mathcal{O}^\vee_{\lambda}) \twoheadrightarrow \widehat{A^{\dagger}(\mathcal{O}^\vee_{\lambda})}$, the resulting quotient spaces of $P(\lambda)$, or $P(\lambda)'$, under the identifications of Proposition \ref{prop:comp-comb}, match with the common description of Lusztig's canonical quotient.

\end{remark}

\subsubsection{Special pieces}\label{sect:special-piece}

For a partition $\lambda\in\mathcal{P}^s$ and an element $\theta \in S^{\dagger}(\lambda)$, we set the integers 
\[
\theta_{\min}= \min\{ \mu\in S(\lambda)\: :\: p(\mu)=\theta\}\;,\quad \theta_{\max}= \max\{ \mu\in S(\lambda)\: :\: p(\mu)=\theta\}\;.
\]
and
\begin{align*}
\tilde \theta_{\min}&= \begin{cases}
    0 & \mbox{ if } \theta_{\max} = \min S_0(\lambda)\;, s=-1 \\
    \theta_{\min} & \mbox{ otherwise}
\end{cases}, \\
\tilde \theta_{\max}&= \begin{cases}
    \infty & \mbox{ if } \theta_{\min} = \max S_0(\lambda)\;, s=1 \\
    \theta_{\max} & \mbox{ otherwise}
\end{cases}.
\end{align*}
In these terms we define the \textit{blocks} of the partition $\lambda$ to be the set of partitions
\[
\mathrm{Blk}(\lambda) = \{ \lambda(\theta)\}_{\theta\in S^{\dagger}(\lambda)}\;, 
\]
given as intervals of $\lambda$ of the form
\[
\lambda(\theta) = \lambda_{ \tilde \theta_{\min}\leftrightarrow \tilde \theta_{\max}}
\]
for each $\theta\in S^{\dagger}(\lambda)$. 

We also write 
\[
\lambda(\theta)^\circ = \lambda(\theta)\setminus (\tilde \theta_{\min}\;\tilde \theta_{\max})\in \mathcal{P}^s\;.
\]

It will be convenient to refer to the elements of $S^{\dagger}(\lambda)$, the indexing set of $\mathrm{Blk}(\lambda)$, also as blocks.
It follows that a decomposition
\[
\lambda = \lambda^{\#} \cup \lambda^{\#} \cup \bigcup_{\theta\in S^{\dagger}(\lambda)} \lambda(\theta)\;,
\]
holds, for a partition $\lambda^{\#}\in \mathcal{P}^{-s}$. 
We take note of the set of integers
\[
\mathbb{I}(\lambda): = \supp(\lambda^{bp}\setminus \lambda^{\#})\;.
\]
This is the set of bad parity parts that lie in some block.

\begin{definition}
    A partition $\lambda\in \mathcal{P}^{s}$, for $s\in \{\pm1\}$, is said to be \textit{special}, when its transposed partition satisfies $\lambda^t\in \mathcal P^s$.

    We denote by $\mathcal{P}^s_{\spci} \subset \mathcal{P}^s$ the set of special partitions, and set $\mathcal{P}^s_{\spci}(N) = \mathcal{P}(N)\cap \mathcal{P}^{s}_{\spci}$.
\end{definition}

\begin{lemma} \label{lem:special}
    For a partition $\lambda\in \mathcal{P}^{s}$, with $s\in \{\pm1\}$, the following conditions are equivalent:
\begin{enumerate}
\item $\lambda$ is special.

    \item Each of the blocks $\lambda(\theta)\in \mathrm{Blk}(\lambda)$ satisfies $\lambda(\theta) \in \mathcal{P}^{s}_0$ (i.e. $\lambda(\theta)^{gp}=\lambda(\theta)$).

\item Equality $\lambda^{gp} = \bigcup_{\theta\in S^{\dagger}(\lambda)} \lambda(\theta)$ holds in $\mathcal{P}^{s}$.

\item The set $\mathbb{I}(\lambda)$ is empty.

\end{enumerate}
\end{lemma}
\begin{proof}
    The equivalence of the bottom three conditions is evident. 
    Therefore, it suffices to show the equivalence of the first two conditions.

    By \cite[subsequent paragraph to Corollary 6.3.6]{collingwoodmcgovern}, a partition $\lambda\in \mathcal{P}^s$ is special, if and only if, 
    \[
    |\lambda_{a \leftrightarrow \infty}| = \frac{1+s}2\; (\mathrm{mod} \,2)
    \]
    holds, for all $a\in \supp(\lambda^{\mathrm{bp}})$.

Counting parities, we see that
\[
|\lambda_{a \leftrightarrow \infty}| \equiv_2 |(\lambda^{\mathrm{mf}})_{a \leftrightarrow \infty}| \equiv_2 |\{b\in S_0(\lambda)\,:\, a<b\}|\;.
\]
Indeed, our definition for the blocks of $\lambda$ is such that, for any $a\in \supp(\lambda^{\mathrm{bp}})$, the parity of $|\{b\in S_0(\lambda)\,:\, a<b\}|$ equals to $\frac{1-s}2$, if and only if, a block $\theta\in S^\dagger(\lambda)$ exists, for which $a\in \supp(\lambda(\theta))$. Thus, the condition of $\lambda$ being special is equivalent to $\lambda(\theta)^{\mathrm{bp}}= \lambda_{\emptyset}$, for all blocks $\theta\in S^\dagger(\lambda)$.

    
\end{proof}

We may now characterize the \textit{special} classes in the unipotent locus of $\mathbf{G^\vee}(\mathbb{C})$ as
\[
\mathcal{U}_{\spci}^\vee = \{\mathcal{O}^\vee_{\lambda}\in \mathcal{U}^\vee\;:\; \lambda\in \mathcal{P}^s_{\spci}\}\;.
\]
Given a partition $\lambda\in \mathcal{P}^{s}(N)$ and a subset $I\subset \supp(\lambda^{bp})$, we define an operation
\[
T^I(\lambda) = \lambda \; \cup_{c\in I} (c-1\; c+1) \setminus \cup_{c\in I}(c^2)\in  \mathcal{P}^{s}(N) \;,
\]
where $(02)=(2)\in \mathcal{P}$ is assumed if necessary.
Note we will ultimately only be interested in the operation when $I\subset \mathbb I(\lambda)$, but allowing $I\subset \supp(\lambda^{bp})$ will be convenient for proofs.
\begin{lemma}\label{lem:TI-operation}
Let $\lambda\in \mathcal P^s(N)$ be a partition, and $\lambda' = T^I(\lambda)$, for a subset $I\subset \mathbb I(\lambda)$.

Then,
\begin{enumerate}
    \item\label{it:TI-operation1} The identity
    \[
    \mathbb{I}(\lambda') = \mathbb{I}(\lambda) \setminus I
    \]
    holds.
\item
A unique map
\begin{equation}\label{eq:iotaast}
\iota^\ast = \iota^\ast_{\lambda,\lambda'}: S^{\dagger}(\lambda') \to S^{\dagger}(\lambda)
\end{equation}
exists, so that $\widetilde{\iota^\ast(\theta)}_{\min} \le \widetilde{\theta}_{\min} \le \widetilde{\theta}_{\max} \le \widetilde{\iota^\ast(\theta)}_{\max}$ holds, for each block $\theta\in S^\dagger(\lambda')$.

\item The map $\iota^\ast$ is surjective, and
\[
(\lambda(\theta)^\circ)^{\mathrm{gp}} = \bigcup_{\theta'\in (\iota^\ast)^{-1}(\theta)} (\lambda'(\theta')^\circ )^{\mathrm{gp}}
\]
holds, for any block $\theta\in S^{\dagger}(\lambda)$.

\end{enumerate}
    In particular, $T^{\mathbb{I}(\lambda)}(\lambda)$ is a special partition.
\end{lemma}
\begin{proof}
    Writing $I = I'\sqcup\{c\}$ and noting that $\lambda' = T^{I'}(T^{\{c\}}(\lambda))$, we see that induction on $|I|$ allows us, by defining $\iota^\ast:= \iota^\ast_{\lambda, T^{\{c\}}(\lambda)} \circ\iota^\ast_{T^{\{c\}}(\lambda),\lambda'}$, to reduce the statement to the case when $I$ is a singleton. 
    
    We now assume $I=\{c\}$, for a given part $c\in \supp(\lambda(\theta_0)^{\mathrm{bp}})$ and a block $\theta_0\in S^{\dagger}(\lambda)$.

    Let us note that $(c-1 \, c+1)\in \mathcal{P}^s_0$ and that $m(\lambda', c\pm1) = m(\lambda, c\pm1)+1$. Hence, an equality
    \[
    S_0(\lambda') = \left\{ \begin{array}{ll} S_0(\lambda) + \{c-1,c+1\} & c>1 \\  S_0(\lambda) + \{2\}& c=1\end{array}\right.
    \]
    may be written in the boolean group $P\left(\lambda'\right)$.

    Subsequently, we see that the altered block structure is given by
    \begin{equation}\label{eq:blks}
    \mathrm{Blk}(\lambda') = \left\{ \begin{array}{ll} \mathrm{Blk}(\lambda)\setminus \lambda(\theta_0)\; \cup \left\{\lambda_{\widetilde{(\theta_0)_{\min}}\leftrightarrow c-1}\cup(c-1),(c+1) \cup\lambda_{c+1\leftrightarrow \widetilde{(\theta_0)_{\max}}}\right\} & c>1 \\  \mathrm{Blk}(\lambda)\setminus \lambda(\theta_0)\; \cup \left\{(2) \cup\lambda_{2\leftrightarrow (\theta_0)_{\max}}\right\}& c=1\end{array}\right.\;.
    \end{equation}

    The last claim is formally verifiable through simple case analysis, depending on the configuration of $\{c-1,c+1\}$ within $S_0(\lambda)$. For example, when $c-1,c+1\not\in S_0(\lambda)$ and $c>1$, we must have
    \[
    S_0(\lambda) = \{\widetilde{\theta^1}_{\min}<\widetilde{\theta^1}_{\max} <\ldots< \widetilde{\theta^t}_{\min} < \widetilde{\theta^t}_{\max}\}\setminus\{0,\infty\}\;,
    \]
    for $\theta^1,\ldots,\theta^t\in S^\dagger(\lambda)$, with $\theta_0= \theta^{i_0}$. Then,
    \[
    S_0(\lambda') = \{\widetilde{\theta^1}_{\min}<\widetilde{\theta^1}_{\max} <\ldots< \widetilde{\theta^{i_0}}_{\min} < c-1< c+1<\widetilde{\theta^{i_0}}_{\max}<\ldots< \widetilde{\theta^t}_{\min} < \widetilde{\theta^t}_{\max}\}\setminus\{0,\infty\}\;,
    \]
and \eqref{eq:blks} follows in that case.

    In other words, in all cases, the block $\lambda(\theta_0)$ is either being split into two blocks in $\mathrm{Blk}(\lambda')$, or shortened (when $c=1$), while all other blocks in $\mathrm{Blk}(\lambda)$ remain unchanged.

    In particular, we see that 
    \[
    \mathbb I(T^{I}(\lambda))= \supp\left(\bigcup_{\theta\in S^\dagger(T^I(\lambda))} \lambda(\theta)^{\mathrm{bp}} \right) = \supp\left(\bigcup_{\theta\in S^\dagger(\lambda)} \lambda(\theta)^{\mathrm{bp}} \right)\setminus \{c\}=\mathbb I(\lambda)\setminus\{c\}\;,
    \] 
    and this concludes the proof of the first identity.

    The map $\iota^\ast$ can now be naturally defined: When $c=1$, this is the visible bijection between $\mathrm{Blk}(\lambda')$ and $\mathrm{Blk}(\lambda)$. Otherwise, $\iota^\ast(\theta_1) = \iota^\ast(\theta_2) = \theta_0$ is set, where $(\theta_1)_{\max} = c-1$
 and $(\theta_2)_{\min} = c+1$, and the natural bijection elsewhere.

    Finally, the partition $T^{\mathbb{I}(\lambda)}(\lambda)$ is special since $\mathbb I(T^{\mathbb I(\lambda)}(\lambda)) = \emptyset$ (Lemma \ref{lem:special}).

\end{proof}

\begin{example}
For $\lambda = (1^2 2^4 3^3 4^4 5^2 8^2 9^2)\in \mathcal{P}^1(79)$ we have $\mathrm{Blk}(\lambda)= \{(1^2),(3^3 4^4 5^2 8^2 9^2)\}$ and $\mathbb{I}(\lambda)= \{4,8\}$. Thus, $\lambda' = T^{\{4\}}(\lambda) = (1^2 2^4 3^4 4^2 5^3 8^2 9^2)$ has $\mathrm{Blk}(\lambda') = \{(1^2), (3^4), (5^3 8^2 9^2)\}$ and $\mathbb{I}(\lambda') = \{8\}$.

The partition $\lambda''= T^{\{4,8\}}(\lambda) = T^{\{8\}}(\lambda') = (1^2 2^4 3^4 4^2 5^3 7^1 9^3)$, with $\mathrm{Blk}(\lambda'') = \{(1^2), (3^4), (5^3 7^1), (9^3)\} $ is special.
\end{example}

Let us also describe the family of inverted operations to those of the form $T^I$ on $\mathcal{P}^{s}(N)$.
For a partition $\lambda\in \mathcal{P}^{s}$, we say that a block $\theta\in S^{\dagger}(\lambda)$ is \textit{admissible}, if either $\tilde{\theta}_{\min} = \theta_{\min}=2$ holds, or a block $\theta'\in S^{\dagger}(\lambda)$ exists, for which $\theta'_{\max} =\theta_{\min}-2$.

Let us denote by $S^{\dagger\dagger}(\lambda)\subset S^{\dagger}(\lambda)$ the set of admissible blocks.
We write the set of integers
\[
\mathbb{J}(\lambda) = \{\theta_{\min}-1\::\: \theta\in S^{\dagger\dagger}(\lambda)\}\;.
\]
For a partition $\lambda\in \mathcal{P}^{s}(N)$ and a subset $J\subset \mathbb{J}(\lambda)$, we define
\[
T_J(\lambda) = \lambda \;\cup_{c\in J}(c^2)  \setminus \cup_{c\in J} (c-1\; c+1) \in  \mathcal{P}^{s}(N) \;
\]
using the convention that removing $0$ does nothing.
\begin{lemma}
    For $J\subset \mathbb J(\lambda)$, $T_J(\lambda)$ is well-defined.
\end{lemma}
\begin{proof}
    We need to check that whenever $d-1,d+1\in J$ occurs, $m(d,\lambda)\geq 2$ must also hold.

    Indeed, this situation may only occur when two admissible blocks $\theta,\theta'\in S^{\dagger\dagger}(\lambda)$ have $\theta_{\min}=d$ and $\theta'_{\min}=d+2$. Hence, $\theta_{\max}= d$ and $S(\lambda(\theta))=\{d\}$. Since $\theta$ is admissible, $\tilde\theta_{\min}= \theta_{\min}$. By construction of blocks, the last equality for a singleton block must imply $d\not\in S_0(\lambda)$. Thus, $m(d,\lambda)$ must be even, and the result follows since $d\in S(\lambda)$.


\end{proof}

\begin{lemma}\label{lem:TI-J}
    For a partition $\lambda \in \mathcal P^s(N)$ and a subset $I\subset \mathbb I(\lambda)$, we have
    \begin{enumerate}
        \item\label{it-lem:TI-J1} $\mathbb J(T^I(\lambda)) = \mathbb J(\lambda)\cup I$,
        \item\label{it-lem:TI-J2} $T_I(T^I(\lambda)) = \lambda$.
    \end{enumerate}
\end{lemma}
\begin{proof}
    As in the proof of Lemma \ref{lem:TI-operation}, we may assume $I = \{c\}$.
    Suppose that $c\in \supp(\lambda(\theta))$.
    By the proof of Lemma \ref{lem:TI-operation}, the blocks of $T^I(\lambda)$ are the same as for $\lambda$, except that $\lambda(\theta)$ is either split into the two blocks $\lambda_{\tilde \theta_{\min}\leftrightarrow c-1}\cup(c-1),(c+1) \cup\lambda_{c+1\leftrightarrow \tilde \theta_{\max}}$, when $c>1$, or shortened to $(2)\cup \lambda_{2\leftrightarrow  \theta_{\max}}$, when $c=1$.
    We therefore see that $\mathbb J(T^I(\lambda)) = \mathbb J(\lambda) \cup \{c\}$ as required.

    The operations in \ref{it-lem:TI-J2} are evidently mutually inverse, when well-defined. Indeed, it is well-defined by \ref{it-lem:TI-J1}.
\end{proof}
\begin{lemma}\label{lem:TJ-operation}
  For a partition $\lambda \in \mathcal P^s(N)$ and a subset $J\subset \mathbb J(\lambda)$, we have
    \begin{enumerate}
        \item\label{it:TJ-operation1} $\mathbb{J}(T_J(\lambda)) =\mathbb{J}(\lambda) \setminus J$,
        \item\label{it:TJ-operation2} $\mathbb{I}(T_J(\lambda)) =\mathbb{I}(\lambda) \cup J$,
        \item\label{it:TJ-operation3} $T^J(T_J(\lambda)) = \lambda$. 
    \end{enumerate}
\end{lemma}
\begin{proof}
    It suffices to establish the inclusion $J\subset \mathbb{I}(T_J(\lambda))$. Indeed, when setting $\mu = T_J(\lambda)$, $\lambda = T^J(\mu)$ is evident, $\mathbb{I}(\lambda) = \mathbb{I}(\mu)\setminus J$ by Lemma \ref{lem:TI-operation}\eqref{it:TI-operation1}, and \eqref{it:TJ-operation2} follows. Similarly, $\mathbb{J}(\lambda) = \mathbb{J}(\mu) \cup J$ by Lemma \ref{lem:TI-J}\eqref{it-lem:TI-J1} and \eqref{it:TJ-operation1} follows, since $\mathbb{I}(\mu), \mathbb{J}(\mu)$ are disjoint.
    
    As in previous proofs, arguing by induction on $|J|$, we may assume that $J = \{c\}$ is a singleton.
    Then, $c = \theta_{\min}-1$ for an admissible block $\theta\in S^{\dagger\dagger}(\lambda)$. We have an equality
\[
    S_0(T_J(\lambda)) = \left\{ \begin{array}{ll} S_0(\lambda) + \{c-1,c+1\} & c>1 \\  S_0(\lambda) + \{2\}& c=1\end{array}\right.
    \]
    in the boolean group $P\left(\lambda\right)$.

    Thus, when $c>1$, the blocks of $\mathrm{Blk}(\lambda)$ and $\mathrm{Blk}(T_J(\lambda))$ are identical, except for the pair $\lambda(\theta'),\lambda(\theta)\in \mathrm{Blk}(\lambda)$, with $\theta'_{\max}=c-1$, being merged into a single block 
    \[
    \lambda_{\widetilde\theta'_{\min}\leftrightarrow\widetilde\theta_{\max}}\cup(c^2)\setminus (c-1,c+1)\in \mathrm{Blk}(T_J(\lambda))\;.
    \]
    Similarly, when $c=1$, all blocks remain unchanged, except for $\lambda(\theta)= \lambda_{2\leftrightarrow \theta_{\max}}$ being replaced with $\lambda_{0\leftrightarrow \theta_{\max}}\cup (1^2)\setminus(2)$.

    In either case, we see that $c\in \mathbb{I}(T_J(\lambda))$.

    

\end{proof}

\begin{lemma}\label{lem:order}
    Let $\lambda\in \mathcal P^s_{spc}$ and $J_1,J_2\subset \mathbb J(\lambda)$.
    Then $J_1\subset J_2$ if and only if $T_{J_2}(\lambda)\le T_{J_1}(\lambda)$.
\end{lemma}
\begin{proof}
    $(\Rightarrow)$ Suppose $J_1\subset J_2$. Let $\Delta = J_2\setminus J_1$. Let $\mu = T_{J_1}(\lambda)$.
    Then $T_{J_2}(\lambda) = T_{\Delta}(\mu)$ and it suffices to show that $T_{\Delta}(\mu)\le \mu$.
    By induction we may reduce to the case when $\Delta = \{c\}$ is a singleton and the result then follows immediately from the fact that $(c^2)\le (c-1\,c+1)$, for all $c$.

    $(\Leftarrow)$ Suppose towards a contradiction that $c\in J_1\setminus J_2$. Let $k$ be the smallest integer with $\lambda_k = c+1$. Then, equalities of the form $(a-1)+(a+1) = a+a$ imply
    \[
    \sum_{i\ge k}T_{J_2}(\lambda)_i = \sum_{i\ge k}\lambda_i = \left(\sum_{i\ge k}T_{J_1}(\lambda)_i\right) + 1 \;,
    \]
    which contradicts $T_{J_2}(\lambda)\le T_{J_1}(\lambda)$.
    
\end{proof}

\begin{definition}
    For any partition $\lambda\in \mathcal{P}^{s}(N)$, we let the \textit{relative special piece} of $\lambda$ be the set of partitions
    \[
\spc(\lambda) = \{T_J(\lambda)\,:\, J\subset \mathbb{J}(\lambda)\}\;\subset\; \mathcal{P}^s(N)\;.
    \]    
\end{definition}


\begin{proposition}\label{prop:spc-comb}
For any conjugacy class $\mathcal{O}^\vee_{\lambda}\in \mathcal{U}^\vee$ as in Proposition \ref{prop:part-orbit}, given by a partition $\lambda\in \mathcal{P}^{s_G}(N_G)$, we have a combinatorial description
\[
\spc(\mathcal{O}^\vee_{\lambda})  = \{\mathcal{O}^\vee_{\mu}\;:\;\mu\in \spc(\lambda)\}\subset \mathcal{U}^\vee
\]
for the relative special piece of the unipotent locus of $\mathbf{G^\vee}(\mathbb{C})$ that is defined by $\mathcal{O}_{\lambda}^\vee$.
In particular, the set $\spc(\mathcal{O}^\vee_{\lambda})$ consists of $2^{|\mathbb{J}(\lambda)|}$ conjugacy classes.

\end{proposition}

\begin{proof}


Let us first assume that $\lambda\in \mathcal{P}^{s_G}_{\spci}(N_G)$ is a special partition.
In \cite[Proposition 4.2]{kp-aster} (attributed to Spaltenstein), it was shown that for all $\mu\in \mathcal{P}^{s_G}(N_G)$, the class $\mathcal{O}^\vee=\mathcal{O}^\vee_{T^{\mathbb{I}(\mu)}(\mu)}\in \mathcal{U}^\vee_{\spci}$ is the unique special unipotent class which satisfies $\mathcal{O}_{\mu}^\vee\in \overline{\mathcal{O}^\vee}$ and is minimal with respect to the topological partial order on $\mathcal{U}^\vee$.
It then follows that
\[
\spc(\mathcal{O}^\vee_{\lambda}) = \{\mathcal{O}^\vee_{\mu}\;:\;\mu\in \mathcal{P}^{s_G}(N_G)\,, \; T^{\mathbb{I}(\mu)}(\mu) = \lambda\}\;.
\]
Now, if $T^{\mathbb I(\mu)}(\mu) = \lambda$, then by Lemma \ref{lem:TI-J} we have that $\mathbb I(\mu) \subset \mathbb J(\lambda)$.
So by Lemma \ref{lem:TI-J}(2) every $\mathcal O^\vee_\mu\in \spc(\mathcal O^\vee_\lambda)$ is of the form $T_{J}(\lambda)$ for $J\subset \mathbb J(\lambda)$.
Conversely, by Lemma \ref{lem:TJ-operation}(3) every partition $\mu$ of the form $T_J(\lambda)$ for some $J\subset \mathbb J(\lambda)$ satisfies $T^{\mathbb I(\mu)}(\mu) = \lambda$.
Therefore
$$
\spc(\mathcal O^\vee_\lambda) = \{\mathcal O^\vee_\mu: \mu = T_J(\lambda),\ \exists J\subset \mathbb J(\lambda)\}
$$
as required.

Let us now take a general partition $\lambda\in \mathcal{P}^{s_G}(N_G)$. 
We mark the special partition $\lambda'= T^{\mathbb{I}(\lambda)}(\lambda)$. By definition of relative special pieces we have $\spc(\mathcal{O}^\vee_{\lambda}) = \spc(\mathcal{O}^\vee_{\lambda'}) \cap \overline{\mathcal{O}^\vee_{\lambda}}$.
By Proposition \ref{prop:part-orbit} and Lemma \ref{lem:order}, $\mathcal{O}^\vee_{T_{J_1}(\lambda')} \in \overline{\mathcal{O}^\vee_{T_{J_2}(\lambda')}}$ holds, if and only if, $J_2\subset J_1$. 
In particular, since $\lambda = T_{\mathbb{I}(\lambda)}(\lambda')$, we obtain a description 
\[
\spc(\mathcal{O}^\vee_{\lambda}) = \{ \mathcal{O}^\vee_{T_J(\lambda')} \;:\; \mathbb{I}(\lambda)\subset J\subset \mathbb{J}(\lambda')  \}\;.
\]
Now note from Lemma \ref{lem:TJ-operation} that $\mathbb J(\lambda) = \mathbb J(\lambda')\setminus\mathbb I(\lambda)$ so $\mathbb J(\lambda') = \mathbb J(\lambda) \cup \mathbb I(\lambda)$.
Moreover $\mathbb I(\lambda) \cap \mathbb J(\lambda) = \emptyset$ since none of the elements of $\mathbb J(\lambda)$ lie inside a block.
Therefore $\mathbb{I}(\lambda)\subset J\subset \mathbb{J}(\lambda')$ if and only if $J = \mathbb I(\lambda) \cup J'$ for some $J'\subset \mathbb J(\lambda)$ and the proposition now follows from the equality $T_{J'}(\lambda) = T_{\mathbb{I}(\lambda) \cup J'}(\lambda')$.

\end{proof}

\begin{remark}
It follows from Proposition \ref{prop:spc-comb} and its proof that the topological partial order on the conjugacy classes of $\spc(\mathcal{O}^\vee_{\lambda})\subset \mathcal{U}^\vee$ is naturally isomorphic to the hypercube lattice of the subsets of $\mathbb{J}(\lambda)$.

\end{remark}

We record several basic properties of the construction.

\begin{lemma}\label{lem:even}
Suppose that $c-1, c+1\in \mathbb{J}(\lambda)$, for an integer $c$ and a partition $\lambda\in \mathcal{P}^{s_G}(N_G)$.

Then, $c\not\in S_0(\lambda)$, and, in particular, the multiplicity $m(c,\lambda)$ is even.
\end{lemma}

\begin{proof}
Clearly, the assumption implies $c\in S(\lambda)$ and that $(p(c))_{\min} = (p(c))_{\max} = c$. By construction of $S^{\dagger}(\lambda)$ that cannot happen when $c\in S_0(\lambda)$.
\end{proof}

\begin{lemma}\label{lem:tail}
Suppose that $2\in S(\lambda)$, for a partition $\lambda\in \mathcal{P}^{s_G}(N_G)$. 

Then, either $1\in \mathbb{J}(\lambda)$, or $2\not\in A$, for all $A\in P^{\dagger}(\lambda)_0$.
\end{lemma}

\begin{proof}
Note, that we are in the $s_G=-1$ situation. We consider the projection $S(\lambda)\to S^{\dagger}(\lambda)$ and write $\theta_0 = p(2)$. 

Suppose that $1\not\in \mathbb{J}(\lambda)$.
By construction of $S^{\dagger}(\lambda)$, we must have $|p^{-1}(\theta_0)\cap S_0(\lambda)|=  |\{(\theta_0)_{\max} \}|  = 1$, while $|p^{-1}(\theta)\cap S_0(\lambda)|=2$, for all $\theta_0\neq \theta\in S^{\dagger}(\lambda)$.
Hence, for any $A\in P^{\dagger}(\lambda)_0$, we must have $A\cap p^{-1}(\theta_0)=\emptyset$.

\end{proof}

\begin{lemma}\label{lem:head}
Suppose that $s_G=1$ and a partition $\lambda\in \mathcal{P}^{s_G}(N_G)$ is given.

Let $a = \max S(\lambda)$. Then, $a\not\in A$, for all $A\in P^{\dagger}(\lambda)_0$.
\end{lemma}

\begin{proof}
A similar argument to that in the previous proof of Lemma \ref{lem:tail} holds. Namely, from the construction of $S^{\dagger}(\lambda)$, we have $|p^{-1}(p(a))\cap S_0(\lambda)|=  |\{p(a)_{\min} \}|  = 1$, while $|p^{-1}(\theta)\cap S_0(\lambda)|=2$, for all $p(a)\neq \theta\in S^{\dagger}(\lambda)$.
Hence, for any $A\in P^{\dagger}(\lambda)_0$, the parity condition of $P(\lambda)_0$ forces $a\not\in A $.


\end{proof}

\subsubsection{Primitivity}

Suppose that partitions $\lambda\in \mathcal{P}^{s_G}(N_G)$ and $\mu\in \spc(\lambda)$ are given. Then, a subset $I\subset \mathbb{I}(\mu)$ exists for which $\lambda = T^I(\mu)$. 
Dualizing the map $\iota^\ast_{\mu,\lambda}$ from \eqref{eq:iotaast}, we obtain an injective map
\[
\iota_{\mu,\lambda}: P^{\dagger}(\mu)\hookrightarrow  P^{\dagger}(\lambda)\;,
\]
when the groups involved are viewed as boolean function spaces on $S^{\dagger}(\mu)$ and $S^{\dagger}(\lambda)$.

\begin{lemma}
     An inclusion $\iota_{\mu,\lambda}(P^{\dagger}(\mu)_0)\subset P^{\dagger}(\lambda)_0$ holds.
\end{lemma}

\begin{proof}
Let $A\in P^\dagger(\mu)_0$ be a given subset of $S(\mu)$. We need to show that the set $\iota_{\mu,\lambda}(A)\cap S_0(\lambda)$ is of even cardinality.

Recall, that by construction of $\lambda = T^I(\mu)$, we have an identity
\[
S_0(\lambda) =  S_0(\mu)+\{c+1\}_{c\in I\cap\{1\}}+\sum_{c\in I\setminus\{1\}} \{c-1,c+1\}\in P(\lambda) \;,
\]
in terms of boolean function calculus. When $1\in I$, we know that $1\not\in \mathbb{J}(\mu)$ and by Lemma \ref{lem:tail}, we have $2\not\in A$, implying $2\not\in \iota_{\mu,\lambda}(A)$.

For each $c\in I\setminus\{1\}$, we have $c-1 = (\theta_1)_{\max}$ and $c+1 = (\theta_2)_{\min}$, for blocks $\theta_1,\theta_2\in \mathrm{Blk}(\lambda)$ with $\iota_{\mu,\lambda}^\ast(\theta_1)=\iota_{\mu,\lambda}^\ast(\theta_2)\in \mathrm{Blk}(\mu)$. Thus, $c-1\in \iota_{\mu,\lambda}(A)$ holds, if and only if, $c+1\in \iota_{\mu,\lambda}(A)$.

Altogether, we see that the parity of $|\iota_{\mu,\lambda}(A)\cap S_0(\lambda)|$ equals to that of $|\iota_{\mu,\lambda}(A)\cap S_0(\mu)|=|A\cap S_0(\mu)|$, which was assumed to be even.

\end{proof}

By invoking the identification of Proposition \ref{prop:quotient} on $\iota_{\mu,\lambda}$, we arrive at a definition of an embedding of character groups
\[
\iota_{\mathcal{O}^\vee_{\mu},\mathcal{O}^\vee_{\lambda}} : A^{\dagger} (\mathcal{O}^\vee_{\mu}) \hookrightarrow A^{\dagger}(\mathcal{O}^\vee_{\lambda})\;.
\]
Now, for a partition $\lambda\in \mathcal{P}^{s_G}(N_G)$ and an integer $c\in \mathbb{J}(\lambda)$, we define a character 
\begin{equation}\label{eq:tc}
\mathfrak{t}_c \in \widehat{P(\lambda)}\,, \quad 
\mathfrak{t}_c(A) = \left\{\begin{array}{ll} (-1)^{|A\cap \{ c-1,c+1\} |} &  c>1 \\ (-1)^{|A\cap \{ 2\}| } & c=1 \end{array}\right.\;,
\end{equation}
where $A\in P(\lambda)$ is viewed as a subset of $S(\lambda)$.
Let us also recall the notion of primitivity that was outlined in Definition \ref{defi:prim}.

\begin{proposition}\label{prop:prim}
Let $\lambda\in \mathcal{P}^{s_G}(N_G)$ be a partition, and $\epsilon\in A^{\dagger}(\mathcal{O}^\vee)$ a character of the corresponding component group. 
Let $B_\epsilon\subset S(\lambda)$ be the subset corresponding to $\epsilon$ under the identification of Proposition \ref{prop:comp-comb}.

For each $\mu\in \spc(\lambda)$, $\epsilon$ is $\mathcal{O}_{\mu}^\vee$-primitive, if and only if, $\mathfrak{t}_c(B_\epsilon)\neq1$, for all $c\in \mathbb{J}(\lambda)\setminus \mathbb{J}(\mu)$.
    
\end{proposition}

\begin{proof}
For $c\in \mathbb{J}(\lambda)$, we take note of the partition $\lambda_c = T_{\{c\}}(\lambda)$.
Let us see that $\mathfrak{t}_c(B_\epsilon)=1$ holds, if and only if, $\epsilon\in \mathrm{Im}(\iota_{\mathcal{O}^\vee_{\lambda_c}, \mathcal{O}^\vee_{\lambda}})$.

When $c=1$, one direction of the latter claim follows from Lemma \ref{lem:tail}. Conversely, since $\iota_{\lambda_1,\lambda}$ is a bijection in this case, we only need to observe that any subset $B\subset S^{\dagger}(\lambda_1)\setminus \{p(2)\}$ must satisfy the parity condition $B\in P(\lambda_1)_0$.

When $c>1$, the fibers of the surjection $\iota^\ast_{\lambda_c,\lambda}: S^{\dagger}(\lambda) \to S^{\dagger}(\lambda_c)$ are all singletons, except for the occurrence of $(\iota^\ast_{\lambda_c,\lambda})^{-1}(\theta) =\{\theta', \theta''\}$, where $\theta'_{\max} = c-1$ and $\theta''_{\min} = c+1$.
Thus, as boolean functions on $S(\lambda)$, we have the equality
\[
\mathrm{Im}(\iota_{\mathcal{O}^\vee_{\lambda_c}, \mathcal{O}^\vee_{\lambda}}) = \{\epsilon'\in P^{\dagger}(\lambda)_0\: :\: \epsilon'(c-1) = \epsilon'(c+1)\}\;,
\]
which is equivalent to our claim.

Now, suppose that $\epsilon$ is not $\mathcal{O}_{\mu}^\vee$-primitive. Then, a class $\mathcal{O}^\vee\in \mathcal{U}^\vee$ exists with $\mathcal{O}_{\mu}^\vee \leq \mathcal{O}^\vee\lneq \mathcal{O}_{\lambda}^\vee$ in the topological partial order on $\mathcal{U}^\vee$, so that $\epsilon\in \mathrm{Im}(\iota_{\mathcal{O}^\vee, \mathcal{O}^\vee_{\lambda}})$.
A partition $\mu'\in \spc(\lambda)$ can then be found, so that $\mathcal{O}^\vee= \mathcal{O}^\vee_{\mu'}$. In particular, $\mu' = T_J(\lambda)$, for $\emptyset\neq J\subset \mathbb{J}(\lambda)\setminus \mathbb{J}(\mu)$. Picking $c\in J$, we see that $\lambda_c = T^{J\setminus\{c\}}(\mu')$, $\mathrm{Im}(\iota_{\mathcal{O}_{\mu'}^\vee, \mathcal{O}^\vee_{\lambda}})\subset \mathrm{Im}(\iota_{\mathcal{O}_{\lambda_c}^\vee, \mathcal{O}^\vee_{\lambda}})$, and consequently, $\mathfrak{t}_c(B_\epsilon)=1$.

Similarly, when it is assumed that $\epsilon$ is $\mathcal{O}_{\mu}^\vee$-primitive, we obtain that $\epsilon\not\in \mathrm{Im}(\iota_{\mathcal{O}_{\lambda_c}^\vee, \mathcal{O}^\vee_{\lambda}})$, for any $c\in \mathbb{J}(\lambda)\setminus \mathbb{J}(\mu)$, since $\mathcal{O}_{\mu}^\vee \leq \mathcal{O}_{\lambda_c}^\vee\lneq \mathcal{O}_{\lambda}^\vee$.



\end{proof}

\subsubsection{Barbasch--Vogan--Lusztig--Spaltenstein duality}

The Lie-theoretic duality between types $B$ and $C$ is known to be manifested in the context of unipotent conjugacy classes. For each integer $n\geq1$, explicit maps
\[
\mathcal{P}^{-1}(2n) \to \mathcal{P}^{1}(2n+1),\quad \mathcal{P}^{1}(2n+1) \to \mathcal{P}^{-1}(2n)\;
\]
are defined in \cite[Chapter 3]{spal-book} or \cite[Appendix A]{barbaschvogan}, all of which we will simply denote as $d$. 
Here we list some meaningful properties of these duality maps.

\begin{proposition}\label{prop:d-prop}
Let $d: \mathcal{P}^{s_G}(N_G) \to \mathcal{P}^{-s_G}(N_G-s_G)$ be the Barbasch--Vogan--Lusztig--Spaltenstein map.

Then, 
\begin{enumerate}
    \item The image of $d$ is the set of special partitions $\mathcal{P}^{-s_G}_{\spci}(N_G-s_G)$.
\item The decomposition of $\mathcal{P}^{s_G}(N_G)$ into the fibers of the map $d$ amounts precisely to the decomposition
\[
\mathcal{P}^{s_G}(N_G) = \bigsqcup_{\lambda\in \mathcal{P}^{s_G}_{\spci}(N_G)} \spc(\lambda)
\]
of the set of partitions into its special pieces.
In particular, for any $\lambda \in \mathcal{P}^{s_G}(N_G)$, $d$ remains constant on the relative special piece $\spc(\lambda)$.

\item Setting $d': \mathcal{P}^{-s_G}(N_G-s_G) \to \mathcal{P}^{s_G}(N_G)$ to be the duality map in the reverse direction, we have $d'(d(\lambda)) = T^{\mathbb{I}(\lambda)}(\lambda)$, 
for any partition $\lambda\in \mathcal{P}^{s_G}(N_G)$.
    
\end{enumerate}

\end{proposition}

It is evidently read from Proposition \ref{prop:d-prop} that the duality maps restrict to explicit bijections $\mathcal{P}^{-1}_{\spci}(2n) \cong \mathcal{P}^{1}_{\spci}(2n+1)$ between special partitions, for each $n\geq1$. In more accurate terms, the duality sets up a bijection between the sets of special pieces that compose $\mathcal{P}^{-1}(2n)$ and $\mathcal{P}^{1}(2n+1)$, respectively.

\subsection{Representation theory}

We study smooth $G$-representations over the complex field. We write a pair $(\pi, V)$, or more often simply $\pi$, to refer to a complex vector space $V$ on which $G$ acts through a homomorphism $\pi : G \to GL(V)$. 
Recall that the standard notion of smoothness in the $p$-adic setting is that the stabilizer in $G$ of every vector in $V$ contains an open compact subgroup.
Let $\irr(G)$ denote the collection of isomorphism classes of irreducible such representations.

We take note of the involution $\pi\mapsto \pi^t$ on $\irr(G)$ that is known as the \textit{Aubert duality}. The reader may be invited to the introduction section of \cite{am-aub} for a review of the various definitions and manifestations of that duality.


\subsubsection{Local Langlands Reciprocity (split groups)}

The collection $\irr(G)$, for our groups of interest, was successfully described in arithmetic terms that we now recall. We assume the variant of the Langlands reciprocity that is derived from Arthur's endoscopic treatment, and refer to \cite[Appendix B]{atobe-gan} for a succinct review.

Let $W_F$ be the Weil subgroup of the absolute Galois group of the field $F$. We write $|\cdot|$ for the norm function on $W_F$.
Repeating the definition of the previous section, we set $\Phi(G)$ to be the $\mathbf{G^\vee}(\mathbb{C})$-conjugation classes of continuous group homomorphisms
\[
\phi:W_F\times \mathrm{SL}_2(\mathbb C) \to  \mathbf{G^\vee}(\mathbb{C})\;,
\]
whose restriction to $\mathrm{SL}_2(\mathbb C)$ is an algebraic map, while $\phi(W_F)$ consists of semisimple elements.
Elements $\phi\in \Phi(G)$ are called the \textit{$L$-parameters} of the group $G$. 
Langlands reciprocity constructs a canonical finite-to-one surjective map 
\[
\irr(G) \to \Phi(G)\quad \pi \mapsto \phi_\pi\;,
\]
with favourable properties. 
For an $L$-parameter $\phi\in \Phi(G)$, the set of isomorphism classes of representations
\[
\Pi_{\phi} = \{\pi\in \irr(G)\;:\; \phi_\pi = \phi\}
\]
is the \textit{$L$-packet} attached to $\phi$.
We recall that for any $L$-parameter $\phi\in \Phi(G)$, a unipotent conjugacy class $\mathcal{O}_{\phi}^\vee\in \mathcal{U}^\vee$ is attached by taking the class of the element
\[
u_\phi := \phi\left(1,\begin{pmatrix}
    1 & 1 \\ 0 & 1
\end{pmatrix}\right)\in \mathbf{G^\vee}(\mathbb{C})\;.
\]
In particular, a partition $\lambda(\phi)\in \mathcal{P}^{s_G}(N_G)$ is an invariant that we define by the identity $\mathcal{O}^\vee_{\lambda(\phi)}=\mathcal{O}^\vee_{\phi}$.

\subsubsection{Infinitesimal characters}

The reciprocity also gives rise to the infinitesimal character invariant of representations in $\irr(G)$ which is coarser than the $L$-parameter.

Let $\Lambda(G)$ be the collection of $\mathbf{G^\vee}(\mathbb{C})$-conjugation classes of continuous group homomorphisms $\chi:W_F \to  \mathbf{G^\vee}(\mathbb{C})$, with $\chi(W_F)$ consisting of semisimple elements. We refer to $\Lambda(G)$ as the collection of \textit{infinitesimal characters} of $G$.

We fix the homomorphism
\begin{equation}\label{eq:rF}
r_F: W_F \to W_F\times \mathrm{SL}_2(\mathbb C),\quad r_F(w) = \left(w,\begin{pmatrix}
    |w|^{1/2} & 0 \\ 0 & |w|^{-1/2}
\end{pmatrix}\right)\;.
\end{equation}
For each $L$-parameter $\phi\in \Phi(G)$, its infinitesimal character $\chi_{\phi}:= \phi\circ r_F\in \Lambda(G)$ is now defined by precomposition.
A decomposition 
\begin{equation}\label{eq:lpackets}
\irr(G) = \bigsqcup_{\chi\in \Lambda(G)} \irr_{\chi} (G)
\end{equation}
now arises, when setting
\[
\irr_{\chi}(G) := \bigsqcup_{\phi\in \Phi(G)\,:\, \chi_{\phi}=\chi} \Pi_{\phi} \;.
\]
For a representation $\pi\in \irr_{\chi}(G)$, we say that its infinitesimal character is $\chi_{\pi}:=\chi\in \Lambda(G)$.

\subsubsection{Unipotent representations}

We say that an infinitesimal character $\chi\in \Lambda(G)$ is \textit{unramified}, when it is trivial on the inertia subgroup $I_F< W_F$. We denote by $\Lambda_u(G)$ the collection of unramified infinitesimal characters.

We recall that the quotient $W_F/I_F$ is a cyclic group generated by the image of (a choice of) a Frobenius element $\frb\in W_F$. 
Thus, a character $\chi\in \Lambda_u(G)$ is determined by the conjugacy class of the semisimple element $s_{\chi} = \chi(\frb)\in \mathbf{G^\vee}(\mathbb{C})$. In practice, we may identify $\Lambda_u(G)$ with the set of semisimple conjugacy classes in $\mathbf{G^\vee}(\mathbb{C})$.
Indeed, this is the point of view taken by the theory of Satake parameters. For a fixed hyperspecial maximal compact subgroup $K<G$, we say that a representation $\pi\in \irr(G)$ is \textit{spherical}, when its space contains a non-zero $K$-invariant vector.
For each unramified $\chi\in \Lambda_u(G)$, there is a unique spherical representation 
\[
\delta({\chi})\in \irr_{\chi}(G)\;.
\]
We define the sets of \textit{unipotent $L$-parameters}
\[
\Phi_u(G) = \{\phi\in \Phi(G)\;:\; \chi_{\phi}\in \Lambda_u(G)\}\;,
\]
and \textit{unipotent irreducible $G$-representations}
\[
\irr_u(G) = \bigsqcup_{\chi\in \Lambda_u(G)} \irr_{\chi}(G) = \bigsqcup_{\phi\in \Phi_u(G)} \Pi_{\phi} \;.
\]
This latter set happens to coincide with the Lusztig notion of unipotent representations that is defined in terms of parahoric restriction \cite{Lu-unip1}. 

\begin{remark}
Note also that an $L$-parameter $\phi\in\Phi(G)$ is unipotent, if and only if, its restriction $\phi|_{I_F}$ is a trivial homomorphism. In particular, an $L$-parameter $\phi\in \Phi_u(G)$ is determined by the conjugacy class of the pair of commuting elements $(\phi(\frb), u_{\phi})$, which amounts to the Jordan decomposition of the element $g_{\phi}:= \phi(\frb)u_{\phi}\in \mathbf{G^\vee}(\mathbb{C})$ in the reductive algebraic group. 
Thus, the set of unipotent $L$-parameters $\Phi_u(G)$ is in a natural bijection with the conjugacy classes of $\mathbf{G^\vee}(\mathbb{C})$. Yet, this point of view will not be prominent in our discussion.

\end{remark}

\subsubsection{Spinor norm character}\label{sec:spinornorm}

Let us note that the center $Z_{\mathbf{SO}_{2n+1}}$ is trivial, while $Z_{\mathbf{Sp}_{2n}}$ is a group of $2$ elements. In the latter case, we write $-1\in \mathbf{Sp}_{2n}(\mathbb{C})$ for the non-trivial element in the center.

We write $\kappa_0: W_F \to \{\pm1\}$ for the quadratic character corresponding to the unique unramified quadratic extension of the field $F$.
Viewing $\kappa_0$ as a homomorphism $W_F\to Z_{\mathbf{Sp}_{2n}}$, we see that the tensor operation $\kappa_0\otimes-$ gives an involution on the set of $L$-parameters $\Phi(\mathrm{SO}_{2n+1}(F))$, for any $n\geq1$, which preserves the set of unramified $L$-parameters $\Phi_u(\mathrm{SO}_{2n+1}(F))$.
Clearly, we have $\kappa_0 \otimes \chi_\phi  = \chi_{\kappa_0\otimes \phi}$, for $\phi\in \Phi(\mathrm{SO}_{2n+1}(F))$.
Indeed, this involution may be explicated on the level of corresponding $G$-representations. 

Special orthogonal groups admit a homomorphism
\[
\mathrm{sp}:\mathrm{SO}_{2n+1}(F)\to F^\times/(F^\times)^2
\]
known as the \emph{spinor norm}. It can be characterized as the homomorphism that takes any reflection along an anisotropic vector $w\in V$ to $B(w,w)(F^\times)^2$ (\cite[Theorem V.1.13]{lam-book}), where $(V,B)$ is the quadratic space defining the $p$-adic group as in Section \ref{sec:maxcpcts}.
Composing the norm $\mathrm{sp}$ with the quadratic character of $F^\times$ obtained from $\kappa_0$ via local class field theory, produces a quadratic $\mathrm{SO}_{2n+1}(F)$-character, which is trivial on the Iwahori subgroup $I_{\mathrm{SO}_{2n+1}(F)}$.
Abusing notation, we denote that character by $\kappa_0$ as well.
In these terms, the Langlands reciprocity map satisfies
\[
\kappa_0 \otimes \phi_{\pi} = \phi_{\kappa_0\otimes \pi}\;,
\]
for any representation $\pi \in \irr(\mathrm{SO}_{2n+1}(F))$ (\cite[Proposition 4.2.2]{liwenwei}).

\subsubsection{Explication of $L$-parameters}\label{sect:expl-l}

We denote by $\Phi(N)$ the set of isomorphism classes of complex $N$-dimensional continuous representations $\phi$ of the group $W_F\times \mathrm{SL}_2(\mathbb{C})$, whose restriction to $\mathrm{SL}_2(\mathbb C)$ is algebraic, while $\phi(W_F)$ consists of semisimple elements. (i.e. $L$-parameters for the group $GL_N(F)$.)
It is convenient to define the set $\Phi= \bigoplus_{N\geq1} \Phi(N)$ and treat it as an additive semigroup with respect to the direct sum operation on representations.

We denote by $\Phi_u\subset \Phi$ the collection of representations that are trivial on the inertia group $I_F< W_F$. 
By assuming the standard embedding of $\mathbf{G^\vee}$ into $\mathbf{GL}_{N_G}$, we obtain an embedding $\Phi(G)\subset \Phi(N_G)$. Clearly, $\Phi_u(G)= \Phi_u \cap \Phi(G)$ holds.
We now recall the explicit structure of $L$-parameters in $\Phi(G)$. Since our main focus is set on unipotent representations, we limit this description to $\Phi_u(G)$.
For $\zeta\in \mathbb{C}^\times$ and an integer $k\geq1$, we write 
\[
\zeta\otimes \nu_k\in \Phi_u
\]
for the $W_F\times \mathrm{SL}_2(\mathbb{C})$-representation given by the $k$-dimensional irreducible $\mathrm{SL}_2(\mathbb{C})$-representation, tensored with the $1$-dimensional character of $W_F/I_F$ that is determined by $\frb \mapsto \zeta$.

\begin{remark}\label{rem:inf-ch-comp}
We note that for such $\phi = \zeta \otimes \nu_k\in \Phi_u$, the precomposition $\phi(r_F(\frb))$ may be viewed as a semisimple matrix whose eigenvalues are specified by the set
\[
\{ \zeta q^{\frac{k-1}2}, \zeta q^{\frac{k-3}2},\ldots, \zeta q^{-\frac{k-1}2}\}\;,
\]
all appearing with multiplicity $1$.
\end{remark}

For a partition $\lambda = (\lambda_1\leq \ldots \leq \lambda_{\ell(\lambda)})\in \mathcal{P}$ and a tuple of numbers $\underline{\zeta}= (\zeta_1,\ldots, \zeta_{\ell(\lambda)})$ in $\mathbb{C}^\times$, we write
\begin{equation}\label{eq:l-expl}
\phi_{\underline{\zeta},\lambda} = \sum_{i=1}^{\ell(\lambda)} \zeta_i \otimes \nu_{\lambda_i}\in \Phi_u \;.
\end{equation}

\begin{proposition}\label{prop:lambdaphi}

Every $L$-parameter $\phi\in \Phi_u(G)$ can be written in the form 
\[
\phi = \phi_{(1,\ldots,1),\lambda} +\phi_{(-1,\ldots,-1),\lambda'} + \phi_{(\zeta_1,\ldots,\zeta_k),\mu} + \phi_{(\zeta_1^{-1},\ldots,\zeta_k^{-1}),\mu}\;,
\]
where $\lambda,\lambda'\in \mathcal{P}^{s_G}$ and $\mu\in \mathcal{P}$ are partitions with $|\lambda| + |\lambda'| + 2|\mu| = N_G$. 
This form is unique for a given $L$-parameter, when assuming $\zeta_i\neq \pm1$, for all $1\leq i \leq k$.
The associated partition to the $L$-parameter $\phi\in \Phi(G)$ is given by
\[
\lambda(\phi) = \lambda \cup \lambda' \cup \mu \cup \mu\in \mathcal{P}^{s_G}(N_G)\;.
\]

\end{proposition}

\begin{proof}
This is a straightforward translation into our notation of standard descriptions, such as in \cite[Section 3.1]{atobe-jac}.
\end{proof}

The following property is a well-known feature.

\begin{lemma}\label{lem:temp-closure}
Let $\lambda\in \mathcal{P}^{s_G}(N_G)$ be a partition, and $z\in \{\pm1\}$.
Suppose that $\phi\in \Phi(G)$ is an $L$-parameter satisfying 
\[
\chi_{\phi} = \chi_{\phi_{(z,\ldots,z),\lambda}}\;.
\]
Then, the class $\mathcal{O}_{\phi}^\vee\in \mathcal{U}^\vee$ must be contained in the Zariski closure of the class $\mathcal{O}^\vee_{\lambda}\in \mathcal{U}^\vee$.
    
\end{lemma}

\begin{proof}
Let us consider an algebraic map $\nu_{\lambda}:\mathrm{SL}_2(\C)\to \mathbf{G^\vee}(\C)$ that satisfies $u_{\lambda} = \nu_{\lambda}\begin{pmatrix}
    1 & 1 \\ 0 & 1
\end{pmatrix} \in \mathcal{O}_{\lambda}^\vee$. We denote $s_{\lambda} = \nu_{\lambda}\begin{pmatrix}
    q & 0 \\ 0 & q^{-1}
\end{pmatrix}$. By conjugating we may assume that
\[
\chi_{\phi}(\frb) = \chi_{\phi_{(z,\ldots,z),\lambda}}(\frb) = zs_{\lambda}\;.
\]

Hence, $s_{\lambda} u_{\phi}s_{\lambda}^{-1} = u_{\phi}^q$ holds. By known theory (see \cite[Lemma 3.0.1]{cmbo-arthur}) that forces $u_{\phi}$ to belong to the Zariski closure of the class $\mathcal{O}^\vee_{\lambda}$.

\end{proof}

\subsubsection{Wavefront sets}

We recall that $\mathcal{N}_F$ was defined to be the set of nilpotent $\Ad(G)$-orbits in the Lie algebra $\Lie(G)$.
Similarly, for the algebraic closure of the field $F<\overline{F}$, we denote $\mathcal{N}_{\overline{F}}$ to be the set of nilpotent $\Ad(\mathbf{G}(\overline{F}))$-orbits in the Lie algebra $\Lie(\mathbf{G}(\overline{F}))$.
Clearly, there is a (topological) order-preserving map $\mathrm{alg}:\mathcal{N}_F\to \mathcal{N}_{\overline{F}}$, that takes an $\Ad(G)$-orbit $\mathcal{O}$ to $\mathrm{alg}(\mathcal{O})=\Ad(\mathbf{G}(\overline{F}))(\mathcal{O})$.

It is possible to identify the orbits in $\mathcal{N}_{\overline{F}}$ with the corresponding set of adjoint orbits in $\Lie(\mathbf{G}(\mathbb{C}))$. Furthermore, the exponential map then allows for an identification of $\mathcal{N}_{\overline{F}}$ with the set of unipotent conjugacy classes in the complex group $\mathbf{G}(\mathbb{C})$. Altogether, we may apply Proposition \ref{prop:part-orbit} to obtain a natural parameterization of $\mathcal{N}_{\overline{F}}$ by the set  of partitions $\mathcal{P}^{-s_G}(N_G-s_G)$. 
In this manner, we write $\mathcal{O}_{\lambda}\in \mathcal{N}_{\overline{F}}$, for a partition $\lambda\in \mathcal{P}^{-s_G}(N_G-s_G)$.

The celebrated theory \cite{hcbook} of the Harish-Chandra--Howe character studies the trace distribution $\Theta_{\pi}$ on $G$ that is attached to each representation $\pi\in \irr(G)$. 
The local character expansion states that when $\Theta_{\pi}$ is pushed onto the Lie algebra through the exponential map and restricted to a small enough neighborhood of $0\in \Lie(G)$, it will equal a linear combination of the form 
\[
\sum_{\mathcal{O}\in \mathcal{N}_F} c_{\mathcal{O}}(\pi) \widehat{\mu}_{\mathcal{O}}\;.
\]
Here, the distributions $\{\widehat{\mu}_{\mathcal{O}}\}_{\mathcal{O}\in \mathcal{N}_F}$ on $\Lie(G)$ are Fourier transforms of those given by the corresponding nilpotent orbital integrals, while $\{c_{\mathcal{O}}(\pi)\}_{\mathcal{O}\in \mathcal{N}_F}$ are scalars.
We say that an representation $\pi\in \irr(G)$ \textit{admits an algebraic wavefront orbit}, if there exists a nilpotent orbit $\mathcal{O}\in \mathcal{N}_{F}$ with $c_{\mathcal{O}}(\pi)\neq 0$, such that for any orbit $\mathcal{O}'\in \mathcal{N}_{F}$ with $c_{\mathcal{O}'}(\pi)\neq 0$, $\mathrm{alg}(\mathcal{O}')$ is contained in the Zariski closure of $\mathrm{alg}(\mathcal{O})$.
In this case we can write
\[
\WF(\pi):=\mathrm{alg}(\mathcal{O})\in \mathcal{N}_{\overline{F}}\;,
\]
which is clearly well-defined.
The recent work of Tsai \cite{tsai-jams} provides further insight into this concept.

Clearly, when a representation $\pi$ admits an algebraic wavefront orbit, its Gelfand-Kirillov dimension $\gkd(\pi)$, as defined in the introduction section, will equal half the Zariski dimension of the algebraic variety $\WF(\pi)$.
A main result of \cite{cmowavefront} is a formula for the algebraic wavefront orbit of unipotent representations in terms of the Langlands reciprocity.

\begin{theorem}\label{thm:cmbo}
Let $\chi\in \Lambda_u(G)$ be an unramified infinitesimal character, for which $s_{\chi}\in \mathbf{G}^\vee(\CC)$ is a matrix with real positive eigenvalues.
Then, all irreducible representations in $\irr_{\chi}(G)$ admit an algebraic wavefront orbit.

For a representation $\pi\in \irr_{\chi}(G)$, whose Aubert dual $\pi^t\in \Pi_{\phi}$ is contained in an $L$-packet given by a unipotent $L$-parameter $\phi\in \Phi_u(G)$, the formula
\[
\WF(\pi) = \mathcal{O}_{d(\lambda(\phi))}\in \mathcal{N}_{\overline{F}}
\]
holds, where $d$ is the Barbasch--Vogan--Lusztig--Spaltenstein duality map.

\end{theorem}

This extends an analogous result due to Waldspurger for anti-tempered representations of odd orthogonal groups \cite{waldspurgerwavefront}.

\begin{corollary}\label{cor:WFtwist}
Let $\lambda\in \mathcal{P}^{s_G}(N_G)$ be a partition, and $z\in \{\pm1\}$ a sign.

Then, any representation $\pi\in \Irr(G)$ with infinitesimal character $\chi_\pi=\chi_{\phi_{(z,\ldots,z),\lambda}}$
admits an algebraic wavefront set, which is given by
\[
\WF(\pi) = \mathcal{O}_{\mu}\in \mathcal{N}_{\overline{F}}\;,
\]
with $\mu = d(\lambda(\phi_{(\pi^t)}))$.

Moreover, $\dim(\WF(\pi)) \geq \dim (\mathcal{O}_{d(\lambda)})$ holds, for any such $\pi$.
\end{corollary}

\begin{proof}
The case of $z=-1$ is reduced to that of $z=1$, when noting the facts that the infinitesimal character of $\kappa_0\otimes \pi$ equals that of the $L$-parameter $\kappa_0\otimes \phi_{(z,\ldots,z),\lambda} = \phi_{(-z,\ldots,-z),\lambda}$, that $(\kappa_0\otimes \pi)^t = \kappa_0\otimes \pi^t$, and that the character distributions $\Theta_{\pi}, \Theta_{\kappa_0\otimes \pi}$ coincide around $0\in \mathrm{Lie}(G)$.

With $z=1$ being assumed, and following the notations of the proof of Lemma \ref{lem:temp-closure}, we see that $s_{\chi_{\pi}} = s_{\lambda}\in \mathbf{G}^\vee(\C)$. The latter semisimple element has real positive eigenvalues, since the morphism $\nu_\lambda$ is defined over the integers. Hence, recalling that $\chi_{\pi^t}=\chi_{\pi}$, Theorem \ref{thm:cmbo} is applicable for $\pi^t$.

The duality $d$, viewed as a map $\mathcal{U}^\vee\to \mathcal{N}_{\overline{F}}$ between sets of orbits, is known to be order-reversing, with respect to the Zariski topological order. Thus, the last inequality follows from Lemma \ref{lem:temp-closure}.

\end{proof}

\section{Arthur theory}

Let us recall some properties of the theory of local Arthur packets, that were developed in \cite{Abook} through an intricate analysis of endoscopic transfer. 

For each $A$-parameter $\psi\in \Psi(G)$ (as defined in the introduction section), the associated Arthur packet, or \textit{$A$-packet}, is a finite set of representations $\Pi^A_{\psi}\subset \irr(G)$. 
Irreducible representations in the collection
\[
\bigcup_{\psi\in \Psi(G)}\Pi^A_{\psi} \subset \irr(G)
\]
are called \textit{Arthur-type} representations. In contrast with the $L$-packet decomposition of \eqref{eq:lpackets}, the union in the preceding equation is not disjoint, i.e. $A$-packets are known to overlap.
A key feature of the theory, though one which will not play a direct role in our discussion, is that all Arthur-type representations are unitarizable.

For an $A$-parameter $\psi\in \Psi(G)$, we set the precomposition $\phi_{\psi}: = \psi\circ r'_F \in \Phi(G)$ to be its associated $L$-parameter, where 
\[
r'_F: W_F \times \mathrm{SL}_2(\mathbb{C}) \to W_F\times \mathrm{SL}_2(\mathbb C) \times \mathrm{SL}_2(\mathbb{C}),\quad r'_F(w,x) = \left(w, x,\begin{pmatrix}
    |w|^{1/2} & 0 \\ 0 & |w|^{-1/2}
\end{pmatrix}\right)\;.
\]
is a homomorphism resembling the form of the one defined in \eqref{eq:rF}.
We also set $\chi_{\psi}: = \chi_{\phi_{\psi}}=\psi\circ r'_F\circ r_F\in \Lambda(G)$ to be the associated infinitesimal character.
Marking the flip $\mathrm{fp}(x,y)= (y,x)$ on $\mathrm{SL}_2(\mathbb{C})\times \mathrm{SL}_2(\mathbb{C})$, we obtain an involution  
\[
\psi^t:= \psi\;\circ\; \mathrm{fp}
\]
on the set of $A$-parameters $\Psi(G)$. 
We note that $\chi_{\psi^t}= \chi_{\psi}$ holds.

\begin{proposition}\label{prop:start-Arth}
\cite[Proposition 7.4.1]{Abook}\cite[Proposition 4.1]{moeg09}\cite[Theorem 1.6]{atobe-crelle1}
For any $A$-parameter $\psi\in \Psi(G)$, we have containments
\[
\Pi_{\phi_{\psi}} \subset \Pi^A_{\psi}\subset \irr_{\chi_{\psi}}(G)\;,
\]
and a compatibility with the Aubert duality in $\irr(G)$ in the sense of
\[
\Pi^A_{\psi^t} = \{\pi^t \;:\; \pi \in \Pi^A_{\psi}\}\;.
\]

\end{proposition}

For notation purposes, let us write the domain of $A$-parameters as a group
\[
W_F\times \mathrm{SL}^L_2(\mathbb C) \times \mathrm{SL}^A_2(\mathbb{C})\;,
\]
where each of $\mathrm{SL}^L_2, \mathrm{SL}^A_2$ stands for a copy of the algebraic group $\mathrm{SL}_2$.

A parameter $\psi\in \Psi(G)$ is said to be \textit{tempered}, when its restriction $\psi|_{\mathrm{SL}^A_2(\mathbb{C})}$ is a trivial homomorphism.
For a tempered $A$-parameter $\psi\in \Psi(G)$, an equality $\Pi_{\phi_{\psi}} = \Pi^A_{\psi}$ is known to hold. Moreover, a representation in $\irr(G)$ is tempered in the analytic sense, if and only if, it belongs to an $A$-packet (which is also an $L$-packet) for a tempered parameter.

A parameter $\psi\in \Psi(G)$ is said to be \textit{anti-tempered}, when its restriction $\psi|_{\mathrm{SL}^L_2(\mathbb{C})}$ is a trivial homomorphism.
A representation $\pi\in \irr(G)$ is said to be \textit{anti-tempered}, whenever $\pi^t$ is tempered. 
Thus, by Proposition \ref{prop:start-Arth} anti-tempered irreducible representations clearly coincide with the constituents of $A$-packets that arise from anti-tempered $A$-parameters.

\subsubsection{Arthur's characters}

For $\psi\in \Psi(G)$, we write
\[
\mathcal{S}_{\psi} = Z_{\psi}/Z_{\psi}^{\circ} Z_{\mathbf{G^\vee}}
\]
for the component group of the centralizer subgroup $Z_{\psi} = Z_{\mathbf{G^\vee}(\mathbb{C})}(\mathrm{Im}(\psi))<\mathbf{G^\vee}(\mathbb{C})$, taken modulo the representatives of the center $Z_{\mathbf{G^\vee}} = Z(\mathbf{G^\vee}(\mathbb{C}))$.
We write $\widehat{\mathcal{S}_{\psi}}$ for the dual group of complex characters on $\mathcal{S}_{\psi}$. Both are finite $2$-groups.
Analogous definitions for the group $\mathcal{S}_{\phi}$, and its dual $\widehat{\mathcal{S}_{\phi}}$, are also set in place for each $L$-parameter $\phi\in \Phi(G)$.

For an $A$-parameter $\psi\in \Psi(G)$, the embedding of centralizers $Z_{\psi}< Z_{\phi_{\psi}}$ gives rise to a surjective map $\mathcal{S}_{\psi} \to \mathcal{S}_{\phi_{\psi}}$. Dualizing, we obtain an embedding $\widehat{\mathcal{S}_{\phi_{\psi}}} < \widehat{\mathcal{S}_{\psi}}$ of finite groups.
Arthur has attached a map
\begin{equation}\label{eq:epsilon-arth}
\Pi^A_{\psi} \to \widehat{\mathcal{S}_{\psi}}\qquad  \pi\mapsto \epsilon_{\pi}^{\psi}\;,
\end{equation}
to each $A$-packet $\psi\in \Psi(G)$. Its definition is pinned down by certain endoscopic identities that are required to be satisfied by the constituents of $\Pi^A_{\psi}$.
In particular, the invariant $\epsilon^{\psi}_{\pi}$, for Arthur-type representations $\pi$, provides an approach for an \textit{enhanced} Langlands reciprocity, that is, a meaningful labelling of irreducible representation within a single $L$-packet.
In that regard, the information that is provided by the following proposition will suffice for our needs.

\begin{proposition}\label{prop:arth-book}
\cite[Proposition 7.4.1]{Abook}
For any $A$-parameter $\psi\in \Psi(G)$, the map in \eqref{eq:epsilon-arth} is injective when restricted to the $L$-packet $\Pi_{\phi_{\psi}}\subset \Pi^A_{\psi}$. 
Moreover, an equality
\[
\widehat{\mathcal{S}_{\phi_{\psi}}} = \{\epsilon_{\pi}^{\psi}\: :\: \pi\in \Pi_{\phi_{\psi}}\}
\]
holds.
In particular, for the $L$-parameter $\phi =\phi_{\psi}\in \Phi(G)$, we may parameterize the associated $L$-packet as
\[
\Pi_{\phi}=\{\pi(\phi, \epsilon)\,:\,\epsilon \in\widehat{\mathcal{S}_{\phi}}\}\;,
\]
so that $\epsilon=  \epsilon^{\pi}_{\psi}$, whenever $\pi = \pi(\phi, \epsilon)$.

\end{proposition}

\begin{corollary}\label{cor:temp-anti}
For any tempered or anti-tempered $A$-parameter $\psi\in \Psi(G)$, the map in \eqref{eq:epsilon-arth} is a bijection, and can be used to parameterize the representations
\[
\Pi^A_{\psi}=\{\pi(\psi, \epsilon)\,:\,\epsilon \in\widehat{\mathcal{S}_{\psi}}\}\;,
\]
of the associated $A$-packet, so that $\epsilon=  \epsilon^{\pi}_{\psi}$, whenever $\pi = \pi(\psi, \epsilon)$.
.
\end{corollary}

\begin{proof}
The tempered case follows from Proposition \ref{prop:arth-book} when recalling that $\Pi_{\phi_{\psi}} =\Pi_{\psi}^A$ and that $\widehat{\mathcal{S}_{\phi_{\psi}}}= \widehat{\mathcal{S}_{\psi}}$ hold.

For an anti-tempered $A$-parameter $\psi\in \Psi(G)$, we clearly have $\mathcal{S}_{\psi}= \mathcal{S}_{\psi^t}$. Hence, this case follows from the tempered case through Aubert duality. 
\end{proof}

\begin{remark}
The map in \eqref{eq:epsilon-arth} is in general dependant on a fixed choice of a Whittaker datum for the group $G$, which is not canonical in the case that the group is symplectic. 
Yet, the interest of this work in the characters $\epsilon^{\psi}_{\pi}$ will be limited to the case of $A$-parameters that admit a quasi-basic infinitesimal character $\chi_{z,\mathcal{O}^\vee}$, in the sense of Section \ref{sect:basicA}. 
One can consult the established formulas for change of character under a change of Whittaker datum. These formulas are detailed, for instance, in \cite[Section 10]{ggp} or \cite[Section 3.2.1]{jlz22}, revealing that in these cases $\epsilon^{\psi}_{\pi}$ is in fact independent of the chosen datum.
\end{remark}

\subsection{Basic $A$-packets}\label{sect:basicA}

Among anti-tempered $A$-parameters, we say that $\psi\in \Psi(G)$ is \textit{basic}, when its restriction $\psi|_{W_F\times \mathrm{SL}^L_2(\mathbb{C})}$ is a trivial homomorphism.
More generally, we say that an $A$-parameter $\psi\in \Psi(G)$ is \textit{quasi-basic}, when its restriction $\psi|_{I_F\times \mathrm{SL}^L_2(\mathbb{C})}$ is a trivial homomorphism, and $\psi(\frb)\in Z_{\mathbf{G^\vee}}$.
It is evident that basic $A$-parameters are classified by the collection of algebraic homomorphisms $\mathrm{SL}^A_2\to \mathbf{G^\vee}$, up to conjugation. By the Jacobson-Morozov theorem those may be parameterized by the classes of $\mathcal{U}^\vee$. 
Namely, for each conjugacy class $\mathcal{O}^\vee\in \mathcal{U}^\vee$, we denote by $\psi_{\mathcal{O}^\vee}\in \Psi(G)$ the basic $A$-parameter which satisfies
\[
\psi_{\mathcal{O}^\vee}|_{\mathrm{SL}^A_2(\mathbb{C})} \left(\begin{pmatrix}
    1 & 1 \\ 0 & 1
\end{pmatrix}\right) \in \mathcal{O}^\vee\;.
\]
The set of quasi-basic $A$-parameters in $\Psi(G)$ is easily seen to be described as 
\[
\Psi_{qb}(G):=\{\psi_{z,\mathcal{O}^\vee}\}_{z\in Z_{\mathbf{G^\vee}}, \mathcal{O}^\vee\in \mathcal{U}^\vee}\;,
\]
where each $\psi_{z,\mathcal{O}^\vee}\in \Psi(G)$ is determined by 
\[
\psi_{z,\mathcal{O}^\vee}(\frb) = z\, , \quad \psi_{z,\mathcal{O}^\vee}|_{I_F\times \mathrm{SL}^L_2(\mathbb{C})\times \mathrm{SL}^A_2(\mathbb{C})} = \psi_{\mathcal{O}^\vee}|_{I_F\times \mathrm{SL}^L_2(\mathbb{C})\times \mathrm{SL}^A_2(\mathbb{C})}\;.
\]
For a quasi-basic $A$-parameter $\psi = \psi_{z,\mathcal{O}^\vee} \in \Psi(G)$, let us also write
\[
\chi_{z,\mathcal{O}^\vee}:= \chi_{\psi}\in \Lambda_u(G)\,,\quad \phi_{z,\mathcal{O}^\vee}:= \phi_{\psi}\in \Phi_u(G)\;,
\]
for its associated (unipotent) infinitesimal character and $L$-parameter. 

Since the center $Z_{\mathbf{SO}_{2n+1}}$ is trivial, all quasi-basic $A$-parameters are basic in the case of $\mathbf{G^\vee}= \mathbf{SO}_{2n+1}$. 
In the case of $\mathbf{G^\vee}= \mathbf{Sp}_{2n}$, we treat the central parameter $z\in Z_{\mathbf{Sp}_{2n}}$, for $\psi_{z,\mathcal{O}^\vee}\in \Psi(G)$, merely as a choice of a sign $z\in \{\pm1\}$.
Still in the same case, we clearly have $\kappa_0\otimes \phi_{z,\mathcal{O}^\vee} =\phi_{-z,\mathcal{O}^\vee}$ and $\kappa_0\otimes \chi_{z,\mathcal{O}^\vee} =\chi_{-z,\mathcal{O}^\vee}$.

\begin{definition}
For a unipotent conjugacy class $\mathcal{O}^\vee \in \mathcal{U}^\vee$ and a central element $z\in Z_{\mathbf{G}^\vee}$, we set
\[
\Pi_{z,\mathcal{O}^\vee}:= \Pi^A_{\psi_{z,\mathcal{O}^\vee}}\subset \irr_{\chi_{z,\mathcal{O}^\vee}}(G)
\]
to be the associated \textit{quasi-basic} $A$-packet.
When $z$ is trivial, we write $\Pi_{\mathcal{O}^\vee} = \Pi_{1,\mathcal{O}^\vee}$ and call it a \textit{basic} $A$-packet.
\end{definition}

We take note of the subclass of unipotent representations given as
\[
\irr_0(G):= \bigsqcup_{\psi\in \Psi_{qb}(G)} \irr_{\chi_{\psi}} (G) \subset \irr_u(G)\;,
\]
in which all quasi-basic $A$-packets are contained. We say that representations in $\irr_0(G)$ are \textit{integral}.

Note, that integral representations may be found as constituents of Arthur packets that are not quasi-basic (or not be of Arthur-type). Therefore, it makes sense to define the set of \textit{integral $A$-parameters} as
\[
\Psi_0(G) = \{\psi\in \Psi(G)\::\: \exists \psi' \in \Psi_{qb}(G),\; \chi_{\psi} = \chi_{\psi'}\}\;,
\]
whose associated $A$-packets give rise to all integral Arthur-type representations.

\subsubsection{Formally weakly spherical representations}\label{sect:formal-wsph}

We note that for each $\mathcal{O}^\vee \in \mathcal{U}^\vee$ and $z\in Z_{\mathbf{G^\vee}}$, an identity 
\begin{equation}\label{eq:311}
\mathcal{S}_{\psi_{z,\mathcal{O}^\vee}} = A(\mathcal{O}^\vee)/ Z_{\mathbf{G^\vee}}
\end{equation}
holds by standard $\mathfrak{sl}_2$-triple theory (see, for example, \cite[Remark 3.7.5 (ii)]{collingwoodmcgovern}).
Building upon this identification and the labelling of anti-tempered representations given in Corollary \ref{cor:temp-anti}, we may now define the introduction section notation of \eqref{eq-intro}, as
\[
\delta(z,\mathcal{O}^\vee,\epsilon): = \pi(\psi_{z,\mathcal{O}^\vee}, \epsilon)\in \Pi_{z,\mathcal{O}^\vee}\;,
\]
for each $z\in \ZZ$, $\mathcal{O}^\vee\in \mathcal{U}^\vee$ and $\epsilon\in \widehat{A(\mathcal{O}^\vee)}_0$.
Within this labelling, we take note of the following class of irreducible representations.

\begin{definition}\label{defi:formallyws}
For a unipotent conjugacy class $\mathcal{O}^\vee \in \mathcal{U}^\vee$ and a central element $z\in Z_{\mathbf{G}^\vee}$, we denote the set of irreducible anti-tempered representations
\[
\Pi^{f}_{z,\mathcal{O}^\vee} = \{ \delta(z,\mathcal{O}^\vee, \epsilon)\in\Pi_{z,\mathcal{O}^\vee}\;:\; \epsilon\in A^{\dagger}(\mathcal{O}^\vee)\}
\]
parameterized by the group of characters $A^{\dagger}(\mathcal{O}^\vee) < \widehat{A(\mathcal{O}^\vee)}_0$ of Proposition \ref{prop:quotient}.
We say that the representations in $\Pi^{f}_{z,\mathcal{O}^\vee}$ are \textit{formally weakly spherical}.

\end{definition}

A canonical constituent in $\Pi^f_{z,\mathcal{O}^\vee}$, for each $\mathcal{O}^\vee \in \mathcal{U}^\vee$ and $z\in Z_{\mathbf{G}^\vee}$, is the spherical representation
\[
\delta_{z,\mathcal{O}^\vee}:= \delta(\chi_{z,\mathcal{O}^\vee}) = \delta(z,\mathcal{O}^\vee, \mathrm{trv}) \in \irr_{\chi_{z,\mathcal{O}^\vee}}(G)\;.
\]
Here, $\mathrm{trv}\in \widehat{\mathcal{S}_{\psi_{z,\mathcal{O}^\vee}}}$ is the trivial character.
Furthermore, it is also easy to verify that the group $\mathcal{S}_{\phi_{z,\mathcal{O}^\vee}}$, for the associated $L$-parameter $\phi_{z,\mathcal{O}^\vee}$, is a trivial group, i.e. that the centralizer of $\chi_{\mathcal{O}^\vee}(\frb)\in \mathbf{G^\vee}(\mathbb{C})$ is connected modulo $Z_{\mathbf{G}^\vee}$.
Hence, with the character group $\widehat{\mathcal{S}_{\phi_{z,\mathcal{O}^\vee}}}$ being trivial, we see by Proposition \ref{prop:arth-book} that the associated $L$-packet to a quasi-basic $A$-parameter is the singleton set
\[
\{\delta_{z,\mathcal{O}^\vee}\} = \Pi_{\phi_{z,\mathcal{O}^\vee}} \subset \Pi^f_{z,\mathcal{O}^\vee}\subset \Pi_{z,\mathcal{O}^\vee}\subset \irr_{\chi_{z,\mathcal{O}^\vee}}(G)
\]
consisting of the spherical representation that admits the corresponding infinitesimal character.

\subsection{Explication of $A$-parameters}\label{sect:Aexplicit}

Let us recall the existing theory on the explication of the composition of $A$-packets in combinatorial terms, when applied to the case of Arthur-type representations in $\irr_0(G)$.

\subsubsection{Tables}

A \textit{table} $\gotM$ will consist of a finite indexing set $I(\gotM)$ and a set of pairs of integers $(a_i,b_i)\in \mathbb{Z}_{>0}\times \mathbb{Z}_{>0}$, for each $i\in I(\gotM)$.
We write $\mathcal{T}$ for the set of tables. 
For an integer $N\geq1$, we set $\mathcal{T}(N)\subset \mathcal{T}$ to be the set of tables $\gotM$ with $|\gotM|:= \sum_{i\in I(\gotM)} a_ib_i = N$.
We fix natural maps
\[
i: \mathcal{P} \hookrightarrow \mathcal{T}\,,\quad p: \mathcal{T} \twoheadrightarrow \mathcal{P}
\]
between the set of tables and the set of partitions $\mathcal{P}$.

Given a partition $\lambda =( \lambda_1\leq \ldots \leq \lambda_{\ell(\lambda})\in \mathcal{P}$, a table $\gotM_{\lambda} = i(\lambda)\in \mathcal{T}$ is constructed by taking the indexing set $I(\gotM_{\lambda}) = \{1,\ldots, \ell(\lambda)\}$ and setting $(a_i,b_i) = (\lambda_i,1)$, for each $i\in I(\gotM_{\lambda})$.
In an adjoint manner, for a table $\gotM\in \mathcal{T}$, we construct
\[
\lambda_{\gotM} =  p(\gotM) = \cup_{i\in I(\gotM)} (a_i^{b_i})\in \mathcal{P}\;.
\]
It is evident that $p(i(\lambda)) = \lambda$ holds, for any partition $\lambda\in \mathcal{P}$.
For $\gotM^1,\gotM^2\in \mathcal{T}$, we write $\gotM^1\cup\gotM^2\in \mathcal{T}$ for the table given by the pairs $\{(a_i,b_i)\}_{i \in I(\gotM^1)\cup I(\gotM^2)}$, with $I(\gotM^1\cup \gotM^2) = I(\gotM^1)\sqcup I(\gotM^2)$.
For a table $\gotM\in\mathcal{T}$ and a set $S \subset \mathbb{Z}_{>0}\times \mathbb{Z}_{>0}$, we may also write a table $\gotM \cup S\in \mathcal{T}$, by implicitly assuming an auxiliary indexing set $I_S$, so that $I(\gotM\cup S) = I(\gotM) \sqcup I_S$ with $S = \{(a_i,b_i)\}_{i\in I_S}$.
We write $\mathcal{T}^1_0$ (respectively, $\mathcal{T}^{-1}_0$) for the set of tables $\gotM\in \mathcal{T}$, for which $a_i+b_i$ is even (respectively, odd), for all $i\in I(\gotM)$.
For each choice of a sign $s\in \{\pm1\}$, we define the sets of tables
\[
\mathcal{T}^{s} = \{ \gotM \in \mathcal{T}\;:\; \exists \gotM_0\in \mathcal{T}^{s}_0,\; \exists \gotN\in \mathcal{T},\; \gotM = \gotM_0\cup \gotN\cup \gotN \} \;,
\]
and write $\mathcal{T}^{s}(N) = \mathcal{T}(N) \cap \mathcal{T}^{s}$, for every integer $N\geq1$.
The maps $i,p$ clearly restrict to well-defined maps
\[
i: \mathcal{P}^s(N) \hookrightarrow \mathcal{T}^s(N)\,,\quad p: \mathcal{T}^s(N) \twoheadrightarrow \mathcal{P}^s(N)\;,
\]
for any $N\geq1$ and a choice of sign $s\in \{\pm1\}$.

\subsubsection{Integral $A$-packets}

We denote by $\Psi(N)$ the set of isomorphism classes of complex $N$-dimensional continuous representations $\psi$ of the group $W_F\times \mathrm{SL}_2(\mathbb{C})\times \mathrm{SL}_2(\mathbb{C})$, whose restriction to $\mathrm{SL}_2(\mathbb C)\times \mathrm{SL}_2(\mathbb{C})$ is algebraic, while $\psi(W_F)$ is bounded and consists of semisimple elements. (i.e. $A$-parameters for the group $GL_N(F)$.)
A natural map (in fact, injection) $\Psi(N) \hookrightarrow \Phi(N)$ is given by $\psi\mapsto \phi_{\psi}: = \psi\circ r'_F$.

For $z\in Z_{\mathbf{G^\vee}}$ and integers $a,b\geq1$, we write 
\[
z \otimes \nu_a \otimes \nu_b\in \Psi(ab)
\]
for the tensored representation, with $\nu_k$ denoting the $k$-dimensional irreducible $\mathrm{SL}_2(\mathbb{C})$-representation, and $z$ denoting the $1$-dimensional character of $W_F/I_F$ that is determined by $\frb \mapsto z$.
An explication shows that for $\psi = z \otimes \nu_a \otimes \nu_b$, we have
\begin{equation}\label{eq:phipsi}
\phi_{\psi} = \sum_{i=0}^{b-1} zq^{i-\frac{b-1}2}\otimes \nu_a \in \Phi(N)\;.
\end{equation}
In similarity with Section \ref{sect:expl-l}, we embed $\Psi(G)\subset \Psi(N_G)$, and treat $\Psi= \bigoplus_{N\geq1} \Psi(N)$ as an additive semigroup with respect to the direct sum operation on representations.

For a table $\gotM\in  \mathcal{T}(N)$ and an element $z\in Z_{\mathbf{G^\vee}}$, we define

\[
\psi_{z, \gotM} = \sum_{i\in I(\gotM)} z\otimes \nu_{a_i}\otimes \nu_{b_i} \in \Psi(N)\;.
\]

\begin{proposition}\label{prop:integralA}
\begin{enumerate}
    \item 
The assignment $(z,\gotM)\mapsto \psi_{z,\gotM}$ is a bijection between the set $Z_{\mathbf{G^\vee}}\times \mathcal{T}^{s_G}(N_G)$ and the set $\Psi_0(G)$ of integral $A$-parameters.
\item 
Restricting the previous assignment to the case of partitions, gives a bijection $(z,\lambda)\mapsto \psi_{z,\gotM_{\lambda}}$ between the set $Z_{\mathbf{G^\vee}}\times \mathcal{P}^{s_G}(N_G)$ and the subset of tempered $A$-parameters in $\Psi_0(G)$.

\item\label{it:intA3}
For any partition $\lambda\in  \mathcal{P}^{s_G}(N_G)$ and 
$z\in Z_{\mathbf{G^\vee}}$, the associated quasi-basic $A$-parameter may be described as
\[
\psi_{z, \mathcal{O}^\vee_{\lambda}} = (\psi_{z,\gotM_{\lambda}})^t\in \Psi_0(G)\;,
\]
while the associated tempered $L$-parameter is given by
\[
\phi_{(z,\ldots,z), \lambda} = \phi_{\psi_{z,\gotM_{\lambda}}}\in \Phi_u(G)\;,
\]
in the sense of the notation in \eqref{eq:l-expl}.

\item\label{it:intA4}
For any $\psi = \psi_{z,\gotM}\in \Psi_0(G)$, we have $\lambda(\phi_{\psi}) = \lambda_{\gotM}$.
\end{enumerate}
\end{proposition}

\begin{proof}
It is a direct consequence of standard properties of $A$-parameters (see, for example, \cite[Section 2.3]{atobe-crelle1}) that any $\psi\in \Psi(G)$ with infinitesimal character $\chi_{\psi}= \chi_{z,\mathcal{O}^\vee}$, for $z\in Z_{\mathbf{G^\vee}}$ and $\mathcal{O}^\vee\in \mathcal{U}^\vee$, must be of the form $\psi = \psi_{z,\gotM}$, for a table $\gotM\in \mathcal{T}^{s_G}(N_G)$.

For $\lambda\in \mathcal{P}^{s_G}(N_G)$, clearly $\psi_{z,\gotM_{\lambda}}$ is a tempered $A$-parameter that may be viewed as an $L$-parameter. By Proposition \ref{prop:lambdaphi} it then follows that $\chi_{\psi_{z,\gotM_{\lambda}}} = \chi_{z,\mathcal{O}^\vee_{\lambda}}$. Hence, $\psi_{z,\gotM_{\lambda}}\in \Psi_0(G)$.

For a given $\psi_{z,\gotM}\in \Psi(G)$, we observe the $L$-parameter $\phi':=\psi_{z,\gotM}\circ \Delta\in \Phi_u(G)$ that is defined through the diagonal map $\Delta: \mathrm{SL}_2(\mathbb{C}) \to \mathrm{SL}_2(\mathbb{C}) \times \mathrm{SL}_2(\mathbb{C})$. We have $\phi'(\frb) = z$, while $\chi_{\phi'} = \chi_{\psi_{z,\gotM}}$ holds. Thus, 
\[
\phi' = \phi_{(z,\ldots,z),\lambda'}=\phi_{\psi_{z,\gotM_{\lambda'}}}\;,
\]
for a partition $\lambda'\in  \mathcal{P}^{s_G}(N_G)$. In particular, $\chi_{\psi_{z,\gotM}} = \chi_{\psi_{z,\gotM_{\lambda'}}}$ and $\psi_{z,\gotM}\in \Psi_0(G)$.


The rest of the statement amounts to simple fact collecting, assisted by the formula in \eqref{eq:phipsi} and Proposition \ref{prop:lambdaphi}.

\end{proof}

\subsubsection{Component groups}

Let us write $\mathcal{T}^{mf}\subset \mathcal{T}$ for the set of multiplicity-free tables, that is, $\gotM\in \mathcal{T}$ with $(a_i,b_i)\neq (a_j,b_j)$, for all distinct $i,j\in I(\gotM)$.
For a table $\gotM\in \mathcal{T}$, a unique decomposition $\gotM = \gotM^{mf} \cup \gotM^{m} \cup \gotM^{m}$ with $\gotM^{mf}\in \mathcal{T}^{mf}$ and $\gotM^{m}\in \mathcal{T}$ exists.
Clearly, for a table $\gotM\in \mathcal{T}$, we have $\gotM\in \mathcal{T}^{\pm1}$, if and only if, $\gotM^{mf}\in \mathcal{T}^{\pm1}_0$.

Let us now fix a sign $s\in \{\pm1\}$ and a table $\gotM\in \mathcal{T}^s$.
There are unique tables $\gotM^{gp}\in \mathcal{T}_{0}^{s}$ and $\gotM^{bp}\in \mathcal{T}_{0}^{-s}$, so that $\gotM = \gotM^{gp} \cup \gotM^{bp} \cup \gotM^{bp}$ holds.
In particular, $\gotM^{mf} = (\gotM^{gp})^{mf}$ holds.
Let us denote the set 
\[
S(\gotM)= \{(a,b)\in \mathbb{Z}_{>0}\times \mathbb{Z}_{>0}\;:\; \exists i\in I(\gotM^{gp}),\, (a,b) = (a_i,b_i)\}\;.
\]
A natural map 
\[
\kappa: S(\gotM) \to \supp(\lambda_{\gotM})
\]
is given by $\kappa(a,b)= a$.
Note that the partition $\lambda_{\gotM}\in \mathcal{P}^s$ may not belong to $\mathcal{P}^s_0$, even when $\gotM\in \mathcal{T}^s_0$. For example, $\gotM = \gotM^{gp} = \{(3,3), (4,2)\}\in \mathcal{T}^1_0$, while $\lambda_{\gotM} = (3^34^2)\in \mathcal{P}^1$ has $(\lambda_{\gotM})^{gp} = (3^3)$.
Thus, we further refine the construction by writing 
\begin{equation}\label{eq:flat}
S(\gotM) = S(\gotM)^{\sharp} \sqcup S(\gotM)^{\flat}\;,
\end{equation}
where $S(\gotM)^{\sharp} = \kappa^{-1}(S(\lambda))$, and attain a surjective map $\kappa: S(\gotM)^{\sharp} \twoheadrightarrow S(\lambda_{\gotM})$ by restriction.

Now, in further similarity with Section \ref{sec:partitions} we consider the power set 
\[
P(\gotM) = \{A\subset S(\gotM)\}
\]
as a $\mathbb{F}_2$-vector space, a group of cardinality $2^{|S(\gotM)|}$, or as a space of boolean functions on $S(\gotM)$.
The decomposition in \eqref{eq:flat} naturally gives a direct sum decomposition
\[
P(\gotM) = P(\gotM)^{\sharp}\oplus P(\gotM)^{\flat}
\]
of vector spaces.
A pullback through $\kappa$ gives an embedding 
\[
\kappa^\ast: P(\lambda_\gotM) \hookrightarrow P(\gotM)^{\sharp} < P(\gotM)\;.
\]
Explicitly, for a set $A\subset S(\lambda_\gotM)$, we have $\kappa^\ast(A) = \{(a_i,b_i)\,:\, i\in I(\gotM^{gp})\,, a_i\in A \}\;. $
We also take note of the subgroup
\[
P(\gotM)_0 = \{A\subset S(\gotM)\;:\; |A\cap S(\gotM^{mf})|\mbox{ is even}\}< P(\gotM)\;.
\]
It can be verified that the map $\kappa^\ast$ embeds the subgroup $P(\lambda_{\gotM})_0$ into $P(\gotM)_0$.

\begin{proposition}\label{prop:expl-a-l}
(e.g. \cite[Section 2]{xu17})

For any $A$-parameter $\psi = \psi_{z,\gotM}\in \Psi_0(G)$, given by $z\in Z_{\mathbf{G^\vee}}$ and a table $\gotM\in  \mathcal{T}^{s_G}(N_G)$, there are natural identifications
\[
\widehat{\mathcal{S}_{\phi_{\psi}}}\cong P(\lambda_{\gotM})_0\,,\quad \widehat{\mathcal{S}_{\psi}}\cong P(\gotM)_0\;,
\]
under which the embedding given by $\kappa^\ast$ corresponds to the containment $\widehat{\mathcal{S}_{\phi_{\psi}}}< \widehat{\mathcal{S}_{\psi}}$.

\end{proposition}

We note that for the case of a tempered $A$-parameter $\psi=\psi_{z,\gotM_{\lambda}} \in \Psi_0(G)$ and its associated quasi-basic $A$-parameter $\psi_{z,\mathcal{O}^\vee_{\lambda}} = \psi^t\in \Psi_0(G)$, the map $\kappa^\ast$ is an isomorphism. The identifications of Proposition \ref{prop:comp-comb}, equation \eqref{eq:311} and Proposition \ref{prop:expl-a-l} all coincide in this case, to produce an identification
\begin{equation}\label{eq:tempered-ident}
\widehat{\mathcal{S}_{\phi_{(z,\ldots,z), \lambda}}}=\widehat{\mathcal{S}_{\psi_{z,\gotM_{\lambda}}}}=\widehat{\mathcal{S}_{\psi_{z,\mathcal{O}^\vee_{\lambda}}}}\;\cong\; P(\lambda)_0= P(\gotM_{\lambda})_0
\end{equation}
of groups.

\subsection{Moeglin's parameters and $A$-packet intersections}\label{sect:moeg-param}

Beyond the case of tempered $A$-parameters, where the associated $A$-packets are well-situated with respect to the Langlands reciprocity (in particular, are disjoint), the composition of general $A$-packets and their possible intersections remains a largely difficult question.
A line of study by Moeglin \cite{moeg06,moeg09-2,moeg11}, Xu \cite{xu17,xu21} and Atobe \cite{atobe-crelle1,atobe-crelle2} offered meaningful tools to approach such issues. We will now sketch some of its features, when applied on the case of integral $A$-packets.


For a table $\gotM\in \mathcal{T}$, we set the integers
\[
\alpha_i: = a_i+b_i\,,\; \beta_i = a_i-b_i\;,
\]
for all $i\in I(\gotM)$.
For a sign $s\in \{\pm1\}$ and a table $\gotM\in \mathcal{T}^s$, we associate the \textit{Moeglin parameter set}
\[
\mathcal{W}(\gotM) = \left\{(\underline{l}, \underline{\eta}) \;
: \; \begin{array}{l}\underline{l}: I(\gotM^{gp}) \to \mathbb{Z},\mbox{ s.t. } 0\leq \underline{l}(i)\leq b_i/2 \;,\\
\underline{\eta}: R_{\underline{l}} \to \{\pm1\}
\;,\\
\mbox{where }R_{\underline{l}} := I(\gotM^{gp}) \setminus \{i\;:\; \underline{l}(i)= b_i/2\}
\end{array}\right\}\;.
\]

An \textit{admissible order} $\prec$ on a table $\gotM\in \mathcal{T}^s$ is any linear order on the indexing set 
\[
I(\gotM^{gp})= \{x_1 \prec \ldots \prec x_t\}\;,
\]
for which the condition
\[
\beta_{x_{i}} \leq  \beta_{x_{j}},\quad\mbox{ or }\quad \left\{\begin{array}{l} \beta_{x_{1}} \geq0 \\ \beta_{x_{i}} >  \beta_{x_{j}}\geq0 \\ \alpha_{x_{i}} \leq  \alpha_{x_{j}}\end{array}\right.
\]
is satisfied, for all $1\leq i< j\leq t$.
For ease of presentation, when an admissible order $\prec$ is fixed on $\gotM\in \mathcal{T}^s$, we will often simply assume that $I(\gotM^{gp})= \{1,\ldots,t\}$ is indexed by integers number, so that $i< j \Leftrightarrow i\prec j$.

For each integral $A$-parameter $\psi= \psi_{z,\gotM}\in \Psi_0(G)$, with $\gotM\in \mathcal{T}^{s_G}(N_G)$ and $z\in Z_{\mathbf{G^\vee}}$, and a choice of an admissible order $\prec$ on $\gotM$, Moeglin's construction gives an embedding 
\[
M_{\prec}: \Pi^A_{\psi} \hookrightarrow \mathcal{W}(\gotM) \;.
\]
For a representation $\pi\in \Pi^A_{\psi}$, we will write $\pi = \pi_{\prec}(\psi,\underline{l},\underline{\eta})$, when $(\underline{l},\underline{\eta})= M_{\prec}(\pi)$.
As visible from its definition, the invariant $(\underline{l},\underline{\eta})\in \mathcal{W}(\gotM)$, for a given representation $\pi\in \Pi^A_{\psi}$, is not canonical, in the sense that it may vary along different choices of admissible orders.
Yet, for a representation $\pi\in \Pi^A_{\psi_{z,\gotM}}$, the canonical Arthur character $\epsilon^{\psi_{z,\gotM}}_{\pi}$ can still be read off Moeglin's parameterization in the following sense.

\begin{proposition}\label{prop:moegarth-formula}
(\cite[Theorem 3.6]{atobe-crelle1} combined with erratum in \cite[Appendix A]{atobe-crelle2})
Let $\prec$ be an admissible order on a table $\gotM\in \mathcal{T}^{s_G}(N_G)$ and $\psi = \psi_{z,\gotM}\in \Psi_0(G)$ an associated integral $A$-parameter. 
Then, a function $\gamma^{\psi}_{\prec}: I(\gotM^{gp}) \to \{\pm1\}$ exists (explicit formula in \cite[Definition 5.2]{xu17}), so that
\[
(-1)^{\epsilon^{\psi}_{\pi}(a_i, b_i)} =  \left\{ \begin{array}{ll}
\gamma^{\psi}_{\prec}(i)(-1)^{ \lfloor b_i/2\rfloor + \underline{l}(i) }  \underline{\eta}(i)^{b_i} & \underline{l}(i) < b_i/2 \\
\gamma^{\psi}_{\prec}(i) & \underline{l}(i) = b_i/2 
\end{array} \right.\;,
\]
holds for all $i\in I(\gotM^{gp})$, and any 
\[
\pi=\pi_{\prec}(\psi, \underline{l},\underline{\eta})\in \Pi^A_{\psi}\;.
\]
Here, we identify $\epsilon^{\psi}_{\pi} \in \widehat{\mathcal{S}_{\psi}} \cong P(\gotM)_0$ as a boolean function on $S(\gotM)$. 


\end{proposition}

For a table $\gotM\in \mathcal{T}^{s}$, it will be useful to set the convention 
\[
\widetilde{I(\m^{gp})} = \left\{ \begin{array}{ll}
  I(\m^{gp})   & s=1 \\
    I(\m^{gp})\sqcup \{0\}  & s=-1
\end{array}  \right.\;,
\] 
while assuming $(a_0,b_0) = (0,1)$, when $s=-1$.
Similarly, for a parameter $(\underline{l},\underline{\eta})\in \mathcal{W}(\gotM)$ we assume the convention $\underline{l}(0) = 0$ and $\underline{\eta}(0) = 1$.

\subsubsection{Intersections between $A$-packets}
In \cite{atobe-crelle1}, Atobe gave an algorithmic description of all possible pairs of tuples $(\psi_1, \prec_1, \underline{l}_1,\underline{\eta}_1)$, $(\psi_2, \prec_2, \underline{l}_2,\underline{\eta}_2)$ that give isomorphic representations
\[
\pi_{\prec_1}(\psi_1, \underline{l}_1,\underline{\eta}_1)\cong \pi_{\prec_2}(\psi_2, \underline{l}_2,\underline{\eta}_2)\in \Pi^{A}_{\psi_1}\cap \Pi^{A}_{\psi_2}\;.
\]

Let us recall parts of this theory that will be essential to our discussion.

\begin{proposition}\cite[Theorem 5.2 and Corollary 5.3]{atobe-crelle1}\label{prop:ui}
Let $\prec$ be an admissible order on a table $\gotM\in \mathcal{T}^{s_G}(N_G)$, for which we assume $I(\gotM^{gp}) = \{1\prec \ldots \prec t\}$.
Let $\psi = \psi_{z,\gotM}\in \Psi_0(G)$ be the associated integral $A$-parameter, for a choice of $z\in \ZZ$.
Let
\[
\pi = \pi_{\prec}(\psi, \underline{l},\underline{\eta})\in \Pi^{A}_{\psi}
\]
be a given representation.

Suppose that $t>k \in \widetilde{I(\gotM^{gp})}$ is such that $\beta_k\geq -1$ holds, and that one of the following conditions is satisfied:
\begin{enumerate}
    \item\label{case1}
\[
\left\{ \begin{array}{lll}   
a_{k+1} - a_k & = &(b_k-2l_k) -  (b_{k+1}-2l_{k+1}) \\
l_{k+1} - l_k & > &0 \\
l_{k+1} - l_k & > &b_{k+1}-b_k \\
\eta_{k+1}  & =&  (-1)^{b_k+1} \eta_k
\end{array}  \right.\;,
\]
\item\label{case2}
\[
\left\{ \begin{array}{lll}   
a_{k+1} - a_k & = &-(b_k-2l_k) +  (b_{k+1}-2l_{k+1}) \\
l_{k+1} - l_k & < &0 \\
l_{k+1} - l_k & < &b_{k+1}-b_k \\
\eta_{k+1} &=&  (-1)^{b_k+1} \eta_k 
\end{array}  \right.\;,
\]
\item\label{case3}
\[
\left\{ \begin{array}{lll}   
a_{k+1} - a_k & = &(b_k-2l_k) +  (b_{k+1}-2l_{k+1}) \\
l_{k+1} + l_k & < &b_k \\
l_{k+1} + l_k & < &b_{k+1} \\
\eta_{k+1} &=&  (-1)^{b_k} \eta_k 
\end{array}  \right.\;.
\]
\end{enumerate}

Here, we shortcut notation to $l_i = \underline{l}(i)$ and  $\eta_i = \underline{\eta}(i)$.
Then, we have an inclusion
\[
\pi \in \Pi^{A}_{\psi_{z,\gotM'}}\;,
\]
where $\gotM'\in \mathcal{T}^{s_G}(N_G)$ is the table given by
\[
\gotM' = \{(a_i,b_i)\}_{i\in I(\gotM)\setminus \{k,k+1\}} \cup \{ (a_{\ast},b_{\ast}), (a_{\ast\ast},b_{\ast\ast})\}\;,
\]
if added values are non-zero, where
\[
a_{\ast} = \frac{a_k+a_{k+1}}2  + \frac{b_{k+1}-b_k}2, \qquad b_{\ast} = \frac{b_k+b_{k+1}}2  + \frac{a_{k+1}-a_k}2\;,
\]
\[
a_{\ast\ast} = \frac{a_k+a_{k+1}}2  - \frac{b_{k+1}-b_k}2, \qquad b_{\ast\ast} = \frac{b_k+b_{k+1}}2  - \frac{a_{k+1}-a_k}2\;.
\]

\end{proposition}

\begin{corollary}\label{prop:ui-cor}

Let $\gotM, z, \psi, \prec$ be as in Proposition \ref{prop:ui}, and suppose that $t>k \in \widetilde{I(\gotM^{gp})}$ is such that $\beta_k\geq -1$ and $\alpha_k = \beta_{k+1}$ hold.
Suppose that $(\underline{l},\underline{\eta})\in \mathcal{W}(\gotM)$ is a Moeglin parameter that satisfies
\[
\left\{ \begin{array}{lll}   
l_{k+1} = l_k & = &0 \\
\eta_{k+1} &=&  (-1)^{b_k} \eta_k 
\end{array}  \right.\;.
\]
Let
\[
\gotM' = \{(a_i,b_i)\}_{i\in I(\gotM)\setminus \{k,k+1\}} \cup \{ (a_\ast,b_\ast)\}\in \mathcal{T}^{s_G}(N_G)\;,
\]
where $a_\ast = \frac12(\beta_k + \alpha_{k+1})$, $b_\ast =  b_k+ b_{k+1}$, be a table, constructed out of $\gotM$ and $k$.
Let $(\underline{l}',\underline{\eta}')\in \mathcal{W}(\gotM')$ be the Moeglin parameter constructed out of $(\underline{l},\underline{\eta})$ by setting
\[
\underline{l}'(\ast) = 0,\quad \underline{\eta}'(\ast)= \eta_k\;,
\]
while retaining equalities $\underline{l}'=\underline{l}, \underline{\eta}' = \underline{\eta}$ on $I(\gotM^{gp})\setminus \{k,k+1\}$.
Let $\prec'$ be the admissible order on $\gotM'$ that is constructed out of $\prec$ by replacing the relative position of $\{k,  k+1\}$ with the index $\ast$.

Then, the parameter $(\underline{l},\underline{\eta})$ is in the image of $M_{\prec}$, for the $A$-parameter $\psi_{z,\gotM}$, if and only if, the parameter $(\underline{l}',\underline{\eta}')$ is in the image of $M_{\prec'}$, for the $A$-parameter $\psi_{z,\gotM'}$.
In case the equivalent conditions hold, we have
\[
\pi_{\prec}(\psi_{z,\gotM}, \underline{l}, \underline{\eta}) \cong \pi_{\prec'}(\psi_{z,\gotM'}, \underline{l}', \underline{\eta}')\in \Pi^{A}_{\psi_{z,\gotM}}\cap \Pi^{A}_{\psi_{z,\gotM'}}\;.
\]

\end{corollary}

\begin{proof}
Assuming that $\pi_{\prec}(\psi_{z,\gotM}, \underline{l}, \underline{\eta})$ is well-defined, this is a particular case of Proposition \ref{prop:ui}\eqref{case3}, where $\prec', \underline{l}', \eta'$ are stated in \cite[Theorem 5.2]{atobe-crelle1}.

\end{proof}

We say that a table $\gotM\in \mathcal{T}^s$ is \textit{non-negative}, if for all $i\in I(\gotM^{gp})$, we have $\beta_i\geq-1$.

\begin{remark}
    Our definition for non-negativity of $A$-parameters notably differs from that of \cite{atobe-crelle1} and other sources, where the condition $\beta_i\geq0$ is taken.
\end{remark}

\begin{proposition}\label{prop:phantom}
Let $\prec$ be an admissible order on a table $\gotM\in \mathcal{T}^{s_G}(N_G)$. We assume $1\in I(\gotM^{gp})$ to be the minimal index with respect to $\prec$.
Let $\psi = \psi_{z,\gotM}\in \Psi_0(G)$ be the associated integral $A$-parameter. 
Let
\[
\pi = \pi_{\prec}(\psi, \underline{l},\underline{\eta})\in \Pi^{A}_{\psi}
\]
be a given representation.
Suppose that we have $\beta_1 \in \{-1,0,1\}$, $d:=\min \{ a_1, b_1\}>1$ and that the conditions
\[
\left\{  \begin{array}{ll}
 \underline{l}(1) = 0 & \mbox{ if } \beta_1=0  \\
 \underline{l}(1) = 0,\; \underline{\eta}(1) = -1 & \mbox{ if } \beta_1 =1 \\
\underline{l}(1) = 1,\; \underline{\eta}(1) = -1  & \mbox{ if } \beta_1= -1 
\end{array}
\right.
\]
hold.

Then, for any $1\leq c < d$, we have an inclusion $\pi \in \Pi^{A}_{\psi_{z,\gotM_c}}$, where $\gotM_c\in \mathcal{T}^{s_G}(N_G)$ is the (not non-negative) table defined as
\[
\gotM_c = \left\{  
\begin{array}{ll}
\{(a_i,b_i)\}_{i\in I(\gotM)\setminus \{1\}} \cup \{ (c,c), (d-c,d+c)\} & \mbox{ if }\beta_1=0 \\
\{(a_i,b_i)\}_{i\in I(\gotM)\setminus \{1\}} \cup \{ (c+1,c), (d-c,d+c+1)\} & \mbox{ if }\beta_1 = 1 \\
\{(a_i,b_i)\}_{i\in I(\gotM)\setminus \{1\}} \cup \{ (c,c+1), (d-c,d+c+1)\} & \mbox{ if } \beta_1 = -1 
\end{array}
\right.\;.
\]

\end{proposition}
\begin{proof}
    This is an explication of cases of 'operation (P)' in \cite[Theorem 3.5]{atobe-crelle2}.
\end{proof}

The following proposition fully characterizes the possible intersections of an $A$-packet arising from $\psi_{z,\gotM}$ for a non-negative table $\gotM\in \mathcal{T}^{s_G}(N_G)$ with any other $A$-packet.

\begin{proposition}\label{prop:ui-exhaust}
\cite[cases of Theorem 3.5]{atobe-crelle2}
    
Suppose that a representation $\pi \in \Pi_{\psi} \cap \Pi_{\widehat{\psi}}$ belongs to an intersection of $A$-packets, for two distinct $A$-parameters $\psi,\widehat{\psi}\in \Psi_0(G)$.
We write $\psi= \psi_{z,\gotM}$ and $\widehat{\psi}= \psi_{z,\widehat{\gotM}}$ for tables $\gotM, \widehat{\gotM}\in \mathcal{T}^{s_G}(N_G)$, and $z\in \ZZ$.

\begin{enumerate}
\item\label{it:ui-exhaust1}
Suppose that $\gotM$ and $\widehat{\gotM}$ are both non-negative.
Then, a sequence $\gotM = \gotM_1, \gotM_2, \ldots, \gotM_r = \widehat{\gotM}$ of non-negative tables in $\mathcal{T}^{s_G}(N_G)$ exists, so that inclusions
\[
\pi\in \Pi^{A}_{\psi_{z,\gotM_i}}\,,\quad i=1,\ldots,r
\]
hold, and so that for each $1\leq i < r$ we can write $\{\gotM_i,\gotM_{i+1}\} = \{\overline{\gotM}_i,\overline{\gotM}'_i\}$ and find an admissible order $\prec_i$ on $\overline{\gotM}_i$, under which the inclusion $\pi\in \Pi^{A}_{\psi_{z,\overline{\gotM}'_i}}$ is obtained from the parameter $M_{\prec_i}(\pi)$ by one of the moves described in Proposition \ref{prop:ui}.

\item\label{it:ui-exhaust2}
Suppose that $\gotM$ is non-negative, while $\widehat{\gotM}$ is not non-negative.
Then, there exist a non-negative table $\gotM'\in  \mathcal{T}^{s_G}(N_G)$ with $\pi \in \Pi_{\psi_{z,\gotM'}}$, an admissible order $\prec'$ on $\gotM'$, and a (not non-negative) table $\widehat{\gotM}'\in  \mathcal{T}^{s_G}(N_G)$ which is obtained from $M_{\prec'}(\pi)$ by a move described in Proposition \ref{prop:phantom}.

\end{enumerate}
\end{proposition}

\subsection{Near-tempered $A$-parameters}\label{sect:near-temp}

We say that a table $\gotM\in \mathcal{T}^{s}$ is \textit{near-tempered}, if $b_i\in \{1,2\}$, for all $i\in I(\gotM^{gp})$.
Note, that for a near-tempered table $\gotM$, $b_i$ is determined by $a_i$, for all $i\in I(\gotM^{gp})$, due to the parity condition. Near-tempered tables are non-negative.
We say that an integral $A$-parameter $\psi_{z,\gotM}\in \Psi_0(G)$ is \textit{near-tempered}, when the table $\gotM$ is near-tempered.

\subsubsection{Special pieces in Arthur theory}

Our interest in near-tempered $A$-parameters arises naturally through the proof process of Proposition \ref{prop:intro}.
Given a partition $\lambda\in \mathcal{P}^{s_G}(N_G)$ and a subset $J\subset \mathbb{J}(\lambda)$, we define the near-tempered table
\[
\gotM_{\lambda,J} = i\left(\lambda\setminus  \cup_{c\in J} (c-1\;c+1)\right) \cup \{(c,2)\}_{c\in J}\in \mathcal{T}^{s_G}(N_G)\;.
\]
It is evident that 
\begin{equation}\label{eq:ntpformula}
p(\gotM_{\lambda,J}) = T_J(\lambda)
\end{equation}
holds.

Now, suppose that partitions $\lambda,\mu\in \mathcal{P}^{s_G}(N_G)$ are given so that $\mu\in \spc (\lambda)$. In other words, there is a unique subset $J\subset \mathbb{J}(\lambda)$ with $\mu = T_J(\lambda)$.
For $z\in Z_{\mathbf{G^\vee}}$, we define the $L$-parameter
\[
\phi_{z,\lambda,\mu}: = \phi_{\psi_{z,\gotM_{\lambda,J}}}\in \Phi_u(G)\;.
\]
The following lemma now gives the existence part of Proposition \ref{prop:intro}.

\begin{lemma}\label{lem:half-intro}
Let $\phi =\phi_{z, \lambda,\mu}\in \Phi_u(G)$ be the $L$-parameter defined by partitions $\lambda,\mu\in \mathcal{P}^{s_G}(N_G)$ with $\mu\in \spc (\lambda)$, and $z\in \ZZ$.

Then, both identities
\[
\mathcal{O}^\vee_{\phi} = \mathcal{O}^\vee_{\mu}\in \mathcal{U}^\vee\,,\quad \chi_{\phi} = \chi_{z, \mathcal{O}^\vee_{\lambda}}\in \Lambda_u(G)
\]
hold.

\end{lemma}
\begin{proof}
The equality in $\mathcal{U}^\vee$ follows from Proposition \ref{prop:integralA}\eqref{it:intA4} and \eqref{eq:ntpformula}.

Setting $\lambda' = \lambda\setminus  \cup_{c\in J} (c-1\;c+1)$, we see from formula \eqref{eq:phipsi} that
\[
\phi = \phi_{(z,\ldots,z),\lambda'} + \sum_{c\in J} zq^{1/2}\otimes \nu_c + zq^{-1/2}\otimes \nu_c\;.
\]
On the other hand, straightforward identities of infinitesimal characters show that
\[
\chi_{z, \mathcal{O}^\vee_{\lambda}}= \chi_{\psi_{z, \mathcal{O}^\vee_{\lambda}}} = \chi_{\psi_{z, \gotM_\lambda}} = \chi_{\phi_{(z,\ldots,z),\lambda}}\;.
\]
Since $\phi_{(z,\ldots,z),\lambda} = \phi_{(z,\ldots,z),\lambda'} + \sum_{c\in J} z\otimes \nu_{c-1} + z\otimes \nu_{c+1}$ (taking $\nu_0$ as an $L$-parameter neutral to addition, if necessary), it suffices to show that $zq^{1/2}\otimes \nu_c + zq^{-1/2}\otimes \nu_c$ and  $z\otimes \nu_{c-1} + z\otimes \nu_{c+1}$ have equal infinitesimal characters, for $c\in J$.

Indeed, a direct computation through Remark \ref{rem:inf-ch-comp} shows that, as semisimple conjugacy classes, both elements of $\Lambda_u(G)$ consist of a matrices whose multisets of eigenvalues are described as 
\[
\{ zq^{\frac{c-2}2}, zq^{\frac{c-4}2},\ldots, zq^{-\frac{c-2}2}\}\;,
\]
all with multiplicity $2$, added with $\{zq^{\frac{c}2}, zq^{-\frac{c}2} \}$ of multiplicity $1$.

\end{proof}

\begin{theorem}\label{thm:almost-intro}
Let $\lambda\in \mathcal{P}^{s_G}(N_G)$ be a partition.
Suppose that $\phi\in \Phi(G)$ is an $L$-parameter satisfying $\chi_{\phi} = \chi_{z,\mathcal{O}^\vee_{\lambda}}$, for $z\in \ZZ$, and $d(\lambda(\phi)) = d(\lambda)$.
Then, a partition $\mu\in \spc(\lambda)$ exists, so that $\phi = \phi_{z,\lambda,\mu}$.
    
\end{theorem}

\begin{proof}

We denote the partition $\mu = \lambda(\phi)\in \mathcal{P}^{s_G}(N_G)$. 
By Proposition \ref{prop:d-prop}, $\lambda,\mu$ must share the same special piece. In other words, both partitions share the same special partition $\nu:= T^{\mathbb{I}(\lambda)}(\lambda) =T^{\mathbb{I}(\mu)}(\mu)$. 
In particular, $\mu = T_{\mathbb{I}(\mu)}( T^{\mathbb{I}(\lambda)}(\lambda))$. Setting 
\[
I = \mathbb{I}(\lambda)\setminus \mathbb{I}(\mu)\,,\quad J = \mathbb{I}(\mu)\setminus \mathbb{I}(\lambda)\subset \mathbb{J}(\lambda)\;,
\]
we may shortcut the equation to the form $\mu = T_{J}( T^{I}(\lambda))$, so that $I$ and $J$ are disjoint sets.
Let us denote the largest integer $a_{\max}\in \supp(\lambda)$, and record the multiplicity $k:= m(a_{\max}, \lambda)>0$.
We will now prove by induction on the parameter $a_{\max}$, that $I$ is empty and that $\phi = \phi_{\psi_{z,\gotM_{\lambda,J}}}$. 

Let us consider the infinitesimal character $S:= \chi_{z,\mathcal{O}^\vee_{\lambda}}= \chi_{\phi_{(z,\ldots,z),\lambda}} =  \chi_{\phi}\in \Lambda_u(G)$ as a semisimple matrix. 
An observation that follows from Remark \ref{rem:inf-ch-comp} is that for all eigenvalues $\alpha$ of $S$, $2 \log_q(z^{-1}\alpha)$ are integers with absolute value $\leq a_{\max}-1$.
Moreover, both $zq^{\pm \frac{a_{\max}-1}2}$ appear as eigenvalues of $S$ with multiplicity $k$.
We note that $a_{\max}\not\in I$ always holds. Indeed, otherwise, we should have $m(a_{\max}+1, \mu)>0$, 
which would force a summand of the form $s\otimes \nu_{a_{\max}+1}$ to appear in the $L$-parameter $\phi$, according to Proposition \ref{prop:lambdaphi}. Thus, both $sq^{\frac{a_{\max}}2}$ and $sq^{-\frac{a_{\max}}2}$ would appear as eigenvalues of $S$, contradicting the previously observed bounds.
We note that same bounds on $S$-eigenvalues force any summand of the form $s\otimes \nu_{a_{\max}}$ in $\phi$ to satisfy $s=z$. 
Thus, an occurrence $a_{\max}-1\in I$ would have implied that $m(a_{\max},\mu)= k+1$, and hence, that $z\otimes \nu_{a_{\max}}$ appears $k+1$ times as a summand of $\phi$. That would have caused a contradiction to the multiplicity of $zq^{\pm \frac{a_{\max}-1}2}$ as an eigenvalue of $S$. 
Hence, $a_{\max}-1\not\in I$.

Suppose now that $a_{\max}-1\not\in J$ (and $a_{\max} >2$, since otherwise we are done). 
In this case, $m(a_{\max}, \mu) = k$. Arguing as before, we see that $z\otimes \nu_{a_{\max}}$ must appear $k$ times in $\phi$.
Let us write
\[
\phi = \phi_{(z,\ldots,z),\mu_0} + \phi_{(\zeta_1,\ldots,\zeta_m), \mu_1} + \phi_{(\zeta_1^{-1},\ldots,\zeta_m^{-1}), \mu_1}
\]
in the form of Proposition \ref{prop:lambdaphi}.
We construct the partition
\[
\lambda' =
\left\{ \begin{array}{ll}  \lambda \setminus (a_{\max}^k)  & k\mbox{ is even}  \\  \lambda \setminus (a_{\max}^k) \cup (a_{\max}-2) & k\mbox{ is odd} \end{array}\right.\;,
\]
and similarly, $\mu'$ out of $\mu$.
By construction, $\lambda', \mu'\in \mathcal{P}^{s_G}(N_{G'})$, for a smaller rank group $G'$ of same type as $G$.
We may write
\[
\phi':= \phi_{(z,\ldots,z),\mu'_0} + \phi_{(\zeta_1,\ldots,\zeta_{m'}), \mu'_1} + \phi_{(\zeta_1^{-1},\ldots,\zeta_{m'}^{-1}), \mu'_1}\in \Phi_u(G'),
\]
for a well-defined $L$-parameter, with $\chi_{\phi'} = \chi_{z,\mathcal{O}^\vee_{\lambda'}}$, $\lambda(\phi')= \mu'$, $m'\in \{m,m-1\}$, and a decomposition $\mu' = \mu'_0 \cup \mu'_1 \cup\mu'_1$, so that $\mu'_0\in \mathcal{P}^{s_G}$.
Note, that $ \mathbb{I}(\lambda)\subset \mathbb{I}(\lambda')\cup \{a_{\max}-1,a_{\max}\}$ and $ \mathbb{J}(\lambda)\subset \mathbb{J}(\lambda')\cup \{a_{\max}-1\}$ hold. Hence, $I\subset \mathbb{I}(\lambda')$, $J\subset \mathbb{J}(\lambda')$ and it is easy to verify that $\mu' = T_J(T^I(\lambda'))$ holds.
Since Proposition \ref{prop:d-prop} implies $d(\mu') = d(\lambda')$, by the induction hypothesis we know that $I$ is empty and that $\phi' = \phi_{\psi_{z,\gotM_{\lambda',J}}}$. The claim now clearly follows.

We are left with the case of $a_{\max}-1\in J$. 
Let us write $l = m(a_{\max}-1, \lambda)$. Again, we see that both $zq^{\pm \frac{a_{\max}-2}2}$ appear as eigenvalues of $S$ with multiplicity $l$. Yet, from the assumption on $J$, we have $m(a_{\max}-1,\mu) = l+2$, implying the existence of two summands in $\phi$ of the form $s_i\otimes \nu_{a_{\max}-1}$, with $s_i\neq z$, for $i=1,2$. Once again by $S$-eigenvalue bounds, we are forced into the situation of $\{s_1,s_2\}= \{zq^{\frac12},zq^{-\frac12}\}= \{\zeta_m, \zeta_m^{-1}\}$.
We also have $m(a_{\max},\mu)=k-1$. When combined with the previously noted bounds on $S$-eigenvalues, that equality implies that $z\otimes \nu_{a_{\max}}$ appears $k-1$ times in $\phi$. 
The assumption on $J$ also shows that $m(a_{\max}-2,\lambda)>0$, or that $a_{\max}=2$ (in which case, $k$ is even).

We define an $L$-parameter
\[
\phi' = \left\{ \begin{array}{rll} \phi_{(z,\ldots,z),\,\mu_0\setminus (a_{\max}^{k-1})}  & + \phi_{(\zeta_1,\ldots,\zeta_{m-1}), \, \mu_1 \setminus (a_{\max}-1)} + \phi_{(\zeta_1^{-1},\ldots,\zeta_{m-1}^{-1}),\, \mu_1\setminus (a_{\max}-1)} & k\mbox{ is odd}  \\ \phi_{(z,\ldots,z),\,\mu_0\setminus (a_{\max}^{k-1})\cup (a_{\max}-2)} & + \phi_{(\zeta_1,\ldots,\zeta_{m-1}),\, \mu_1 \setminus (a_{\max}-1)} + \phi_{(\zeta_1^{-1},\ldots,\zeta_{m-1}^{-1}),\, \mu_1\setminus (a_{\max}-1)}  & k\mbox{ is even} \\
\phi_{(z,\ldots,z),\,\mu_0\setminus (2^{k-1})} & + \phi_{(\zeta_1,\ldots,\zeta_{m-1}), \,\mu_1 \setminus (1)} + \phi_{(\zeta_1^{-1},\ldots,\zeta_{m-1}^{-1}), \, \mu_1\setminus (1)} & a_{\max}=2
\end{array}\right.
\]
in $\Phi_u(G')$, and 
\[
\lambda' = \left\{ \begin{array}{ll} \lambda\setminus (a_{\max}-2\; a_{\max}^k) & k\mbox{ is odd}  \\ \lambda\setminus (a_{\max}^k)  & k\mbox{ is even} \end{array}\right.\in \mathcal{P}^{s_G}(N_{G'})\;.
\]
It now follows that $\chi_{\phi'} = \chi_{z,\mathcal{O}^\vee_{\lambda'}}$, that $I\subset \mathbb{I}(\lambda')$, $J':=J\setminus \{a_{\max}-1\}\subset \mathbb{J}(\lambda')$. In particular, $\lambda(\phi') = T_{J'}(T^I(\lambda'))$.
By the induction hypothesis, $I$ is empty and we have $\phi' = \phi_{\psi_{z,\gotM_{\lambda',J'}}}$. Consequently, $\phi = \phi_{\psi_{z,\gotM_{\lambda,J}}}$.

\end{proof}

A particular consequence is Proposition \ref{prop:intro}, which now follows as a corollary of Theorem \ref{thm:almost-intro}, Lemma \ref{lem:half-intro} and the description of Proposition \ref{prop:spc-comb}.

\subsubsection{$A$-packets in the near-tempered case}

In what follows we exploit the near-tempered condition to reach a higher level of precision on the composition of $A$-packets that are attached to $A$-parameters within that class.

\begin{proposition}\label{prop:admss}

Let $\prec$ be any linear order on the set $I(\gotM^{gp})$, for a near-tempered table $\gotM\in \mathcal{T}^{s}$. 
We assume $I(\gotM^{gp}) = \{1\prec \ldots \prec t\}$.

\begin{enumerate}
    \item 
If the order $\prec$ is admissible, then:
\begin{itemize}
\item
For all $1\leq i < j\leq t$, we have $a_i \leq a_j+1$.
\item 
If $(a_1,b_1) = (2,1)$, then $1 < a_i$, for all $i\in I(\gotM^{gp})$.
\end{itemize}

\item 
If $a_i \leq a_j+1$ holds, for all $1\leq i < j\leq t$, and $(1,2)\not\in S(\gotM)$, then $\prec$ is an admissible order.

\item
If $a_i \leq a_j$ holds, for all $1\leq i < j\leq t$, then $\prec$ is an admissible order.

\end{enumerate}
\end{proposition}
\begin{proof}
\begin{enumerate}
    \item 
    For $1\leq i<j\leq t$, the near-tempered property implies $|b_j-b_i|\leq1$. Hence, inequalities
    \[
    \beta_j - \beta_i, \: \alpha_j - \alpha_i \leq  a_j - a_i +1 
    \]
    hold. Yet, one of $\beta_j - \beta_i,  \alpha_j - \alpha_i$ must be non-negative.
    Assuming $(a_1,b_1) = (2,1)$ and $(a_i,b_i)= (1,2)$ would result in $\beta_1 > \beta_i$ and $\beta_i <0$ which contradicts admissibility.
\item 
By assumption $\beta_i\geq 0$ holds, for all $1\leq i \leq t$. Hence, it suffices to verify that both $\beta_i > \beta_j$ and $\alpha_i > \alpha_j$ cannot be valid simultaneously, for indices $1\leq i < j\leq t$.
Indeed, the parity condition would have implied both $\alpha_i - \alpha_j$ and $\beta_i-\beta_j$ are no smaller than $2$. Summing the two inequalities would have resulted in $2a_i \geq 2a_j +4 $, which is a contradiction.

\item
An occurrence of $\beta_i>\beta_j$ for $1\leq i< j\leq t$ is impossible, since the parity condition would have implied $\beta_i-\beta_j\geq2$ and $a_i-a_j\geq1$.

\end{enumerate}

\end{proof}

For a near-tempered table $\gotM\in \mathcal{T}^{s}$, we say that a linear order $\prec$ on $I(\gotM^{gp})$ (or, on $\gotM$) is \textit{standard admissible}, when $a_i\leq a_j$ holds, for all $i\prec j$ in $I(\gotM^{gp})$.
Clearly, standard admissible orders always exist. By Proposition \ref{prop:admss}, a standard admissible order is admissible.

\begin{proposition}\label{prop:z-formula}

Let $\prec$ be an admissible order on the set $I(\gotM^{gp})$, for a near-tempered table $\gotM\in \mathcal{T}^{s_G}(N_G)$, for which we assume $I(\gotM^{gp}) = \{1\prec \ldots \prec t\}$.
Then, the function $\gamma^{\psi}_{\prec}$, in case of the $A$-packet $\psi= \psi_{z,\gotM}\in \Psi_0(G)$, that was considered in Proposition \ref{prop:moegarth-formula}, is given by the formula
\[
\gamma^{\psi}_{\prec}(i) = (-1)^{|Z_{i}|}\;,
\]
for all $i \in I(\gotM^{gp})$, where
\[
Z_{i} = \{ 1\leq j <i \;:\; a_j = a_i +1\} \cup \{ i<  j \leq t_{\rho} \;:\; a_j = a_i -1\}\;.
\]
In particular, when $\prec$ is standard admissible, the function $\gamma^{\psi}_{\prec}\equiv 1$ is constant.

\end{proposition}

\begin{proof}
    Follow from the formula in \cite[Definition A.1]{atobe-crelle2}, which reproduces \cite[Definition 5.2]{xu17}.
\end{proof}

\begin{lemma}\label{lem:lpacket}
Let $\gotM \in\mathcal{T}^{s_G}(N_G) $ be a near-tempered $A$-parameter, $\psi = \psi_{z,\gotM}\in \Psi_0(G)$ the associated $A$-packet for a choice of $z\in \ZZ$, and $\pi\in \Pi^{A}_{\psi}$ a representation.
Suppose that the character $\epsilon= \epsilon^{\psi}_{\pi}\in \widehat{\mathcal{S}_{\psi}}$ belongs to the subgroup $\widehat{\mathcal{S}_{\phi_{\psi}}}$.
Then, $\pi$ must be included in the associated $L$-packet $\Pi_{\phi_{\psi}}\subset \Pi^{A}_{\psi}$. More precisely, we have
\[
\pi = \pi(\phi_{\psi},\epsilon)\in \Pi_{\phi_{\psi}}\;.
\]  
\end{lemma}

\begin{proof}
We denote the representation $\pi' = \pi( \phi_{\psi},\epsilon)\in \Pi_{\phi_{\psi}}$.
We recall that according to the setup of Proposition \ref{prop:arth-book}, we have $\epsilon^{\psi}_{\pi'} = \epsilon$ in $\widehat{\mathcal{S}_{\psi}}$.
Let us fix a standard admissible order $\prec$ on $\gotM$ and write $\pi = \pi_{\prec}(\psi, \underline{l},\underline{\eta})$ and $\pi' = \pi_{\prec}(\psi, \underline{l'},\underline{\eta'})$ in terms of Moeglin's parameters. 
Using Proposition \ref{prop:expl-a-l}, we identify the inclusion $\widehat{\mathcal{S}_{\phi_{\psi}}}<\widehat{\mathcal{S}_{\psi}}$ with $P(\lambda_{\gotM})_0< P(\gotM)_0$. 
Let $i\in I(\gotM^{gp})$ be an index.
Suppose first that $b_i=2$. Since $\epsilon\in P(\lambda_{\gotM})< P(\gotM)^{\sharp}$ is assumed and $b_i$ is even, we must have $\epsilon(a_i,b_i)=0$. Now, since $\gamma^{\psi}_{\prec}$ is trivial (Proposition \ref{prop:z-formula}), the formula in Proposition \ref{prop:moegarth-formula} must imply $\underline{l}(i) = \underline{l'}(i) =1$.
Otherwise, we have $b_i=1$. We are clearly restricted to $\underline{l}(i) = \underline{l'}(i) =0$, while the formula of Proposition \ref{prop:moegarth-formula} forces 
$\underline{\eta}(i) = \underline{\eta'}(i)$.
Put together, we see that $\underline{l} = \underline{l'}$ and $\underline{\eta} = \underline{\eta'}$ hold. Thus, $\pi$ and $\pi'$ are isomorphic representations.

\end{proof}

Let us now explicate Proposition \ref{prop:ui} and Corollary \ref{prop:ui-cor} on cases of $A$-packet intersections, for the case of near-tempered $A$-parameters.

\begin{proposition}\label{prop:ui-nt}
Let $\prec$ be an admissible order on a near-tempered table $\gotM\in \mathcal{T}^{s_G}(N_G)$, for which we assume $I(\gotM^{gp}) = \{1\prec \ldots \prec t\}$.
Let $\psi = \psi_{z,\gotM}\in \Psi_0(G)$ be an associated integral $A$-parameter, for a choice of $z\in \ZZ$.
Let
\[
\pi = \pi_{\prec}(\psi, \underline{l},\underline{\eta})\in \Pi^{A}_{\psi}
\]
be a given representation.
Suppose that $t> k \in \widetilde{I(\gotM^{gp})}$ satisfies one of the following conditions:

\begin{enumerate}
    \item\label{it-nt}
\[
\left\{ \begin{array}{lll}   
a_{k+1} - a_k & = &2 \\
b_k = b_{k+1} & = &1 \\
\eta_{k+1} &=&  -\eta_k 
\end{array}  \right.\;,
\]

    \item \label{it-nt2}

\[
\left\{ \begin{array}{lll}   
a_{k+1} - a_k & = &2 \\
b_k= b_{k+1} & = &2 \\
l_{k+1} + l_{k} &=& 1 
\end{array}  \right.\;,
\]
    
    \item\label{it-nt3}

\[
\left\{ \begin{array}{lll}   
a_{k+1} - a_k & = &3 \\
l_{k+1} =l_{k} &=& 0 \\
\eta_{k+1} &=&  (-1)^{b_k}\eta_k 
\end{array}  \right.\;,
\]

    \item\label{it-nt4}

\[
\left\{ \begin{array}{lll}   
a_{k+1} - a_k & = &4 \\
b_{k} = b_{k+1}& = & 2\\
l_{k+1} =l_{k} &=& 0 \\
\eta_{k+1} &=&  \eta_k 
\end{array}  \right.\;.
\]

\end{enumerate}
Here, we shortcut notation to $l_i = \underline{l}(i)$ and  $\eta_i = \underline{\eta}(i)$.
Each such case gives rise to a table $\gotM'\in \mathcal{T}^{s_G}(N_G)$, for which $\pi\in \Pi^{A}_{\psi_{z,\gotM'}}$.

In case \eqref{it-nt}, we have
\[
\gotM' = \{(a_i,b_i)\}_{i\in I(\gotM)\setminus \{k,k+1\}} \cup \{ (a_\ast,b_\ast)\}\;,
\]
so that $(a_\ast,b_\ast)= (a_k +1, 2)$, and we may write
\[
\pi \cong \pi_{\prec'}(\psi_{z,\gotM'}, \underline{l}', \underline{\eta}')\;,
\]
where the order $\prec'$ on $\gotM'$ is constructed out of $\prec$ by replacing the relative position of $k \prec k+1$ with the index $\ast$, and setting
\[
\underline{l}'(\ast) = 0,\quad \underline{\eta}'(\ast)= \eta_k\;,
\]
while retaining equalities $\underline{l}'=\underline{l}, \underline{\eta}' = \underline{\eta}$ on $I(\gotM^{gp})\setminus \{k,k+1\}$.

\end{proposition}

\begin{proof}
It is straightforward to verify that the near-tempered restriction of 
\[
(b_i,l_i)\in \{(1,0),(2,0),(2,1)\}\;, \; \forall 0\leq i <t\;,
\]
forces all instances of cases \eqref{case1} and \eqref{case2} of Proposition \ref{prop:ui} to incarnate as case \eqref{it-nt2} of our current statement.

Case \eqref{case3} of Proposition \ref{prop:ui} is now explicated into the three remaining cases. In particular, case \eqref{it-nt} falls under the scope of Corollary \ref{prop:ui-cor}, which is explicated by $\gotM$.

\end{proof}

\begin{corollary}
Suppose that $\gotM'$ is a table obtained out of a table $\gotM\in \mathcal{T}^{s_G}(N_G)$ using one of the procedures described in Proposition \ref{prop:ui}.
If $\gotM$ is non-negative and $\gotM'$ is near-tempered, then $\gotM$ must be near-tempered and $\gotM'$ is obtained through an application of the case \eqref{it-nt} of Proposition \ref{prop:ui-nt} on $\gotM$.

\end{corollary}

\begin{proof}
It is evident from the description in Proposition \ref{prop:ui}, that for each $i\in I(\gotM^{gp})$, there must be $i'\in  I(\gotM'^{gp})$ with $b_i\leq b_{i'}$. Hence, $\gotM$ must be near-tempered.
Similarly, we that each of the cases \eqref{it-nt2}, \eqref{it-nt3}, \eqref{it-nt4} of Proposition \ref{prop:ui-nt} produces $i'\in I(\gotM'^{gp})$ with $2<b_{i'}$.

\end{proof}


We are set to explore the intersections of $A$-packets that arise from a tempered parameter with $A$-packets that arise from near-tempered $A$-parameters.
In the following lemma, characters $\epsilon^{\psi}_{\pi}\in \widehat{\mathcal{S}_{\psi_{z,\gotM_{\lambda}}}}$, for tempered representations $\pi\in \Pi^{A}_{\psi_{z,\gotM_{\lambda}}}$ defined by a partition $\lambda\in \mathcal{P}^{s_G}(N_G)$ and $z\in \ZZ$, are treated as boolean functions on $S(\lambda)$, according to the identification $\widehat{\mathcal{S}_{\psi_{z,\gotM}}}\cong P(\lambda)_0$ of \eqref{eq:tempered-ident}.

\begin{lemma}\label{lem:Js}

Let $\lambda\in \mathcal{P}^{s_G}(N_G)$ be a partition, and $z\in \ZZ$. Let $\psi = \psi_{z,\gotM_{\lambda}}\in \Psi_0(G)$ be the associated tempered integral $A$-parameter.
Let $J\subset \mathbb{J}(\lambda)$ be a subset, and $\psi'=\psi_{z,\gotM_{\lambda,J}}\in \Psi_0(G)$ be the associated near-tempered $A$-parameter.
The following properties hold.

\begin{enumerate}
    \item\label{it:Js1}
There is an equality
\[
\Pi^{A}_{\psi} \cap \Pi^{A}_{\psi'}  = \{\pi\in \Pi^{A}_{\psi}\;:\;  \mathfrak{t}_c(\epsilon^{\psi}_{\pi})\neq1,\,\forall c\in J\}\;,
\]
with $\mathfrak{t}_c$ defined as in \eqref{eq:tc}.

\item\label{it:Js2}
Let $\prec_0$ be a standard admissible order on $\gotM_{\lambda,J}$.
For any representation $\pi = \pi_{\prec_0}(\psi', \underline{l},\underline{\eta}) \in \Pi^{A}_{\psi} \cap \Pi^{A}_{\psi'}$, and any $i\in I(\gotM_{\lambda,J}^{gp})$, we have $\underline{l}(i)=0$, and
\[
\left\{\begin{array}{ll}(-1)^{\epsilon^{\psi}_{\pi}(a_i,1)}= (-1)^{\epsilon^{\psi'}_{\pi}(a_i,1)} = \underline{\eta}(i) & \mbox{ if }b_i=1 \\
 \epsilon^{\psi'}_{\pi}(a_i,2) = 1
& \mbox{ if }b_i=2
\end{array} \right.\;.
\]
\end{enumerate}
\end{lemma}

\begin{proof}

We first give a proof of the first part, out of which the second readily follows.

Suppose that $\pi\in \Pi^{A}_{\psi}$ is such that $\mathfrak{t}_c(\epsilon^{\psi}_{\pi})\neq1$ holds, for all $c\in J$.
Let us write $\pi= \pi_{\prec_0}(\psi, \underline{l^0},\underline{\eta})$, with $\underline{l^0}\equiv 0$, and $\prec_0$ a standard admissible order on $\gotM_{\lambda}$. Propositions \ref{prop:moegarth-formula} and \ref{prop:z-formula} show that $(-1)^{\epsilon^{\psi}_{\pi}(a_i,1)} = \underline{\eta}(i)$ holds, for all $i\in I(\gotM_{\lambda}^{gp})$, in this case.
Hence, $\underline{\eta}(i) = -\underline{\eta}(j)$ holds, for all $i,j\in I(\gotM_{\lambda}^{gp})$, with $a_{i}+1 = a_{j}-1 \in J$. 
By successive applications, for each $c\in J$, of case \eqref{it-nt} of Proposition \ref{prop:ui-nt}, we can now arrive at an inclusion $\pi\in \Pi^{A}_{\psi'}$.

Conversely, suppose that we have $\pi= \pi_{\prec_0}(\psi', \underline{l}',\underline{\eta}')\in \Pi^{A}_{\psi}\cap \Pi^{A}_{\psi'}$. 
By Proposition \ref{prop:moegarth-formula}, we see an equality
\begin{equation}\label{eq:proof-int1}
\epsilon^{\psi_J}_{\pi}(a_i,2) = 1 - \underline{l}'(i)\;,
\end{equation}
for any $i \in I(\gotM_{\lambda}^{gp})$ with $a_i\in J$.
Let us denote the set 
\[
J' = \{c\in J\,:\, \underline{l}'(i) = 0,\,\mbox{when }a_i=c\}\;.
\] 
Hence, $\underline{l}'(i) =1$, for all  $i \in I(\gotM_{\lambda,J}^{gp})$ with $a_i\in J\setminus J'$.
Successively applying Corollary \ref{prop:ui-cor} (that is, case \eqref{it-nt} of Proposition \ref{prop:ui-nt} in reverse), we see that 
\[
\pi \cong \pi_{\prec_0} ( \psi'', \underline{l}'',\underline{\eta}'')\in \Pi^{A}_{\psi''}\;,
\]
for $\psi'' = \psi_{z,\gotM_{\lambda,J\setminus J'}}$, 
with $\underline{l}''(i) = \underline{l}'(i)$, for all $i \in I(\gotM_{\lambda,J\setminus J'}^{gp})$, such that $a_i\in J\setminus J'$, and 
\begin{equation}\label{eq:proof-intersect}
    \underline{\eta}''(i) = -\underline{\eta}''(j)\;,
\end{equation}
for all $i,j\in I(\gotM_{\lambda,J\setminus J'}^{gp})$, with $a_{i}+1 = a_{j}-1 \in J'$. 
From equation \eqref{eq:proof-int1} we now see that $\epsilon^{\psi''}_{ \pi} (c,2) = 0$, for all $c\in J\setminus J'$. Thus, $\epsilon^{\psi''}_{ \pi}$ lies in the subgroup $P(T_{J\setminus J'}(\lambda))$ of $P(\gotM_{\lambda, J\setminus J'})$, and by Lemma \ref{lem:lpacket} we must have $\pi\in  \Pi_{\phi_{\psi''}}$.

Since $L$-packets are disjoint and $\Pi^{A}_{\psi} = \Pi_{\phi_{\psi}}$ is tempered, we must have $\psi = \psi''$ and $J=J'$. 
Consequently, when identifying $(-1)^{\epsilon^{\psi}_{\pi}}$ with $\underline{\eta}''$ as before, we deduce $\mathfrak{t}_c(\epsilon^{\psi}_{\pi})\neq1$, for all $c\in J$, from the condition in equation \eqref{eq:proof-intersect}.
\end{proof}

\subsubsection{Further refinement}

In order to advance towards the proofs of our main theorems, we need to gather some refined information regarding Moeglin parameters and Arthur characters of representations found in intersections of $A$-packets.
The following lemma sums up phenomena that were detected in \cite{xu21}.

\begin{lemma}\label{lem:switch}
Let $\prec$ be an admissible order on a near-tempered table $\gotM\in \mathcal{T}^{s_G}(N_G)$, for which we assume $I(\gotM^{gp}) = \{1\prec \ldots \prec t\}$.
Let $\psi = \psi_{z,\gotM}\in \Psi_0(G)$ be the associated $A$-packet, for a choice of $z\in \ZZ$.
Suppose that $1\leq k\leq t$ is such that $(a_{k}, b_k) = (a,2)$, for $a>1$.
Let
\[
\pi=\pi_{\prec}(\psi, \underline{l},\underline{\eta}) \in \Pi^{A}_{\psi}
\]
be any representation with $\underline{l}(k) =0$.

If $a_{k-1} \leq a-1   = a_{k+1} $ holds, then $(-1)^{\epsilon^{\psi}_{\pi}(a-1,1)} = \underline{\eta}(k)$.

If $a_{k-1} = a+1  \leq a_{k+1} $ holds, then $(-1)^{\epsilon^{\psi}_{\pi}(a+1,1)} = -\underline{\eta}(k)$.

\end{lemma}

\begin{proof}
In the former case, since the parity condinition forces $b_{k+1}=1$, it follows from \cite[Lemma 5.7]{xu21} that $\underline{\eta}(k+1) = -\underline{\eta}(k)$. 
By Proposition \ref{prop:admss}, we know that $a-1 \leq a_j$ holds, for all $k+1 < j$, and that $a_j \leq a_{k-1}+1 \leq a_{k+1}$, for all $j< k-1$.
Thus, we have an equality $Z_{k+1} = \{k\}$, for the set that was defined in Proposition \ref{prop:z-formula}, and $\gamma^{\psi}_{\prec}(k+1)=-1$, by the same proposition.
The claim now follows from the formula of Proposition \ref{prop:moegarth-formula}.

In the latter case, \cite[Lemma 5.6]{xu21} shows $\underline{\eta}(k-1) = \underline{\eta}(k)$. The claim follows similarly from $Z_{k-1} = \{k\}$.

\end{proof}

\begin{lemma}\label{it:Js3}

Let $\lambda, z, J, \psi, \psi'$ be as in Lemma \ref{lem:Js}.
Let $\prec$ be any admissible order on the near-tempered table $\gotM_{\lambda,J}$, and 
\[
\pi = \pi_{\prec}(\psi', \underline{l},\underline{\eta}) \in \Pi^{A}_{\psi} \cap \Pi^{A}_{\psi'}
\]
a representation. 
If $k\in I(\gotM_{\lambda,J}^{gp})$ is such that $b_k=2$ and $\underline{l}(k) = 0$, then
\[
(-1)^{\epsilon^{\psi}_{\pi}(a_k-1,1)} = \underline{\eta}(k)=(-1)^{\epsilon^{\psi}_{\pi}(a_k+1,1)+1}\;.
\]
Here, we may take $\epsilon^{\psi}_{\pi}(0,1)=0$ if necessary.

\end{lemma}

\begin{proof}

Let us assume $I(\gotM_{\lambda,J}^{gp}) = \{1\prec \ldots \prec t\}$, and write $a=a_k$.

We first treat the case of $a>1$.
It follows from Lemma \ref{lem:Js} that $\epsilon^{\psi}_{\pi}(a-1,1) = 1-\epsilon^{\psi}_{\pi}(a+1,1)$. 
If either $a_{k+1} = a-1$ or $a_{k-1}= a+1$ hold, the claim follows from Lemma \ref{lem:switch}.
Otherwise, by Proposition \ref{prop:admss}, we are left with the case of $a_{k-1} < a < a_{k+1}$. 
Arguing by Corollary \ref{prop:ui-cor}, we see an inclusion 
\[
\pi \cong \pi_{\prec_\ast}(\psi'', \underline{l}_\ast, \underline{\eta}_\ast)\in \Pi^{A}_{\psi''}\;,
\]
for $\psi'' = \psi_{z, \gotM_{\lambda, J\setminus \{a\}}}$.
Here, $\prec_\ast$ is an admissible order on the table $\gotM_{\lambda, J\setminus \{a\}}$, that is constructed out of $\prec$, by replacing the relative position of $k$ with $(a_\ast, b_\ast) = (a-1,1) \prec_\ast (a+1,1)= (a_{\ast\ast}, b_{\ast\ast})$. The parameter $(\underline{l}_\ast, \underline{\eta}_\ast)\in \mathcal{W}(\gotM_{\lambda, J\setminus \{a\}})$ is constructed out of $(\underline{l}, \underline{\eta})$ by setting $\underline{\eta}_\ast(\ast) = \underline{\eta}(k)$ and $\underline{ \eta}_\ast(\ast\ast) = -\underline{\eta}(k)$.
Using the formula of Proposition \ref{prop:z-formula}, it is evident that $\gamma^{\psi''}_{\prec_\ast} (\ast)=1$. From Proposition \ref{prop:moegarth-formula} we then have
\[
(-1)^{\epsilon^{\psi''}_{\pi}(a-1,1)} = \underline{\eta}_\ast(\ast) = \underline{\eta}(k)\;.
\]
Yet, from part \eqref{it:Js2} of Lemma \ref{lem:Js} we know that $\epsilon^{\psi}_{\pi}(a-1,1) = \epsilon^{\psi''}_{\pi}(a-1,1)$, and the statement follows.

The case of $a=1$ needs to be treated with a slightly different argument.  
First, we know that $k=1$ (Proposition \ref{prop:admss}). We take $\prec_\ast$ to be the admissible order on $ \gotM_{\lambda, J\setminus \{1\}}$ constructed out of $\prec$ by replacing $(a_1,b_1) = (1,2)$ with $(2,1)$.
We know that $\pi \in \Pi^{A}_{\psi_{\gotM_{z,\lambda,J\setminus\{1\}}}}$ from the first part Lemma \ref{lem:Js}. Hence, we may write $\pi =  \pi_{\prec_\ast}(\psi_{\gotM_{z,\lambda,J\setminus\{1\}}}, \underline{l}_\ast, \underline{\eta}_\ast)$. 
Also, by same lemma, we have $\mathfrak{t}_1(\epsilon^{\psi}_{\pi})\neq1$, which implies $\epsilon^{\psi''}_{\pi}(2,1) = \epsilon^{\psi}_{\pi}(2,1)=1$. 
By Proposition \ref{prop:z-formula}, $\gamma^{\psi''}_{\prec_\ast}(1) = 1$. Hence, it follows from Proposition \ref{prop:moegarth-formula} that $\underline{\eta}_\ast(1) =-1$. 
Thus, the inclusion $\pi \in \Pi^{A}_{\psi'}$ may be obtained out of $\psi''$ as a case of Proposition \ref{prop:ui-nt}. In particular, from injectivity of $M_{\prec}$ we conclude that $\underline{\eta}(1) = - \underline{\eta}_\ast(1) =1$.

\end{proof}

\subsection{Weak Arthur packets}\label{sect:wArthur}

Let us recall Definition \ref{defi:intro} for a weak Arthur packet $\Pi^{w}_{z,\mathcal{O}^\vee}$ that is attached to a unipotent conjugacy class $\mathcal{O}^\vee\in \mathcal{U}^\vee$ and a possible sign $z=\pm1$. 
We also recall that according to Corollary \ref{cor:WFtwist}, all representations $\pi\in \irr_{\chi_{z, \mathcal{O}^\vee}}(G)$, for $z\in\ZZ$, admit an algebraic wavefront orbit $\WF(\pi)\in \mathcal{N}_{\overline{F}}$. 
The following consequence is now evident.

\begin{proposition}[following \cite{cmbo-arthur}]\label{prop:weakWF}
For a partition $\lambda \in \mathcal{P}^{s_G}(N_G)$ and $z\in \ZZ$, we have
\[
\begin{array}{rl}
\Pi^{w}_{z,\mathcal{O}_{\lambda}^\vee}
& = \{\pi \in \irr_{\chi_{z,\mathcal{O}_{\lambda}^\vee}}(G)\;:\; \WF(\pi) = \mathcal{O}_{d(\lambda)}\} \\
&  
=  \{\pi \in \irr_{\chi_{z,\mathcal{O}_{\lambda}^\vee}}(G)\;:\;  d ( \lambda(\phi_{(\pi^t)}))  = d(\lambda)\}
\end{array}
\;,
\]
and an inclusion $\Pi_{z,\mathcal{O}_{\lambda}^\vee} 
 \subset \Pi^{w}_{z,\mathcal{O}_{\lambda}^\vee} $ holds.

\end{proposition}

\begin{proof}
For $\phi = \phi_{ \psi_{z,\gotM_{\lambda}} } = \phi_{(z,\ldots,z),\lambda}$, we have $\chi_{\phi} = \chi_{z,\mathcal{O}^\vee_{\lambda}}$. Thus, by Corollary \ref{cor:WFtwist}, we obtain $\dim(\WF(\pi)) \geq \dim (\mathcal{O}_{d(\lambda)})$, for all $\pi\in \Irr_{\chi_{z,\mathcal{O}_{\lambda}^\vee}}(G)$.

Since $\lambda(\phi)=\lambda$, the same corollary implies that $\WF(\pi^t) =\mathcal{O}_{d(\lambda)}$, for all (tempered) $\pi\in \Pi_{\phi}= \Pi^{A}_{\psi_{z,\gotM_{\lambda}}}$. By Proposition \ref{prop:integralA}\eqref{it:intA3} and Proposition \ref{prop:start-Arth}, that can also be stated as an equality $\WF(\pi) =\mathcal{O}_{d(\lambda)}$, for all representations $\pi \in \Pi_{z,\mathcal{O}_{\lambda}^\vee}$ in the quasi-basic $A$-packet.


\end{proof}

We now explicate certain unions of $A$-packets, arising from near-tempered parameters, in terms of the Langlands reciprocity (i.e. as unions of $L$-packets).

\begin{proposition}\label{prop:unionAL}

Let $\lambda\in \mathcal{P}^{s_G}(N_G)$ be a given partition.
Then, for any choice of $J\subset \mathbb{J}(\lambda)$ and $z\in \ZZ$, an equality 
\[
\bigcup_{J'\subset J}\Pi^{A}_{\psi_{z,\gotM_{\lambda,J'}}} =  \bigsqcup_{J'\subset J} \Pi_{\phi_{z, \lambda, T_{J'}(\lambda)}}
\]
of sets of integral representations in $\irr_0(G)$ holds.

\end{proposition}

\begin{proof}
We write $\psi =\psi_{z,\gotM_{\lambda,J}}$. Arguing by induction on the cardinality of $J$, it suffices to prove the containment
\[
\Pi^{A}_{\psi}  \subseteq  \Pi_{\phi_{\psi}} \cup\bigcup_{J'\subsetneq J}\Pi^{A}_{\psi_{z,\gotM_{\lambda,J'}}}\;.
\]
Let $\pi \in \Pi^{A}_{\psi}$ be a given representation. 
Suppose first that $\epsilon^{\psi}_{\pi} \in \widehat{\mathcal{S}_{\phi_{\psi}}}$. By Lemma \ref{lem:lpacket}, we have $\pi\in\Pi_{\phi_{\psi}} $.
Otherwise, using the identification of Proposition \ref{prop:expl-a-l}, we treat $\epsilon^{\psi}_{\pi}$ as a boolean function on the set 
\[
S(\gotM) = \{(c,1)\}_{ c\in S(T_J(\lambda))} \cup \{(c,2)\}_{c\in J} \;,
\]
with the condition of not being contained in $\widehat{\mathcal{S}_{\phi_{\psi}}}$ amounting to the support of $\epsilon^{\psi}_{\pi}$ not being contained in $\{(c,1)\}_{ c\in S(T_J(\lambda))}$.
In other words, there must exist $a\in J$, for which $\epsilon^{\psi}_{\pi}(a,2)=1$.
Let us fix a standard admissible order $\prec'$ on the near-tempered table $\gotM_{\lambda,J}$, and write $\pi = \pi_{\prec'}(\psi, \underline{l}',\underline{\eta}')$.
Since the function $\gamma^{\psi}_{\prec'}$ is trivial by Proposition \ref{prop:z-formula}, the formula in Proposition \ref{prop:moegarth-formula} implies an equality $\epsilon^{\psi}_{\pi}(a,2) = 1+\underline{l}'(i)$, for the index $i\in I(\gotM_{\lambda,J}^{gp})$, such that $(a_i,b_i) = (a,2)$. Thus, $\underline{l}'(i)=0$.
Appropriately constructing $(\underline{l},\underline{\eta})\in \mathcal{W}(\gotM_{\lambda,J\setminus \{a\}})$ out of $(\underline{l}',\underline{\eta}')$ and an admissible order $\prec$ on the table $\gotM_{\lambda, J\setminus\{a\}}$ out of $\prec'$, Corollary \ref{prop:ui-cor} now implies that
\[
\pi\cong \pi_{\prec}(\psi_{\gotM_{z, J\setminus \{a\}}},\underline{l},\underline{\eta})  \in \Pi^{A}_{\psi_{\gotM_{z, J\setminus \{a\}}}} \;.
\]

\end{proof}

\subsubsection{Proof of Theorem \ref{thm:B}}


Given a partition $\lambda\in  \mathcal{P}^{s_G}(N_G)$ and $z\in \ZZ$, we see from Proposition \ref{prop:weakWF}, Theorem \ref{thm:almost-intro} and Lemma \ref{lem:half-intro}, that an equality
\begin{equation}\label{eq:thmB1}
\bigsqcup_{\mu\in \spc(\lambda)} \Pi_{\phi_{z,\lambda,\mu}} = \{\pi^t\::\: \pi\in \Pi^{w}_{z,\mathcal{O}_{\lambda}^\vee}\}
\end{equation}
holds.
Adding in Proposition \ref{prop:unionAL} and recalling compatibility (Proposition \ref{prop:start-Arth}) with Aubert involution, we may write
\begin{equation}\label{eq:thmB2}
\Pi^{w}_{z,\mathcal{O}_{\lambda}^\vee}  = \bigcup_{J\subset \mathbb{J}(\lambda)}\Pi^{A}_{(\psi_{z,\gotM_{\lambda,J}})^t}\;.
\end{equation}
Finally, from a geometric perspective, given any classes $\mathcal{O}^\vee, \mathcal{O}_1^\vee\in \mathcal{U}^\vee$ with $\mathcal{O}_1^\vee\in \spc(\mathcal{O}^\vee)$, by Proposition \ref{prop:spc-comb} there are partitions $\mu\in \spc(\lambda)$ so that $\mathcal{O}^\vee =\mathcal{O}_{\lambda}^\vee$ and $\mathcal{O}_1^\vee=  \mathcal{O}_{\mu}^\vee$.
We then define an $L$-parameter
\[
\phi_{z,\mathcal{O}^\vee,\mathcal{O}_1^\vee}: = \phi_{z,\lambda,\mu} = \phi_{\psi_{z,\gotM_{\lambda,J}}}\in \Phi_u(G)\;,
\]
where $J\subset \mathbb{J}(\lambda)$ is such that $\mu = T_J(\lambda)$, and an $A$-packet
\begin{equation}\label{eq:notation1}
   \Pi_{z,\mathcal{O}^\vee,\mathcal{O}^{\vee}_1} := \Pi^{A}_{(\psi_{z,\gotM_{\lambda,J}})^t}\;, 
\end{equation}
coinciding with the definition of \eqref{eq:introApack}.
In these terms, the equalities \eqref{eq:thmB1} and \eqref{eq:thmB2} now amount to the statement of Theorem \ref{thm:B}.

\subsection{Weakly spherical $A$-packets}\label{sect:weak-sph}

In order to establish Theorem \ref{thm:D}, we first examine the structure of $A$-packets that contain those tempered representations that are Aubert-dual to the formally weakly spherical representations.
Namely, for a partition $\lambda\in \mathcal{P}^{s_G}(N_G)$ and $z\in \ZZ$, we set the tempered $A$-parameter $\psi = \psi_{z,\gotM_{\lambda}}$, and define the set
\[
\Pi^{f}_{z,\lambda} = \{\pi \in \Pi^{A}_{\psi}\: :\: \epsilon^{\psi}_{\pi}\in A^{\dagger}(\mathcal{O}_{\lambda}^\vee)\}
\]
of tempered representations.
Here, we identified $ A^{\dagger}(\mathcal{O}_{\lambda}^\vee)$ as a subgroup of the character group $\widehat{\mathcal{S}_{\psi}}$ using \eqref{eq:311}.

\begin{lemma}\label{lem:key-tech}

Let $\lambda\in \mathcal{P}^{s_G}(N_G)$ be a partition, and $z\in \ZZ$.
Let $J\subset \mathbb{J}(\lambda)$ be a subset, and $\psi=\psi_{z,\gotM_{\lambda}},\, \psi'=\psi_{z,\gotM_{\lambda,J}}\in \Psi_0(G)$ be the associated near-tempered $A$-parameters. 
Let $\prec$ be an admissible order on the table $\gotM_{\lambda,J}$, for which we assume $I(\gotM_{\lambda,J}^{gp}) = \{1\prec \ldots \prec t\}$. Let 
\[
\pi = \pi_{\prec}(\psi', \underline{l},\underline{\eta}) \in \Pi^{A}_{\psi'} \cap \Pi^{f}_{z,\lambda}
\]
be a given representation.
Suppose that $t> k \in \widetilde{I(\gotM_{\lambda,J}^{gp})}$ is an index with $a_{k+1} - a_k =2$.
Then,
\begin{enumerate}
    \item\label{it:key1}
    If $b_k= b_{k+1} = 1$ and $\underline{\eta}(k+1) =  -\underline{\eta}(k)$ hold, then necessarily $a_k+1 \in \mathbb{J}(\lambda) \setminus J$.

    \item\label{it:key2}
    If $b_k= b_{k+1} = 2$ holds, then $\underline{l}(k+1) +  \underline{l}(k) \neq 1$.
\end{enumerate}

\end{lemma}

\begin{proof}

Let us write $a= a_k +1$. 
The formula of Proposition \ref{prop:z-formula} gives
\begin{equation}\label{eq:dagger2}
\gamma^{\psi'}_{\prec}(k) = (-1)^{\#\{ 1\leq i < k\;:\; a_i = a\}},\quad 
\gamma^{\psi'}_{\prec}(k+1) = (-1)^{\#\{ k+1 < i < \leq t\;:\; a_i = a\}}\;.
\end{equation}
We prove the first part: Suppose that we have $b_k  = b_{k+1} = 1$ and $\underline{\eta}(k+1) =  -\underline{\eta}(k)$. 
Proposition \ref{prop:moegarth-formula} shows the equalities
\[
(-1)^{\epsilon^{\psi'}_{\pi}(a_k,1)} = \gamma^{\psi'}_{\prec}(k) \underline{\eta}(k),\quad 
(-1)^{\epsilon^{\psi'}_{\pi}(a_{k+1},1)} = \gamma^{\psi'}_{\prec}(k+1) \underline{\eta}(k+1)\;.
\]
Here, we formally set $\epsilon^{\psi}_{\pi}(0,1) = \epsilon^{\psi'}_{\pi}(0,1)=0$ and $\gamma^{\psi'}_{\prec}(0)=1$, if necessary. 
In combination with Lemma \ref{lem:Js}\eqref{it:Js2}, we obtain
\begin{equation}\label{eq:dagger1}
(-1)^{\epsilon^{\psi}_{\pi}(a_k,1)}\gamma^{\psi'}_{\prec}(k) = (-1)^{1+\epsilon^{\psi}_{\pi}(a_{k+1},1)} \gamma^{\psi'}_{\prec}(k+1)\;.
\end{equation}
Suppose first that $a\not\in \mathbb{J}(\lambda)$. 
In this case, when $a\neq1$, we must have $p(a_k)=p(a_{k+1})$, where $p:S(\lambda)\to S^{\dagger}(\lambda)$ is the projection map (Section \ref{sect:quot}).
Thus, when identifying $A^{\dagger}(\mathcal{O}_{\lambda}^\vee)$ with $P^{\dagger}(\lambda)_0$ (Proposition \ref{prop:expl-a-l}), we see that $\epsilon^{\psi}_{\pi}(a_k,1)=\epsilon^{\psi}_{\pi}(a_{k+1},1)$. The same conclusion follows in the $a=1$ case from Lemma \ref{lem:tail}.
Since $a\not\in J$, we have $a_i\neq a$ for all $1\leq i\leq t$, implying that $\gamma^{\psi'}_{\prec}(k)=\gamma^{\psi'}_{\prec}(k+1)=1$ by \eqref{eq:dagger2}. We reach a contradiction to \eqref{eq:dagger1}.
It follows that $a\in \mathbb{J}(\lambda)$.

Suppose now that $a\in J$. That means there is a unique $1\leq i_0 \leq t$ with $a = a_{i_0}$. Hence, $\gamma^{\psi'}_{\prec}(k)= - \gamma^{\psi'}_{\prec}(k+1)$ must follow from \eqref{eq:dagger2}. 
Yet, from Lemma \ref{lem:Js} we know that $\mathfrak{t}_a(\epsilon^{\psi}_{\pi})\neq1$. This fact implies $\epsilon^{\psi}_{\pi}(a_k,1)= 1-\epsilon^{\psi}_{\pi}(a_{k+1},1)$ and contradicts \eqref{eq:dagger1} again.

We prove the second part: Let us assume the contrary, that is, $b_k  = b_{k+1} = 2$ and $\underline{l}(k+1) +  \underline{l}(k) =1$.
In this case $a_k =a -1 , a_{k+1}=a +1\in J$. By Lemma \ref{lem:even}, the multiplicity $m(a,\lambda)$ must be even.
By Lemma \ref{lem:Js}\eqref{it:Js2}, we have 
\[
\epsilon^{\psi'}_{\pi}(a_k,2) = \epsilon^{\psi'}_{\pi}(a_{k+1},2)=1\;.
\]
Adding in the equation of Proposition \ref{prop:moegarth-formula}, we see equalities
\[
\gamma^{\psi'}_{\prec}(k) = (-1)^{\underline{l}(k)}\;,\quad \gamma^{\psi'}_{\prec}(k+1) = (-1)^{\underline{l}(k+1)}\;,
\]
which together with the assumption imply $\gamma^{\psi'}_{\prec}(k)\gamma^{\psi'}_{\prec}(k+1)=-1$.
From \eqref{eq:dagger2} we see that $\#\{1\leq i\leq t\: : \: a_i=a\}$ is odd. That in particular implies, using \eqref{eq:ntpformula}, that the multiplicity $m(a, T_J(\lambda))$ is odd. 
Consequently, $m(a, \lambda) = m(a, T_J(\lambda))+2$ is odd as well, providing a contradiction.

\end{proof}

\begin{theorem}\label{thm:lastin3}
Let $\lambda\in \mathcal{P}^{s_G}(N_G)$ be a partition, and $z\in \ZZ$. Let $\psi = \psi_{z,\gotM_{\lambda}}\in \Psi_0(G)$ be the associated tempered $A$-parameter.
Suppose that $\psi'\in \Psi(G)$ is an $A$-parameter, so that the intersection $\Pi^{A}_{\psi'}\cap \Pi^{f}_{z,\lambda}$ is non-empty.
Then, a subset $J\subset \mathbb{J}(\lambda)$ must exist, so that $\psi' = \psi_{z,\gotM_{\lambda,J}}$.
Moreover, the set of representations $\Pi^{A}_{\psi'}\cap \Pi^{f}_{z,\lambda}$ consists precisely of those tempered representations $\pi\in \Pi^{A}_{\psi_{z,\gotM_{\lambda}}}$, for which the character $\epsilon^{\psi}_{\pi}\in A^{\dagger}(\mathcal{O}^\vee_{\lambda})$ is $\mathcal{O}^\vee_{T_J(\lambda)}$-primitive.

\end{theorem}

\begin{proof}

The last statement is a consequence of Lemma \ref{lem:Js}\eqref{it:Js1} and Proposition \ref{prop:prim}.

Assume the contrary, that is, that $\psi'$ is not of the desired form. 
For the intersection to be non-empty $\chi_{\psi'} = \chi_{\psi}$ must be satisfied. In particular, a table $\gotM'\in \mathcal{T}^{s_G}(N_G)$ exists, for which $\psi'= \psi_{z,\gotM'}$ (Proposition \ref{prop:integralA}).

Suppose first that $\gotM'$ is non-negative. 
By Proposition \ref{prop:ui-exhaust}\eqref{it:ui-exhaust1} we deduce that there is a set $J\subset \mathbb{J}(\lambda)$, an admissible order $\prec$ on $\gotM_{\lambda, J}$ and an $A$-parameter $\psi''= \psi_{z,\gotM''}\in \Psi(G)$ \textit{not} of the form $\psi_{z,\gotM_{\lambda,J'}}$, for any $J'\subset \mathbb{J}(\lambda)$, so that
\[
\pi= \pi_{\prec}(\psi_{z,\gotM_{\lambda,J}}, \underline{l},\underline{\eta})\in \Pi^{A}_{\psi_{z,\gotM_{\lambda,J}}}\cap \Pi^{A}_{\psi''}
\]
holds, and the inclusion $\pi\in \Pi^{A}_{\psi''}$ is obtained out of $M_{\prec}(\pi)$ through an application of one of the moves described in Proposition \ref{prop:ui-nt}.
Let us assume $I(\gotM_{\lambda,J}^{gp}) = \{1\prec \ldots \prec t\}$, and that $t> k \in \widetilde{I(\gotM_{\lambda,J}^{gp})}$ is the index for which the move to produce the inclusion $\pi\in \Pi^{A}_{\psi''}$ is performed. 

If $a_{k+1}-a_k=2$, confronting cases \eqref{it-nt} and \eqref{it-nt2} of Proposition \ref{prop:ui-nt} with Lemma \ref{lem:key-tech}, implies that $a\in  \mathbb{J}(\lambda)\setminus J$ and that $\gotM''= \gotM_{\lambda, J\cup\{a\}}$. 
This is a contradiction to our assumption.
Hence, we are left with cases \eqref{it-nt3} and \eqref{it-nt4} of Proposition \ref{prop:ui-nt}. We assume that 
\[
\left\{\begin{array}{l} 
a_{k+1} - a_k  \in \{3,4\} \\
\underline{l}(k+1) = \underline{l}(k) =0 \\
\underline{\eta}(k+1)=  (-1)^{b_k}\underline{\eta}(k)
\end{array}\right.\;.
\]
We may also assume, without loss of generality, that $b_{k+1}=2$, since analogous arguments are valid with roles of $k, k+1$ switched. 
In particular, $a_{k+1}\in J$. By Lemma \ref{lem:Js}\eqref{it:Js2} we have $\epsilon^{\psi'}_{\pi}(a_{k+1},2)=1$. Then, from Proposition \ref{prop:moegarth-formula}, we also have $\gamma^{\psi'}_{\prec}(k+1) = 1$. 
Therefore, by Proposition \ref{prop:z-formula} the number of indices $1\leq i\leq t$ with $a_i = a_{k+1}-1$ is even. This implies that the multiplicity $m:=m(a_{k+1}-1,T_J(\lambda))$ is even.
Let us take note of the identity
\begin{equation}\label{eq:ident-lem}
(-1)^{\epsilon^{\psi}_{\pi}(a_{k+1}-1,1)} = \underline{\eta}(k+1) 
\end{equation}
that follows from Lemma \ref{it:Js3}.

We assume first that $a_{k+1}-2\in \mathbb{J}(\lambda)$. 
By Lemma \ref{lem:even}, this assumption forces $m(a_{k+1}-1,\lambda)$ to be even. That can be the case, only when $a_{k+1}-2\in J$ and $m(a_{k+1}-1,\lambda) = m+2$.
Since $\prec$ is admissible and $(a_{k+1}-2, 2)\in S(\gotM_{\lambda,J})$, we see that $a_k = a_{k+1}-3$, and $b_k=1$. Moreover, $\gamma^{\psi'}_{\prec}(k) =-1$ follows from Proposition \ref{prop:z-formula}.
Hence, applying Proposition \ref{prop:moegarth-formula} again, together with the assumption and \eqref{eq:ident-lem}, we see that
\[
(-1)^{\epsilon^{\psi}_{\pi}(a_k,1)} = -\underline{\eta}(k) = \underline{\eta}(k+1)=(-1)^{\epsilon^{\psi}_{\pi}(a_{k+1}-1,1)}\;.
\]
By Lemma \ref{lem:Js}\eqref{it:Js1}, we also have $\mathfrak{t}_{a_{k+1}-2}(\epsilon^{\psi}_{\pi})\neq1$. In other words, $\epsilon^{\psi}_{\pi}(a_{k},1) = 1-\epsilon^{\psi}_{\pi}(a_{k+1}-1,1)$, which is a contradiction.

Now we assume that $a_{k+1}-2\not\in \mathbb{J}(\lambda)$. 
Since $a_{k+1}-1, a_{k+1}-3\in S(\lambda)$ are valid in all cases, $a_{k+1}-1, a_{k+1}-3$ must share the same block in $S^{\dagger}(\lambda)$ in this case.
In particular, since $\pi\in \Pi^{f}_{z,\lambda}$ is assumed, we must have 
\begin{equation}\label{eq:ident-lem2}
\epsilon^{\psi}_{\pi}(a_{k+1}-3,1) = \epsilon^{\psi}_{\pi}(a_{k+1}-1,1)\;.
\end{equation}
If $a_k = a_{k+1}-4$, this is a contradiction to Lemma \ref{it:Js3}.
Otherwise, $a_k = a_{k+1}-3$, and reasoning as before, we have, $\gamma^{\psi'}_{\prec}(k) = 1$ because of $(a_k+1,2)\not\in S(\gotM_{\lambda,J})$. Similarly, that implies $(-1)^{\epsilon^{\psi}_{\pi}(a_k,1)} = - \underline{\eta}(k+1)$, which contradicts \eqref{eq:ident-lem} and \eqref{eq:ident-lem2}.

Finally, suppose that $\gotM'$ is not non-negative. In this case, we apply Proposition \ref{prop:ui-exhaust}\eqref{it:ui-exhaust2} to produce a constellation of the form
\[
\pi\cong \pi_{\prec}(\widetilde{\psi}, \underline{l},\underline{\eta})\in \Pi^{A}_{\widetilde{\psi}} \cap \Pi^{A}_{\psi''}\;,
\]
where $\widetilde{\psi} = \psi_{z,\widetilde{\gotM}} \in \Psi_0(G)$, for a non-negative table $\widetilde{\gotM}$, $\prec$ an admissible order on $\widetilde{\gotM}$, while $\psi'' = \psi_{z,\gotM''}\in \Psi(G)$, with a not non-negative table $\gotM''$, is produced out of $(\underline{l},\underline{\eta})$ using a move described in Proposition \ref{prop:phantom}.
The first part of the proof implies that $\widetilde{\gotM} = \gotM_{\lambda,J}$, for a subset $J\subset \mathbb{J}(\lambda)$.
We retain the notation of $I(\gotM_{\lambda,J}^{gp}) = \{1\prec \ldots \prec t\}$.
The possible moves of Proposition \ref{prop:phantom} become limited to only two cases.
One case is that of $(a_1,b_1) =(2,2) $ and $\underline{l}(1)=0$.  The other case is that of $(a_1,b_1) = (3,2)$, $\underline{l}(1)=0$ and $\underline{\eta}(1)=-1$.

In the former case, we have $2\in J$, which means that the multiplicity $m(1,\lambda) = m(1,T_J(\lambda))+1$ is even. 

In the latter case, from Lemma \ref{it:Js3} we see that $\epsilon^{\psi}_{\pi}(2,1) = 1$. By Lemma \ref{lem:tail},  and the inclusion $\pi\in \Pi^{f}_{z,\lambda}$, that implies $1\in \mathbb{J}(\lambda)$. Since $3\in \mathbb{J}(\lambda)$ as well, Lemma \ref{lem:even} forces $m(2,\lambda) = m(2,T_J(\lambda))+1$ to be even.

We see that in both cases $\#\{i\,:\, a_i = a_1-1\}$ is an odd number. By Proposition \ref{prop:z-formula}, that means $\gamma^{\widetilde{\psi}}_{\prec}(a_1,b_1) = -1$.
Yet, by Lemma \ref{lem:Js}\eqref{it:Js2} $\epsilon^{\widetilde{\psi}}_{\pi}(a_1,b_1) = 1$ and we see a contradiction to the formula of Proposition \ref{prop:moegarth-formula}.

\end{proof}

It is notable that the techniques developed by Hazeltine--Liu--Lo in \cite[Section 8]{hll22} provide an alternative framework through which Theorem \ref{thm:lastin3} could possibly be tackled. Exploring the role of the set $\Pi^{f}_{z,\lambda}$ of tempered representations in this context presents an intriguing direction for future investigation.

\subsubsection{Proof of Theorem \ref{thm:D}}

Let $\mathcal{O}^\vee\in \mathcal{U}^\vee$ be a unipotent conjugacy class, $z\in \{\pm1\}$ a sign, $\lambda\in \mathcal{P}^{s_G}(N_G)$ a partition, and $\epsilon\in  A^{\dagger}(\mathcal{O}_{\lambda}^\vee)$ a character.
Let
\[
\delta = \delta(z,\mathcal{O}_{\lambda}^\vee,\epsilon)\in \Pi_{z,\mathcal{O}_{\lambda}^\vee}
\]
be the resulting (formally weakly spherical) anti-tempered representation.
By construction of Corollary \ref{cor:temp-anti}, we see that $\delta^t\in \Pi^{f}_{z,\lambda}$.

Let $\psi\in \Psi(G)$ be an $A$-parameter.
Since the $A$-packet $\Pi^{A}_{\psi^t}$ consists of the irreducible representations that are Aubert-dual to the constituents of $\Pi^{A}_{\psi}$ (Proposition \ref{prop:start-Arth}), we see by Theorem \ref{thm:lastin3}, that $\delta\in \Pi^{A}_{\psi}$ holds, if and only if, $\psi^t = \psi_{z,\gotM_{\lambda,J}}$, for a subset $J\subset \mathbb{J}(\lambda)$, such that $\epsilon$ is $\mathcal{O}^\vee_{T_J(\lambda)}$-primitive.
Indeed, the latter condition is equivalent to $\delta\in \Pi_{z,\mathcal{O}^\vee_{\lambda}, \mathcal{O}_{T_J(\lambda)}^\vee}$ (as defined in \eqref{eq:notation1}).
Theorem \ref{thm:D} now follows, pending the parameterization of $\Pi^{\mathrm{sph}}_{z,\mathcal{O}^\vee}$ given by Theorem \ref{thm:C}, which is proved independently in the following sections.


\section{Reduction to Springer theory}\label{sect:redtospringer}

The remainder of this work is devoted to a proof of Theorem \ref{thm:C}.

We translate the problem of detecting weak sphericity in the anti-tempered $A$-packet $\Pi_{\mathcal{O}^\vee}$, for $\mathcal{O}^\vee\in \mathcal{U}^\vee$, into a property of corresponding $W_G$-representations that arise in the cohomology of Springer fibres.


\subsection{Springer representations}\label{sect:springer}

For a conjugacy class $\mathcal{O}^\vee\in \mathcal{U}^\vee$, we set the \textit{Springer fibre} $\mathcal{B}_{\mathcal{O}^\vee}$ to be the projective complex variety of Borel subgroups of $\mathbf{G}^\vee(\mathbb{C})$ that contain a fixed representative $u\in \mathcal{O}^\vee$.
The \textit{Springer representation} \cite{Springer1978} is a linear action of the finite Weyl group $W_G$ on the finite-dimensional complex cohomology space
\[
    H^\ast(\mathcal{B}_{\mathcal{O}^\vee}) = \bigoplus_{i=0}^{d_{\mathcal{O}^\vee}} H^i(\mathcal{B}_{\mathcal{O}^\vee})\;.
\]
Here we mark $d_{\mathcal{O}^\vee} = 2 \dim(\mathcal{B}_{\mathcal{O}^\vee})$ as the top degree, for which the space is non-zero.
The action of the centralizer group of $u$  on the variety $\mathcal{B}_{\mathcal{O}^\vee}$, naturally produces  a linear action of the component group $A(\mathcal{O}^\vee)$ on the same cohomology space $H^\ast(\mathcal{B}_{\mathcal{O}^\vee})$. In fact, the actions of $W_G$ and $A(\mathcal{O}^\vee)$ commute.
Thus, for any irreducible local system $\epsilon\in \widehat{A(\mathcal{O}^\vee)}$ on $\mathcal{O}^\vee$, the space
\[
\Sigma(\mathcal{O}^\vee,\epsilon) := \mathrm{Hom}_{A(\mathcal{O}^\vee)}( \epsilon, H^\ast(\mathcal{B}_{\mathcal{O}^\vee}))\;,
\]
when non-zero, is a $W_G$-representation.

Special attention of Springer theory is typically drawn to the sub-representation given by the top degree cohomology space
\[
\sigma(\mathcal{O}^\vee,\epsilon) := \mathrm{Hom}_{A(\mathcal{O}^\vee)}( \epsilon, H^{d_{\mathcal{O}^\vee}}(\mathcal{B}_{\mathcal{O}^\vee}))< \Sigma(\mathcal{O}^\vee,\epsilon),\;
\]
When non-zero, $\sigma(\mathcal{O}^\vee,\epsilon)$ is an irreducible $W_G$-representation. The local system $\epsilon$ is said to be \textit{of Springer-type} in this case.
Moreover, each irreducible complex $W_G$-representation $\sigma$ is isomorphic to a Springer representation $\sigma(\mathcal{O}^\vee,\epsilon)$, for a unique pair $(\mathcal{O}_{\sigma}^\vee,\epsilon_{\sigma}) : = (\mathcal{O}^\vee,\epsilon)$.
    %
    %
For any $\epsilon\in \widehat{ A(\mathcal{O}^\vee) }$, we also write $\overline{\Sigma(\mathcal{O}^\vee,\epsilon)}$ for the representation of $\widetilde{W_G}$ obtained from inflating $\Sigma(\mathcal{O}^\vee,\epsilon)$ through the projection $p:\widetilde{W_G}\to W_G$. 

\begin{remark}\label{rem:shoji}
It is known \cite[Theorem 2.5]{shoji83} that a local system $\epsilon\in \widehat{A(\mathcal{O}^\vee)}$ is of Springer-type, if and only if, the full cohomology space $\Sigma(\mathcal{O}^\vee,\epsilon)$ is non-zero.
\end{remark}

\subsection{Iwahori--Matsumoto theory}\label{sect:IM}

In order to compare the $p$-adic group situation with the Springer construction, we now revisit the setup of Section \ref{sect:iwa-decomp}, while recalling further details of the Iwahori--Matsumoto theory (\cite{IM}, see also \cite{aubert2021nonabelian}).

\subsubsection{Lie theory}
The choice of maximal torus $T$ for the $F$-split simple group $G$, produces a root datum $(X,Y,R,R^\vee)$, where $R$ denotes the set of roots in the algebraic character lattice $X$, and $R^\vee$ the set of corresponding co-roots in the algebraic co-character lattice $Y$.

Our choice of the Iwahori subgroup $I_G<G$ also pins down a basis of positive roots $\Pi\subset R$. We denote by $s_{\alpha}\in W_G$ the reflection given by a root $\alpha\in R$ and view $(W_G,S)$ as a (finite) Coxeter system, with $S = \{s_{\alpha} \mid \alpha \in \Pi\}$.
An identification of the Weyl group $W_G$ with the quotient $ N(G_{\mathfrak{O}_F}, T_{\mathfrak{O}_F})/T_{\mathfrak{O}_F}<\widetilde{W_G}$ of the normalizer subgroup, and of the group  $T/T_{\mathfrak{O}_F}$ with $Y$, sets up a splitting for the projection
\[
p: \widetilde{W_G} \cong W_G \ltimes Y\;\to\; W_G\;.
\]
Through this identification, the subgroup
\[
W_a := W_G\ltimes \mathbb ZR^\vee < \widetilde{W_G}
\]
is visible, to which we refer as the \textit{affine Weyl group}.
Taking the highest, relative to $\Pi$, root $\alpha_0\in R$, we write $s_0: = (s_{\alpha_0}, -\alpha_0^{\vee})\in W_a$.
Then, $(W_a, S^a)$ becomes an (affine) Coxeter system, where $S^a = S\cup \{s_0\}$.
When $G$ is a symplectic (simply-connected) group, an equality $W_a =\widetilde{W_G}$ is valid. In general, we take note of the subgroup
\[
\Omega: = (N(G, T_{\mathfrak{O}_F}) \cap N(G,I_G))/ T_{\mathfrak{O}_F} < \widetilde{W_G}\;,
\]
which is isomorphic to $ Y/ \mathbb{Z}R^\vee$. 
A decomposition
\[
\widetilde{W_G} = W_a\rtimes \Omega
\]
holds. Moreover, the group $\Omega$ acts by conjugation on the set of Coxeter generators $S^a$ for $W_a$.

\subsubsection{Compact subgroups}

It follows from the Iwahori decomposition of \eqref{eq:iwah-decomp} that any compact subgroup $I_G < K < G$ gives rise to a finite subgroup
\[
W_K = (K\cap N(G, T_{\mathfrak{O}_F}))/ T_{\mathfrak{O}_F} < \widetilde{W_G}\;,
\]
so that a decomposition
\begin{equation}\label{eq:finiteK}
 K=\bigsqcup_{w\in W_K} I_G  w I_G
\end{equation}
holds.

\begin{theorem}\cite[Theorem 2.27]{IM}\label{thm:subgroups}

\begin{enumerate}
\item 
For any compact subgroup $I_G < K < G$, let us denote the group $\Omega_K = W_K\cap \Omega $ and the set $J_K = W_K\cap S^a$. Let $W'_K < W_a$ be the group generated by $J_K$ (so that $(W'_K, J_K)$ is a finite Coxeter system).
Then, the conjugation action of $\Omega_K$ normalizes $W'_K$, and an equality
\[
W_K = W'_K \rtimes \Omega_K < \widetilde{W_G}
\]
holds.
    
\item 
For any subgroup $\Omega' < \Omega$ and a proper subset $J \subsetneq S^a$, stable under the $\Omega'$-action, there is a compact subgroup $I_G < K < G$ with $(\Omega_K, J_K) = (\Omega', J)$.

\item\label{it:maxcomp}\cite[Section 6.1]{aubert2021nonabelian} A compact subgroup $I_G < K < G$ is maximal, if and only if, $S^a-J_K$ is a single non-empty $\Omega_K$ orbit, and $\Omega_K = \mathrm{Stab}_{\Omega}(J_K)$.
    \end{enumerate}
\end{theorem}

Note, that since $W_K$ is a finite group, while the kernel $Y$ of the projection $p: \widetilde{W_G} \to W_G$ is torsion-free, we may naturally identify
\begin{equation}\label{eq:WK}
W_K \cong \widehat{W}_K:= p(W_K) < W_G\;,
\end{equation}
with a subgroup of the Weyl group.
Let us also take note of the subgroup $W'_K\cong \widehat{W}'_K:=p(W'_K)< \widehat{W}_K$, so that $\widehat{W}_K/ \widehat{W}'_K \cong \Omega\cap W_K$.

\subsubsection{Unramified characters}

Suppose that $\zeta: G \to \mathbb{C}^\times$ is a group character, for which the restriction $\zeta|_{I_G}$ is trivial. 
With the previous setting in place, $\zeta$ factors through a character of $W_K$ and consequently produces a character of the subgroup $\widehat{W}_K < W_G$. 
We denote that character as $\widehat{\zeta}_K: \widehat{W}_K \to \CC^\times$.

\begin{lemma}\label{lem:unram-ch}
For any compact subgroup $I_G<K<G$ and a group character $\zeta: G \to \mathbb{C}^\times$ for which $\zeta|_{I_G}$ is trivial, the restriction $(\widehat{\zeta}_K)|_{\widehat{W}'_K}$ is trivial.
If $\zeta|_K$ is non-trivial, then $\widehat{\zeta}_K$ is non-trivial.
\end{lemma}
\begin{proof}
By the Iwahori decomposition \eqref{eq:finiteK}, a non-trivial $\zeta|_K$ must factor through a non-trivial character of $W_K$.
Yet, by \cite[Proposition 2.20]{IM} the commutator subgroup $G'<G$ satisfies
\[
G' I_G =\bigsqcup_{w\in W_a} I_G  w I_G\;.
\]
Thus, $\zeta$ must factor through the quotient $\widetilde{W_G}/ W_a\cong \Omega$.
\end{proof}

\subsection{Classical groups specifics}

\subsubsection{Maximal compacts subgroups of classical groups}

It will be useful to explicate further the Iwahori--Matsumoto theory for the specific groups $G$ of our interest.

Viewed as algebraic characters of the torus $T$ of diagonal matrices, the simple roots $\Pi = \{\alpha_1,\ldots,\alpha_{n_G}\}$ are realized as
\[
\alpha_i\left( \mathrm{diag}(a_1,\ldots,a_{n_G})\right) = a_ia_{i+1}^{-1}\,,\; i=1,\ldots,n_G-1\;,
\]
and $\alpha_{n_G}\left( \mathrm{diag}(a_1,\ldots,a_{n_G})\right) = a_{n_G}$ (for orthogonal $G$), or $=a_{n_G}^2$ (for symplectic $G$).
We then write $S = \{s_1,\ldots, s_{n_G}\}$, with $s_i = s_{\alpha_i}$, for the generators of $W_G$.
The resulting affine Coxeter system $(W_a, S^a)$ corresponds to the affine Dynkin diagram
        \begin{equation}
            \dynkin[extended,labels={s_0,s_1,s_2,s_{n_G-2},s_{n_G-1},s_{n_G}},edge length =  1 cm, root radius = .075cm] C{}.
        \end{equation}
in the symplectic case, or to the diagram
    \begin{equation}\label{eq:diagram-aff}
        \dynkin[extended,labels={s_0,s_1,s_2,s_3,s_{n_G-2},s_{n_G-1},s_{n_G}},edge length = 1 cm, root radius = .075cm] B{} 
    \end{equation}
in the odd orthogonal case.
Recall again that the group $\Omega$ is trivial in the symplectic case. 
For odd orthogonal $G$, $\Omega$ is a two-element group, whose generator we denote as $\omega\in \widehat{W_G}$. The action of $\omega$ on $W_a$ is given by the non-trivial automorphism of the diagram \eqref{eq:diagram-aff}.
%
Thus, the specification of Theorem \ref{thm:subgroups}\eqref{it:maxcomp} onto our case provides the following.

\begin{proposition}\label{prop:maximalcompacts}
    \begin{enumerate}
        \item 
        In the case of $G = \mathrm{Sp}_{2n_G}(F)$, the set of maximal compact subgroups of $G$ that contain $I_G$ is described as $K_0,\dots, K_{n_G}$, where
\[
(\Omega_{K_i},J_{K_i}) = (\{1\},S^a - \{s_i\}),\quad i=0,1\ldots, n_G\;,
\]
is their Iwahori--Matsumoto parameterization as in Theorem \ref{thm:subgroups}.

\item 

        In the case of $G = \mathrm{SO}_{2n_G+1}(F)$, the set of maximal compact subgroups of $G$ that contain $I_G$ is described as $K_0,\dots, K_{n_G}$, where
\[
(\Omega_{K_i},J_{K_i}) = (\Omega,S^a - \{s_i\}),\quad i=2,\ldots, n_G\;,
\]
and,
\[
(\Omega_{K_0},J_{K_0}) = (\{1\},S^a - \{s_0\}),\quad (\Omega_{K_1},J_{K_1}) = (\Omega,S^a - \{s_0,s_1\})\;,
\]
is their Iwahori--Matsumoto parameterization as in Theorem \ref{thm:subgroups}.
    \end{enumerate}
\end{proposition}
The enumeration of maximal compact subgroups in Proposition \ref{prop:maximalcompacts} now agrees with the matrix description of Section \ref{sect:iwa-decomp}.

\subsubsection{Signed permutations}\label{sect:sgn-perm}

As preparation for the combinatorial analysis of Springer representations in Section \ref{sect:walds-la}, we would like to explicate the structure of $W_G$ and of its subgroups of the form $\widehat{W}_{K_i}$, that arise from maximal compact subgroups of $G$.

\begin{definition}
    For an integer $n\geq1$, let $W_n$ be the group of \textit{signed permutations} on $n$ letters. This is the group of permutations $\sigma$ of the set of indices $\{\pm 1,\dots,\pm n\}$, which satisfy $\sigma(-i) = -\sigma(i)$, for all $1\leq i\leq n$.
\end{definition}

Let $\sgn^\pm: W_n\to \{\pm1\}$ be the quadratic group character given by
\[
\sgn^\pm(\sigma)  = \prod_{i=1}^n \sgn^\pm(\sigma(i))\;,
\]
where $\sgn^\pm(i) = 1$ if $i> 0$ and $=-1$ if $i<0$.
Let $W'_n< W_n$ be the kernel subgroup of $\sgn^\pm$.
A natural isomorphism 
\begin{equation}\label{eq:sgnperm}
W_{G} \cong W_{n_G}
\end{equation}
appears, when identifying $s_i\in S$ with $(i,\, i+1)(-i,\; -(i+1))\in W_n$ (in cycle notation), for each $1\leq i\leq n_G-1$, and $s_{n_G}\in S$ with $(n,\,-n)\in W_{n_G}$.
Let us denote the element $\hat{s}_0:= p(s_0)\in W_{n_G}$, viewed as a signed permutation, as well as $\hat{\omega}:= p(\omega)\in W_{n_G}$ in the case of orthogonal $G$.
Since the highest root $\alpha_0\in R$ is given as
\[
\alpha_0\left( \mathrm{diag}(a_1,\ldots,a_{n_G})\right) = \left\{ \begin{array}{ll} a_1^2 &  \mbox{symplectic } G  \\ a_1a_2 & \mbox{orthogonal } G  \end{array}\right.\;,
\]
we see the description 
\[
\hat{s}_0 = \left\{ \begin{array}{ll} (1,\,-1) &  \mbox{symplectic } G  \\ (1,\,-2)(-1,\,2) & \mbox{orthogonal } G  \end{array}\right.\;.
\]
Observing the action of $\omega\in \Omega$ on the affine Dynkin diagram in the orthogonal case, it is easy to verify that $\hat{\omega}= (1,\,-1)$.

\begin{proposition}\label{prop:explctW_n}
For each $1\leq i \leq n_G$, the identification \eqref{eq:sgnperm} sends the subgroup $\widehat{W}_{K_i}< W_{G}$ to the subgroup
\[
W_{n_G,i} := \{\sigma\in W_{n_G}: \sigma(\{\pm1, \dots \pm i\})\subseteq \{\pm1, \dots,\pm i\}\} < W_{n_G}\;.
\]
A factorization $W_{n,i} \cong W_i \times W_{n-i}$ holds, through which we denote the subgroup
\[
W'_i \times W_{n_G-i}\cong W'_{n_G,i} < W_{n_G,i}\;.
\]
In the case of orthogonal $G$, the identification \eqref{eq:sgnperm} sends the subgroup $\widehat{W}'_{K_i}< W_{G}$ to the subgroup $W'_{n_G,i}<W_{n_G}$.

\end{proposition}
\begin{proof}
In the symplectic case, by Proposition \ref{prop:maximalcompacts}, the group $W_{K_i}$, for $1\leq i\leq n_G$, is generated by $\{\hat{s}_0\}\cup \Pi - \{s_i\}$.

In the orthogonal case, by Proposition \ref{prop:maximalcompacts}, the group $W_{K_i}$, for $2\leq i \leq n_G$, is generated by $\{\hat{\omega},\hat{s}_0\}\cup \Pi - \{s_i\}$, while $W_{K_1}$ is generated by $\{\hat{\omega}\}\cup \Pi - \{s_1\}$. The subgroups $\widehat{W}'_{K_i}$ are obtained similarly when omitting the generator $\hat{\omega}$. 

The statement now follows from our description of generators in terms of signed permutations.
\end{proof}

We also write $W_{n_G,0} = W_{n_G}$, which is then identified with $\widehat{W}_{K_0} = W_{n_G}$.

\subsection{Anti-tempered weakly-spherical spectrum}

The effective description of compact subgroups of $G$ of previous sections allows for a reduction of the study of invariants of anti-tempered $G$-representations into properties of associated Springer representations.

The following theorem will be proved in Section \ref{sec:hecke} through methods of categorical equivalences arising from affine Hecke algebras.



\begin{theorem}\label{thm:mainSpringer}
Let $I_G<K<G$ be a compact subgroup, and $\widehat{W}_K < W_G$ its corresponding subgroup from \eqref{eq:WK}. 
Let $\zeta$ be a complex group character of $G$, which is trivial on $I_G$. Let $\widehat{\zeta}_K$ be the resulting group character of $\widehat{W}_K$.
Let $\mathcal{O}^\vee\in \mathcal{U}^\vee$ and $\epsilon\in \widehat{A(\mathcal{O}^\vee)}_0$ be choices of a unipotent conjugacy class and a character, for which the anti-tempered representation
\[
\delta(1,\mathcal{O}^\vee,\epsilon)\in \Pi_{1,\mathcal{O}^\vee}
\]
possesses non-zero $I_G$-invariant vectors.
Then, $\epsilon$ is of Springer-type, and the equality of dimensions 
\[
\dim \mathrm{Hom}_{\widehat{W}_K} \left( \widehat{\zeta}_K , \Sigma(\mathcal{O}^\vee,\epsilon)\right) = \dim \mathrm{Hom}_{K} (\zeta, \delta(1,\mathcal{O}^\vee,\epsilon))
\]
holds.
\end{theorem}

\begin{corollary}\label{cor:ofsection5}
Let $\mathcal{O}^\vee\in \mathcal{U}^\vee$ and $\epsilon\in \widehat{A(\mathcal{O}^\vee)}_0$ be choices of a unipotent conjugacy class and a character, for which the anti-tempered representation
\[
\delta(1,\mathcal{O}^\vee,\epsilon)\in \Pi_{1,\mathcal{O}^\vee}
\]
possesses non-zero $I_G$-invariant vectors.
Then, for all $i=0,\ldots,n_G$, the equality of dimensions of invariant spaces
\[
\dim \Sigma(\mathcal{O}^\vee,\epsilon)^{W_{n_G,\,i}} = \dim \delta(1,\mathcal{O}^\vee,\epsilon)^{K_i}
\]
holds, when the identification $W_{G} \cong W_{n_G}$ is assumed.
In the case of orthogonal $G$, when the quasi-basic $A$-packet $\Pi_{-1,\mathcal{O}^\vee}$ is defined, the equality
\[
\dim \Hom_{W_{n_G,\,i}} (\sgn^{\pm} \boxtimes \mathrm{triv},  
 \Sigma(\mathcal{O}^\vee,\epsilon)) = \dim \delta(-1,\mathcal{O}^\vee,\epsilon)^{K_i}
\]
holds, for all $i=0,\ldots,n_G$, when the quadratic character $\sgn^{\pm} \boxtimes \mathrm{triv}$ is taken under the decomposition $W_{n_G,i} = W_i \times W_{n_G-i}$.

\end{corollary}

\begin{proof}
The first statement follows from Proposition \ref{prop:explctW_n}, when taking trivial $\zeta$ in Theorem \ref{thm:mainSpringer}.

In the orthogonal case, $\kappa_0(\omega)=-1$, since $\kappa_0$ is non-trivial.
For $0<i\leq n_G$, we have $\omega\in W_{K_i}$. Thus, Lemma \ref{lem:unram-ch} claims that $\widehat{\kappa_0}_{K_i}$ must produce a non-trivial character of $\widehat{W}_{K_i}/\widehat{W}'_{K_i}$. According to Proposition \ref{prop:explctW_n}, that claim amounts to $\hat{\chi}_{K_i} = \sgn^{\pm} \boxtimes \mathrm{triv}$ under the identification of $\widehat{W}_{K_i}$ with $W_{n_G,i}$. 
For $i=0$, $\widehat{W}'_{K_0} = \widehat{W}_{K_0}$, while both $\widehat{\kappa_0}_{K_0}$ and $\sgn^{\pm} \boxtimes \mathrm{triv}$ stand for the trivial character.

The second statement now follows similarly when taking $\zeta = \kappa_0$ in Theorem \ref{thm:mainSpringer}, and recalling that $\delta(-1,\mathcal{O}^\vee,\epsilon) \cong \kappa_0\otimes \delta(1,\mathcal{O}^\vee,\epsilon)$.
    
\end{proof}

The Green theory analysis of Springer representations in Section \ref{sect:walds-la} produces the following result.

\begin{theorem}\label{prop:combinatoricsprop}
Let $\mathcal{O}^\vee\in \mathcal{U}^\vee$ and $\epsilon\in \widehat{A(\mathcal{O}^\vee)}_0$ be choices of a unipotent conjugacy class and a character, for which the Springer representation $\sigma(\mathcal{O}^\vee,\epsilon)$ is non-zero.

In the case of symplectic $G$, an inclusion $\epsilon\in A^{\dagger}(\mathcal{O}^\vee)$ holds, if and only if, there exists an index $0\leq i\leq n_G$, for which 
\[
\Sigma(\mathcal{O}^\vee,\epsilon)^{W_{n_G,\,i}}\neq\{0\}\;.
\]

In the case of orthogonal $G$, an inclusion $\epsilon\in A^{\dagger}(\mathcal{O}^\vee)$ holds, if and only if, there exists an index $0\leq i\leq n_G$, for which 
\[
\Hom_{W_{n_G,\,i}}(\sgn^{\pm} \boxtimes \mathrm{triv}, \Sigma(\mathcal{O}^\vee,\epsilon))\neq\{0\}\;.
\]

\end{theorem}

A direct consequence of the combination of Theorem \ref{prop:combinatoricsprop} with Corollary \ref{cor:ofsection5} is that a representation in the anti-tempered $A$-packet $\Pi_{s_G, \mathcal{O}^\vee}$, for any $\mathcal{O}^\vee\in \mathcal{U}^\vee$, is weakly spherical (Definition \ref{def:weak-sphe}), if and only if, it is formally weakly spherical (Definition \ref{defi:formallyws}). The last statement is equivalent to Theorem \ref{thm:C}.

\section{Hecke algebras}\label{sec:hecke}

Our aim now is to prove Theorem \ref{thm:mainSpringer} by means of the representation theory of Hecke algebras. The key advantage of this approach is that a move into an algebraic setting allows for a rigorous limiting process $q\to1$, where $q$ is the cardinality of residue field of $F$. In the limit we recover the Springer theory setup.

In greater detail, for a maximal compact subgroup $I_G < K<G$, we describe a $\mathbb{C}[v,v^{-1}]$-algebra $\mathbf{H}$, which interpolates between $G$-representations in basic $A$-packets and Springer representations of $W_G$, in the following sense. For an anti-tempered representation $\delta\in \Pi_{\mathcal{O}^\vee}$, a $\mathbf{H}$-module $\mathbf{M}(\delta)$ is constructed, whose specialization at $v= \sqrt{q}$ describes the decomposition of $\delta|_K$, while its specialization at $v=1$ describes the decomposition of $\Sigma|_{\widehat{W_K}}$ for the corresponding Springer $W_G$-representation $\Sigma$.

However, it is not a priori clear that $\mathbf{M}(\delta)$ is unique and that its decomposition as a generic module governs both of the multiplicities which we would like to compare.
Indeed, we prove that claim in Section \ref{sec:finiteheckealgebras}, by showing that in the class of \textit{extended} Hecke algebras, to which $\mathbf{H}$ belongs, there is in fact a distinguished isomorphism between the specializations $\mathbf{H}_1,\, \mathbf{H}_{\sqrt{q}}$ (a construction which goes back to \cite{lusztigfinitehecke}).


\subsection{Algebraic reduction}

\begin{definition}
For a subgroup $I_G < H < G$, we denote the \textit{Iwahori-Hecke algebra} $\mathscr{H}_{H}$ to be the space of compactly supported functions $f:H\to \mathbb C$, for which 
    \[
    f(i_1hi_2) = f(h),\quad \forall i_1,i_2\in I_G,h\in H
    \]
holds, endowed with the convolution associative product  
    \[
    (f_1\star f_2)(h) = \int_Hf_1(x)f_2(x^{-1}h)\,dx\;.
    \]
Here, $dx$ stands for the Haar measure on $H$, normalized so that $I_G$ has volume $1$.
\end{definition}


\subsubsection{Categorical equivalences}
Let $(\pi,V)$ be a smooth representation of either $G=H$ or of a compact subgroup $I_G < H < G$. 
There is a natural action of the algebra $\mathscr{H}_H$ on the finite-dimensional space $V^{I_G}$ of $I_G$-invariant vectors, given by
\[
\phi\star v = \int_H\phi(x)\pi(x)v\, dx, \quad \phi\in \mathscr{ H}_H,\; v\in V^I\;.
\]
We denote by $\pi^{I_G}$ the resulting $\mathscr{H}_H$-module.
Let $\Rep_0(H)$ be the full subcategory of smooth complex $H$-representations that are generated by their $I_G$-invariant vectors.
We also write $\mathrm{Mod}( \mathscr{H}_H)$ for the category of $\mathscr{H}_H$-modules.
    %
    %
    %
The following foundational result, for its case of $G=H$, is often taken as the impetus for the study of Iwahori-Hecke algebras.

\begin{theorem}\label{thm:borel}
For either $G=H$ or a compact subgroup $I_G < H <G$, the functor 
\[
M^H:\Rep_0(H) \to \mathrm{Mod}(\mathscr{H}_H)
\]
given by $\pi \mapsto \pi^{I_G}$ is an equivalence of abelian categories.
In particular, for each irreducible representation $\pi$ in $\Rep_0(H)$, $\pi^{I_G}$ is an irreducible $\mathscr{H}_H$-module.
Moreover, for a compact subgroup $I_G<K<G$, when naturally embedding $\mathscr{H}_K$ as a subalgebra of $\mathscr{H}_G$, the diagram 
    \begin{equation}
        \begin{tikzcd}
            \Rep_0(G) \arrow[r,"M^G"] \arrow[d,"\res_K^G"] & \mathrm{Mod}(\mathscr{H}_G) \arrow[d,"\res_{\mathscr{H}_K}^{\mathscr{H}_G}"] \\
            \Rep_0(K) \arrow[r,"M^K"] & \mathrm{Mod}(\mathscr{H}_K)
        \end{tikzcd}
    \end{equation}
    commutes.

\end{theorem}


\begin{proof}
    When $H=G$ this is due to \cite[Corollary 4.11]{Bo}.
    When $H$ is compact this follows from the semisimplicity of both categories. The commutation of the diagram is immediate from definitions.
\end{proof}

\subsubsection{Generic Hecke algebras}

Let $A = \CC[v,v^{-1}]$ be the ring of Laurent polynomials.

\begin{definition}\label{def:genHecke}
Let $(W_1,S_1)$ be a Coxeter system, and $\Omega$ be a finite group equipped with an action on $W_1$ by group automorphisms that stabilize $S_1$.
Let $W = W_1\rtimes \Omega$ be the resulting group extension.
Let $\ell:W\to \mathbb N_0$ be the length function, that extends the Coxeter length on $W_1$ through $\ell|_{\Omega}=0$.

 The \textit{extended generic Hecke algebra} for $W$ is defined to be the $A$-algebra $\mathbf{H}^W$, presented with the basis $\{T_w:w\in W\}\subset \mathbf{H}^W$, as a free $A$-module, that is subject to the multiplicative relations
    \begin{equation}\label{eq:relations2}
        \begin{aligned}
        &T_w\cdot T_{w'}=T_{ww'}, \qquad \text{if }\ell(ww')=\ell(w)+\ell(w'),\\
        &T_s^2=(v^2-1) T_s+v^2,\qquad s\in S_1.
        \end{aligned}
    \end{equation}
\end{definition}
For $a\in \mathbb C^\times$, we let $\theta_a:A\to \mathbb C$ denote the ring homomorphism determined by $v\mapsto a$ and write $\mathbb C_a$ for the $1$-dimensional $A$-module determined by $\theta_a$.
In the context of Definition \ref{def:genHecke}, we write $\mathbf H^W_a= \mathbf H^W \otimes_A \mathbb C_a$ for the specialized complex algebra.
Given a $\mathbf H^W$-module $\mathbf M$ which is free over $A$, we also write $\mathbf M_a=\mathbf M\otimes_A \mathbb C_a$ for the specialized $\mathbf{H}^W_a$-module.
Considering the group algebra $\mathbb{C}\left[ W\right]$ with its natural basis $\{\delta_w\}_{w\in W}$, an isomorphism of complex associative algebras
\[
\mathbf{H}^W_1 \; \cong\; \mathbb{C}\left[ W\right]
\]
sends  $T_w\otimes 1$ to $\delta_w$, for all $w\in W$.
Now, let us note that the Iwahori decomposition of \eqref{eq:iwah-decomp} gives a basis 
\[
\{f_{w}\}_{w\in \widetilde{W_G}}\subset \mathscr{H}_G\;,
\]
for the Iwahori-Hecke algebra, which is parameterized by elements of the Iwahori-Weyl group.
Here, $f_w$ is the characteristic function of the double coset $I_G w I_G \subset G$.

\begin{proposition}\label{prop:generic-affine}\cite[Theorem 3.3, Corollary 3.6]{IM}
Decomposing the Iwahori-Weyl group $\widetilde{W_G} = W_a\rtimes \Omega$ as an extended Coxeter group, produces an isomorphism 
\[
\begin{array}{ccc} \mathbf{H}^{\widetilde{W_G}}_{\sqrt{q}} &\; \cong\;& \mathscr{H}_G  \\  T_w\otimes 1 &  \leftrightarrow & f_w \end{array}\;,
\]
of complex associative algebras.
For any compact subgroup $I_G<K<G$, the decomposition 
\[
W_K = W'_K \rtimes \Omega_K < \widetilde{W_G}\;,
\]
as in Theorem \ref{thm:subgroups} corresponds to an algebra embedding of $\mathbf{H}^{W_K}$ into $\mathbf{H}^{\widetilde{W_G}}$.
The specialization isomorphism sends the resulting sub-algebra $\mathbf{H}^{W_K}_{\sqrt{q}} < \mathbf{H}^{\widetilde{W_G}}_{\sqrt{q}}$ to the sub-algebra $\mathscr{H}_K < \mathscr{H}_G$.

\end{proposition}

\subsubsection{Kazhdan--Lusztig construction}

Springer representations of the Weyl group $W_G$ become relevant for our discussion of weak Arthur packets, when recalling the Kazhdan--Lusztig construction of the local Langlands correspondence for irreducible representations in the principal Bernstein block $\Rep_0(G)$.
Their approach gives a Springer-type geometric realization of affine Hecke algebra modules. A particular convenience is that the realization uses anti-tempered representation as its building blocks in terms of parabolic induction. 
We convey the parts of this theory that are needed for our discussion in the following proposition.

\begin{proposition}\label{prop:kl-antitemp}
Let $\mathcal{O}^\vee\in \mathcal{U}^\vee$ be a unipotent conjugacy class in $\mathbf{G}^\vee(\CC)$.
Let $\pi\in \Pi_{1,\mathcal{O}^\vee}$ be an anti-tempered representation, for which $\pi^{I_G}\neq0$.
Then, there exist a character $\epsilon_{\pi}\in \widehat{A(\mathcal{O}^\vee)}$ and an $A$-free $\mathbf{H}^{\widetilde{W_G}}$-module $\mathbf{M}(\pi)$, for which both isomorphisms
\[
\mathbf{M}(\pi)_{\sqrt{q}} \cong \pi^{I_G}\,,\qquad \mathbf{M}(\pi)_{1} \cong \overline{\Sigma(\mathcal{O}^\vee,\epsilon_{\pi})}
\]
hold, as $\mathscr{H}_G$-modules, and, respectively, $\widetilde{W_G}$-representations, under the identifications of Proposition \ref{prop:generic-affine}.

Moreover, the map $\pi\mapsto \epsilon_{\pi}$ from $\Pi_{1,\mathcal{O}^\vee}\cap \Rep_0(G)$ to $\widehat{A(\mathcal{O}^\vee)}$ is injective and its image consists precisely of the local systems of Springer-type.

\end{proposition}
\begin{proof}
    The analogous result was established in generality for tempered representations of split adjoint groups by Reeder in \cite[Theorem 8.1]{reederhecke} using Kazhdan and Lusztig's \cite{KL} construction of tempered representations.
    Twisting by the Aubert-Zelevinsky involution gives the required result for anti-tempered representations. 

    For groups of arbitrary isogeny we combine the result for adjoint groups with \cite[Lemma 5.3.1]{reederisogenies}.

\end{proof}

\begin{assumption}\label{assumpt:param}
For any $\mathcal{O}^\vee\in \mathcal{U}^\vee$ and $\epsilon\in \widehat{A(\mathcal{O}^\vee)}_0$, such that the anti-tempered representation
\[
\delta = \delta(1,\mathcal{O}^\vee,\epsilon)\in \Pi_{1,\mathcal{O}^\vee}\;,
\]
as parameterized in Section \ref{sect:formal-wsph}, satisfies $\delta^{I_G}\neq0$, one also has $\epsilon_{\delta}= \epsilon$, where $\epsilon_{\delta}$ is the character provided by Proposition \ref{prop:kl-antitemp}.

In other words, the irreducible local systems, that are used in the Kazhdan--Lusztig parameterization of anti-tempered irreducible representations with non-zero $I_G$-invariant vectors, agree with the characters visible in Arthur's endoscopic parameterization. 
\end{assumption}

Assumption \ref{assumpt:param} is a consequence of \cite{waldspurgerendoscopy} for the case of odd orthogonal $G$. Indeed, it was shown that the Kazhdan--Lusztig parameterization satisfies the endoscopic identities that pin Arthur's characters. We expect a similar correspondence to hold in the symplectic case.

\subsection{Deformations for extended finite Hecke algebras}\label{sec:finiteheckealgebras}
This current section is self contained and is dedicated to general deformation procedures for extended finite Hecke algebras, that culminate in Proposition \ref{prop:deformfinite}.

We now take a group $W = W_1\rtimes \Omega$ with the length function $\ell:W\to \mathbb N_0$ as in Definition \ref{def:genHecke}, with $W$ assumed to be \textit{finite}.
For brevity we write $\mathbf{H} = \mathbf{H}^W$.
    %
    %

The key ingredient to prove the well-definedness of deformations of modules is the existence of an algebraic family of isomorphisms $\eta_a:\mathbb C[W]\to \mathbf H_a$ indexed by $a$ in some Zariski open subset of $\mathbb C$ containing $\sqrt q$ and $1$, with $\eta_1 = \id$.
The explicit form of these maps is in general quite complicated, and their existence is highly non-trivial.
We turn to Lusztig's $J$-algebra which provides a convenient formalism for producing such maps.
This formalism contains enough information to extract the abstract properties we need, but it should be noted that these rely on deep properties of the J-algebra established by Lusztig in \cite{cells1,cells2}.

\begin{definition}
    Let $\mathbf J$ denote Lusztig's asymptotic Hecke algebra for $W$ over $A$. 
    This is a free $A$-module with basis $\{t_w: w\in W\}$.
    We refer to \cite[\S 2.2]{geck} for the definition of the multiplicative structure of $\mathbf J$.
    Let $\phi:\mathbf H\to\mathbf J$ denote the injective homomorphism from the Hecke algebra to the asymptotic Hecke algebra as defined in \cite[\S2.2]{geck} due to Lusztig.

\end{definition}

    We write $\mathbf H_a,\mathbf J_a,\phi_a$ for $\mathbf H\otimes_A \mathbb C_a, \mathbf J\otimes_A \mathbb C_a, \phi\otimes_A\mathbb C_a$.

\begin{proposition}
    The map $\phi_1:\mathbb C[W] \to \mathbf J_1$ is an isomorphism and $\mathbf J_1\otimes_{\mathbb C}A \cong \mathbf J$ as algebras.
\end{proposition}
\begin{proof}
    The first part follows from \cite[Example 2.6]{geck}.
    The second part follows from the fact that the structure constants of $J$ lie in $\mathbb C$ (in fact in $\mathbb Z$).
\end{proof}
The upshot of the proposition is that an isomorphism $\psi = \phi_1\otimes_{\mathbb C}A$ between $A[W]$ and $\mathbf J$ is obtained. We see the diagram
\begin{equation}
    \begin{tikzcd}
        A[W] \arrow[r,"\psi"] & \mathbf J \\
        & \mathbf H \arrow[u,"\phi"]
    \end{tikzcd}\;.
\end{equation}
We would like to invert the vertical arrow to get a map from $A[W]\to \mathbf H$. 
However the vertical map is not in general an isomorphism.
If we think of $\phi$ as an algebraic family of homomorphisms indexed over $\mathbb C^\times$, this corresponds to the fact that there are $a\in \mathbb C^\times$ for which $\phi_a$ is not an isomorphism. 
The next proposition identifies a Zariski basic open subset which avoids such points.
\begin{lemma}\label{lem:poincare}
    Let 
    $$g(q) = \sum_{w\in W}q^{l(w)}$$
    be the Poincar\'e polynomial for $W$ and $f(v) = g(v^2)$.
    Then $\phi_a$ is an isomorphism whenever $f(a)\ne 0$.
\end{lemma}
\begin{proof}
    By \cite[Corollary 2.5]{geck}, it suffices to show that $\mathbf H_a$ is semisimple whenever $f(a)\ne 0$.
    When $\Omega$ is trivial, this is a consequence of \cite{gyoja}.
    The proof in the general case is identical.
    The only point of care is to ensure that the generic degrees are polynomials (c.f. \cite[6.(iii)]{gyoja}) which follows immediately from the analysis of Schur elements in \cite[Section 4.8]{geck} (see \cite[Section 8.1.8]{geckpfeiffer} for the relation between Schur elements and generic degrees).
\end{proof}

Let $\phi_f:\mathbf H_f\to \mathbf J_f$ denote the localisation of $\phi$ at $f\in A$.
\begin{corollary}
    The homomorphism $\phi_f$ is an isomorphism.
\end{corollary}
\begin{proof}
    Since $\phi$ is injective, so is $\phi_f$.
    To check surjectivity, it suffices to do so at every maximal ideal $\mathfrak m$ of $A_f$.
    Let $a$ be the point in $\mathbb C^\times$ with $f(a)\ne 0$ corresponding to $\mathfrak m$.
    Since $(A_f)_{\mathfrak m}$ is a local ring, by Nakayama's lemma $(\phi_f)_{\mathfrak m}$ is surjective if $\phi_f \otimes_{A_f} A_f/\mathfrak m = \phi_a$ is surjective.
    Lemma \ref{lem:poincare} implies $\phi_a$ is surjective and this completes the proof.
\end{proof}

Thus we see that on the open set $U = \mathbb C^\times \setminus\{a:f(a) = 0\}$, we have an algebraic collection of isomorphisms 
\[
\eta_a:\mathbb C[W]\to \mathbf H_a\,\quad a\in U
\]
arising from specializations of 
\[
\eta:A_f[W]\to \mathbf H_f, \quad \eta := \phi_f^{-1}\circ \psi_f\;.
\]
%

The next proposition proves the well-definedness of deformations.

\begin{proposition}\label{prop:deformfinite}

\begin{enumerate}
    \item We have $\sqrt{q}\in U$, and the isomorphism 
\[
\eta_{\sqrt q}:\mathbb C[W] \to \mathbf H_{\sqrt q}
\]
of complex algebras is well-defined.

\item Let $\mathbf{M}$ be an $A$-free $\mathbf{H}$-module. Let $\mathbf{M}_{\sqrt q\to 1}$ denote the $W$-representation obtained by pulling back the module structure of $\mathbf{M}_{\sqrt{q}}$ along the isomorphism $\eta_{\sqrt q}$.
Then, an isomorphism of $W$-representations
\[
\mathbf{M}_1 \cong \mathbf{M}_{\sqrt q\to 1}
\]
holds.

\end{enumerate}

\end{proposition}
\begin{proof}
The first statement is a consequence of the polynomial in Lemma \ref{lem:poincare} having positive coefficients.

    Let $\rho_f:\mathbf H_f\to \mathrm{End}_{A_f}(\mathbf M_f)$ be the homomorphism induced from the $\mathbf H_f$-module structure of $\mathbf M_f$ and set
    $$\chi_{\mathbf M_f}(w) := \mathrm{tr}(\rho_f(\eta(T_w))) \quad \in A_f.$$
    Similarly, for $a\in U$ we define $\rho_a:\mathbf H_a\to \mathrm{End}_{\mathbb C}(\mathbf M_a)$ and
    $$\chi_{\mathbf M_a}(w) := \mathrm{tr}(\rho_a(\eta_a(T_w))) \quad \in \mathbb C.$$
    Since $(\mathbf H_f)_a \cong \mathbf H_a$ and $(\mathbf M_f)_a \cong \mathbf M_a$ we have
    \begin{equation}
        \label{eq:chareval}
        \chi_{\mathbf M_a}(w) = \theta_a(\chi_{\mathbf M_f}(w)).
    \end{equation}
    Since $U$ is $\mathbb C^\times$ minus finitely many points, it is path connected so there is a continuous $\gamma:[0,1]\to U$ such that $\gamma(0) = 1, \gamma(1) = \sqrt q$ (in fact the straight line path will work).
    By \eqref{eq:chareval}, $\chi_{\mathbf M_{\gamma(t)}}$ defines a continuous (in fact rational with denominator a power of $f$) family of characters of $W$ and hence it must be constant.
    Since $\eta_1 = \id$, we have that $\chi_{\mathbf M_1}$ is the character for $\mathbf M_1$ and so the result follows.
\end{proof}

\subsection{Proof of Theorem \ref{thm:mainSpringer}}

Let $I_G<K<G$ be a compact subgroup and $\zeta$ a complex character of $G$, which is trivial on $I_G$.
We also pick an anti-tempered representation
\[
\delta= \delta(1,\mathcal{O}^\vee,\epsilon)\in \Pi_{1,\mathcal{O}^\vee}
\]
with $\delta\in \Rep_0(G)$, as in the statement of Theorem \ref{thm:mainSpringer}.
By Theorem \ref{thm:borel}, it is enough to prove that an equality
\[
\dim \mathrm{Hom}_{\widehat{W}_K} \left( \widehat{\zeta}_K , \Sigma(\mathcal{O}^\vee,\epsilon)\right) = \dim \mathrm{Hom}_{\mathscr{H}_K} (\zeta, \delta^{I_G})
\]
holds.
Let $\mathbf{H}:=\mathbf{H}^{W_K} < \mathbf{H}^{\widetilde{W_G}}$ be the sub-algebra described in Proposition \ref{prop:generic-affine}.
By Proposition \ref{prop:deformfinite}, it is now enough to find $A$-free $\mathbf{H}$-modules $\mathbf{\Sigma}, \boldsymbol{\zeta}$, for which the identities
\begin{equation}\label{eq:desid1}
\mathbf{\Sigma}_{\sqrt{q}} \cong \delta^{I_G},\quad \boldsymbol{\zeta}_{\sqrt{q}} \cong \zeta
\end{equation}
would hold as $\mathscr{H}_K$-modules, while 
\begin{equation}\label{eq:desid2}
\mathbf{\Sigma}_{1} \cong \mathrm{res}^{\widetilde{W_G}}_{W_K}\overline{\Sigma(\mathcal{O}^\vee,\epsilon)},\quad \boldsymbol{\zeta}_{1} \cong \zeta
\end{equation}
are holding as $W_K$-representations. Since $\delta^{I_G}\not=\{0\}$, that would also imply that $\mathbf{\Sigma}_1$ is non-zero. By Remark \ref{rem:shoji}, we would then see that $\epsilon$ is of Springer-type.

Indeed, let $\mathbf{M}(\delta)$ be the $\mathbf{H}^{\widetilde{W}_G}$-module supplied by Proposition \ref{prop:kl-antitemp}. Then, according to Assumption \ref{assumpt:param}, $\mathbf{\Sigma}: = \mathrm{res}^{\mathbf{H}^{\widetilde{W_G}}}_{\mathbf{H}} \mathbf{M}(\delta)$ produces the desired module.
The homomorphism $\boldsymbol{\zeta}:\mathbf{H}\to A$ is then produced by the formula
\[
T_w\mapsto \zeta(w)v^{2\ell(w)}, \quad w\in W_K\;.
\]%

\section{Decomposition of Springer representations}\label{sect:walds-la}
\label{sec:springer}
The goal of this section is to prove Theorem \ref{prop:combinatoricsprop}.
We take note (Definition \ref{def:weaks}) of two sequences $E^{s}_0,\ldots, E^{s}_n$, $s\in \{\pm1\}$, of irreducible $W_n$-representations, and declare a $W_n$-representation to be \textit{weakly $s$-spherical}, whenever it admits one of $\{E^s_i\}_{i=0}^n$ as a sub-representation.
Lemma \ref{lem:goodparity} then reduces Theorem \ref{prop:combinatoricsprop} to the following proposition.

\begin{proposition}\label{prop:dagger}
    Let $\lambda\in \mathcal{P}^{s_G}_0(N_G)$ be a partition (of good parity), and $\epsilon\in \widehat{A(\mathcal{O}_{\lambda}^\vee)}_0$ a character.
    Then, an inclusion $\epsilon\in A^\dagger(\mathcal{O}^\vee_\lambda)$ holds, if and only if, the Springer representation $\Sigma(\mathcal O^\vee_\lambda,\epsilon)$ is weakly $s_G$-spherical.
\end{proposition}

Singling out the irreducible constituents of Springer representations $\Sigma(\mathcal O^\vee,\epsilon)$, beyond the top degree case of $\sigma(\mathcal O^\vee,\epsilon)$, is in general a  problem of high complexity.
The analogous problem for $S_n$-representations, when the group is viewed as a Weyl group of type $A_{n-1}$, amounts to computations of the ubiquitous Kostka numbers.
%
%
In our case of $W_n$ (Lie types $B$ and $C$), the multiplicities in question are governed by the so-called \textit{double Kostka polynomials}. These are generalizations of the Kostka polynomials introduced by Shoji in \cite{shoji} and studied notably in \cite{archarhenderson,la,shojiliu,waldspurger,}.
For the problem at hand we will utilize recent deep results on special values of double Kostka polynomials obtained by Waldspurger for type $C$, and later generalised to types $B$ (and $D$) by La.
These results are summarised in Theorems \ref{thm:waldsla} and \ref{thm:maximal}. They provide an explicit algorithmic calculus, that may be used to determine weak $s$-sphericity of $\Sigma(\mathcal O^\vee,\epsilon)$.
Such methods are exploited in the final subsection to establish Proposition \ref{prop:dagger}.

\subsection{First reduction}

\subsubsection{Representation theory of $S_n$ and $W_n$}\label{sect:repSn}


We write $S_n$ for the symmetric group on $n$ letters.
For a partition $\lambda \in \mathcal P(n)$, we write $V_\lambda\in \irr(S_n)$ for the Specht module that is determined by $\lambda$.

\begin{proposition}
    \label{prop:pieri}
    (Pieri's formula) Let $0\le k\le n$ and $\lambda \in \mathcal P(n-k)$ be given.
    Then, a decomposition
    $$\ind_{S_{n-k}\times S_{k}}^{S_n} V_{\lambda} \boxtimes V_{(k)} = \bigoplus_\mu V_\mu$$
    holds, so that the sum ranges over all partitions $\mu\in \mathcal{P}(n)$, whose Young diagram can be obtained from that of $\lambda$ by adding $k$ boxes, no two in the same column.
\end{proposition}

    Irreducible representations of the group of signed permutations $W_n$ are parameterized by (ordered) pairs of partitions $\alpha,\beta\in \mathcal{P}$ with $|\alpha|+|\beta| = n$.
It will be useful to adopt the direct notation $(\alpha,\beta)\in \irr(W_n)$.
Let us now recall the construction of those representations out of the Specht representations for $S_n$.

    Let $\chi_k$ denote the character of $(\mathbb Z/2)^n$ that is trivial on the first $n-k$ components, and non-trivial on the remaining $k$ components.
    The centralizer of $\chi_k$ in $S_n$ is the naturally embedded subgroup $S_{n-k}\times S_k$.
    Let $(V_\alpha\boxtimes V_\beta)\otimes \chi_{|\beta|}$ denote the representation of the group $(S_{|\alpha|}\times S_{|\beta|})\ltimes (\mathbb Z/2)^n \cong W_{|\alpha|}\times W_{|\beta|}$, obtained by extending the $S_{|\alpha|}\times S_{|\beta|}$-representation $V_\alpha\boxtimes V_\beta$ by $\chi_{|\beta|}$ on $(\mathbb Z/2)^n$.
    The induced representation
    $$(\alpha,\beta) := \ind_{W_{|\alpha|}\times W_{|\beta|}}^{W_n}(V_\alpha\boxtimes V_\beta)\otimes \chi_{|\beta|}$$
    is then irreducible.
Indeed, every isomorphism class of an irreducible $W_n$-representation is obtained uniquely in this manner.
As general notation, for complex finite-dimensional representations $\pi_1,\pi_2$ of a finite group, we write
\[
\langle\pi_1,\pi_2\rangle: = \dim_{\mathbb{C}} \Hom (\pi_1,\pi_2)\;.
\]

\begin{proposition}\footnote{Our conventions differ from those in \cite{geisskinch} by a flip. 
    Namely, our representation $(\alpha,\beta)$ is denoted by $\{\beta,\alpha\}$ in \cite{geisskinch}.}
    \label{prop:gk}
    \cite[Theorem III.2]{geisskinch}
If $(\alpha_1,\beta_1)\in \mathrm{Irr}(W_i),(\alpha_2,\beta_2)\in \mathrm{Irr}(W_{n-i})$ and $(\alpha,\beta)\in \mathrm{Irr}(W_n)$, then, viewing $W_i\times W_{n-i}$ as a subgroup of $W_n$ via $W_i\times W_{n-i}\cong W_{n,i}$, we have
\begin{align} \label{eq:res}
    &\langle\ind_{W_i\times W_{n-i}}^{W_n}\left((\alpha_1,\beta_1))\boxtimes (\alpha_2,\beta_2)\right) , (\alpha,\beta)\rangle \\
    =&\langle\ind_{S_{|\alpha_1|}\times S_{|\alpha_2|}}^{S_{|\alpha|}}( V_{\alpha_1}\boxtimes V_{\alpha_2}), V_\alpha\rangle \cdot \langle \ind_{S_{|\beta_1|}\times S_{|\beta_2|}}^{S_{|\beta|}}(V_{\beta_1}\boxtimes V_{\beta_2}), V_\beta\rangle \nonumber\;,
\end{align}
when $|\alpha_1|+|\alpha_2| = |\alpha|$ and $|\beta_1|+|\beta_2| = |\beta|$.
%
\end{proposition}

\subsubsection{Weakly $s$-spherical Weyl group representations}


\begin{definition}
    For an integer $0\le i\le n$ and a sign $s\in \{\pm1\}$, we denote
    \begin{equation}
        E^{s}_i = \begin{cases}
            ((n-i)\cup(i),\emptyset) & \mbox{if $s=1$}\\
            ((n-i),(i)) & \mbox{if $s=-1$}
        \end{cases}\;\in \irr(W_n)\;.
    \end{equation}
\end{definition}
Note, that $E_0^1 = E_0^{-1} = ((n),\emptyset)$ is the trivial $1$-dimensional $W_n$-representation, while $E_n^{-1} = (\emptyset, (n))$ corresponds to the character $\sgn^\pm$ of Section \ref{sect:sgn-perm}.

\begin{proposition}
    \label{prop:firstreduction}
    Let $\pi$ be a complex $W_n$-representation.
    
    Then,
    \begin{enumerate}
        \item There is an index $0\le i_1\le n$ with  $\pi^{W_{n,i_1}}\ne 0$, if and only if, there is an index $0\le i_2\le n$ with
        $$\langle\pi,E^1_{i_2}\rangle \ne 0\;.$$

        \item For any $0\le i \le n$, an equality
        $$\dim \Hom_{W_{n,i}}(\sgn^\pm\boxtimes\triv,\pi) = \langle \pi,E^{-1}_i\rangle$$
        holds.
    \end{enumerate}
\end{proposition}
\begin{proof}
    \begin{enumerate}
        \item 

    Let us write $H_i = \ind_{W_{n,i}}^{W_{n}}\triv$.
    By Frobenius reciprocity,
    $$\dim \pi^{W_{n,i}} =\langle \triv, \res_{W_{n,i}}^{W_{n}}\pi\rangle = \langle H_i,\pi\rangle.$$
    Recalling that $W_{n,i} \cong W_i\times W_{n-i}$, we see that
    
    %
    \begin{align*}
        H_i &= \ind_{W_i\times W_{n-i}}^{W_{n}} ((i),\emptyset)\boxtimes ((n-i),\emptyset)\\
        &=\bigoplus_{\alpha\in \mathcal P(n)}\langle\ind_{S_i\times S_{n-i}}^{S_n}(V_{(i)}\boxtimes V_{(n-i)}),V_\alpha\rangle(\alpha,\emptyset), \quad \text{(Proposition \ref{prop:gk})} \\
        &=\bigoplus_{j\le \min(i,n-i)} ((n-j,j),\emptyset) = \bigoplus_{j\le \min(i,n-i)}E^1_j, \quad \text{(Proposition \ref{prop:pieri})}.
    \end{align*}
    %
    
    \item

    The statement follows similarly, when noting that
    \[
\ind_{W_{n,i}}^{W_{n}}\sgn^\pm\boxtimes \triv =  \ind_{W_{n,i}}^{W_{n}} (\emptyset,(i))\boxtimes ((n-i),\emptyset) = ((n-i),(i)) = E^{-1}_i
\]
holds, by Proposition \ref{prop:gk}.

    \end{enumerate}

\end{proof}

\begin{definition}\label{def:weaks}
    For $n\ge0$ and $s\in \{\pm1\}$ call a $W_n$-representation $\pi$ \textit{weakly $s$-spherical}, whenever an index $0\le i\le n$ exists, so that $\langle \pi,E^s_i\rangle \ne0$.
\end{definition}
\begin{lemma}
    \label{lem:weaksspherical}
    Let $n,m\ge 0$.
    Then a representation $\pi$ of $W_n$ is weakly $s$-spherical, if and only if,  $\ind_{W_{n}\times S_m}^{W_{n+m}}\pi\boxtimes \triv$ is weakly $s$-spherical.
\end{lemma}
\begin{proof}
    It suffices to prove the case when $\pi$ is irreducible so suppose $\pi= (\alpha,\beta)$.
    We have
    \begin{equation}
        \label{eq:inds}
        \mathrm{ind}_{W_n\times S_m}^{W_{n+m}}(\alpha,\beta)\boxtimes \triv = \mathrm{ind}_{W_{n}\times W_m}^{W_{n+m}}\mathrm{ind}_{W_{n}\times S_m}^{W_{n}\times W_m}(\alpha,\beta)\boxtimes \triv.
    \end{equation}
    Now a direct computation shows that
    $$\mathrm{ind}_{S_n}^{W_n}\triv = \bigoplus_i ((i),(n-i)).$$
    Thus equation \ref{eq:inds} is equal to
    \begin{equation}
        \label{eq:inds2}
        \bigoplus_{i=0}^m\mathrm{ind}_{W_{n}\times W_m}^{W_{n+m}}(\alpha,\beta)\boxtimes ((i),(m-i)).
    \end{equation}
    By Proposition \ref{prop:gk} and Pieri's formula, for every constituent $(\gamma,\delta)$ of equation \ref{eq:inds2}, $\gamma$ and $\delta$ have at least as many parts as $\alpha$ and $\beta$ respectively.
    Let us now consider the cases $s=1$ and $s=-1$ separately.
    
    If $s=1$, $(\alpha,\beta)$ is not weakly $s$-spherical, if and only if,  $l(\alpha)\ge 3$ or $l(\beta)\ge 1$. 
    Thus if $(\alpha,\beta)$ is not weakly $s$-spherical then neither is any constituent $(\gamma,\delta)$ of equation \ref{eq:inds2}.
    Conversely, if $(\alpha,\beta)$ is weakly $s$-spherical, it is of the form $((n-j,j),\emptyset)$ for some $j$, and so by Proposition \ref{prop:gk} and Pieri's formula, equation \ref{eq:inds2} contains $((n+m-j,j),\emptyset)$.
    Thus equation \ref{eq:inds2} has a weakly $s$-spherical constituent.
    
    If $s=-1$, $(\alpha,\beta)$ is not weakly $s$-spherical, if and only if,  $l(\alpha)\ge2$ or $l(\beta)\ge2$. 
    Thus if $(\alpha,\beta)$ is not weakly $s$-spherical then neither is any constituent $(\gamma,\delta)$ of equation \ref{eq:inds2}.
    Conversely, if $(\alpha,\beta)$ is weakly $s$-spherical, it is of the form $((n-j),(j))$ for some $j$, and so by Proposition \ref{prop:gk} and Pieri's formula, equation \ref{eq:inds2} contains $((n+m-j),(j))$.
    Thus equation \ref{eq:inds2} has a weakly $s$-spherical constituent.
\end{proof}

\begin{lemma}\label{lem:goodparity}
    Theorem \ref{prop:combinatoricsprop} follows from Proposition \ref{prop:dagger}.
\end{lemma}
\begin{proof}

 Let $\lambda\in \mathcal{P}^{s_G}(N_G)$ be a partition, and $\epsilon\in \widehat{A(\mathcal{O}_{\lambda}^\vee)}_0$ a character.
By Proposition \ref{prop:firstreduction}, we see that Theorem \ref{prop:combinatoricsprop} would follow, once we show that the inclusion $\epsilon\in A^\dagger(\mathcal{O}^\vee_\lambda)$ is equivalent to the representation $\Sigma(\mathcal O^\vee_\lambda,\epsilon)$ being weakly $s_G$-spherical.
Thus, we are left with the reduction to the case of $\lambda$ of good parity (i.e. in $\mathcal{P}^{s_G}_0(N_G)$).
Let $\lambda = \lambda^{gp} \cup \lambda^{bp} \cup \lambda^{bp}$ be the unique decomposition, so that $\lambda^{gp}\in \mathcal{P}^{s_G}_0$ and $\lambda^{bp}\in \mathcal{P}^{-s_G}_0$. We write $\lambda^{bp} = (\lambda_1^{bp},\dots,\lambda_l^{bp})$.
Consider the groups
    $$L_1 = G_{N_G-2|\lambda^{bp}|}, \quad L_2 = \prod_{i=1}^l\mathrm{GL}_{\lambda_i^{bp}}(F), \quad L = L_1\times L_2\;,$$
so that $L$ naturally embeds as a Levi subgroup of $G$, while the Weyl group of $L$ is given as
    $$W_L = W_{L_1}\times \prod_{i=1}^{l}S_{\lambda_i^{bp}}.$$
    Following Propositions \ref{prop:comp-comb} and \ref{prop:quotient}, the character groups $\widehat{A(\mathcal O^\vee_\lambda)}$ and $\widehat{A(\mathcal O^\vee_{\lambda^{gp}})}$ are naturally identified, and under the identification $A^\dagger(\mathcal O^\vee_\lambda)$ matches $A^\dagger(\mathcal O^\vee_{\lambda^{gp}})$.
    Viewing $\epsilon$ as a character of both component groups, we have the identity
    $$\mathrm{ind}_{W_{L_1}\times W_{L_2}}^{W_G}\Sigma(\mathcal O^\vee_{\lambda^{gp}},\epsilon)\boxtimes \triv = \Sigma(\mathcal O^\vee_\lambda,\epsilon)$$
    of  \cite[Proposition 3.3.3]{reeder_2001}.
    By doing the induction in stages, amalgamating one $S_{\lambda_i^{bp}}$ factor into the $W_{L_1}$ term at a time, and applying Lemma \ref{lem:weaksspherical} at each stage, we see that $\Sigma(\mathcal O^\vee_\lambda,\epsilon)$ is weakly $s_G$-spherical, if and only if, $\Sigma(\mathcal O^\vee_{\lambda^{gp}},\epsilon)$ is weakly $s_G$-spherical.
\end{proof}


\subsection{Algorithmic Green theory}\label{sect:algo-green}
Green theory concerns the decomposition of Springer representations $\Sigma(\mathcal O^\vee_\lambda,\epsilon)$ into irreducible representations.
We would like to recall some of the details of the algorithmic approach to this task that was presented in \cite{waldspurger} and \cite{la}.


\subsubsection{Combinatorics of the Springer correspondence}\label{sect:spring-comb}

Let $\lambda\in \mathcal{P}^{s_G}_0(N_G)$ be a fixed partition of good parity.
Let $\epsilon\in \widehat{A(\mathcal{O}^\vee_{\lambda})}_0$ be a fixed character, for which $\sigma(\mathcal{O}_{\lambda}^\vee, \epsilon)$ is an irreducible $W_G$-representation.
Identifying $W_G\cong W_{n_G}$, we may write 
\[
\sigma(\mathcal{O}_{\lambda}^\vee, \epsilon) \cong (\alpha_{\lambda,\epsilon},\beta_{\lambda,\epsilon})\;,
\]
for partitions $\alpha_{\lambda,\epsilon},\beta_{\lambda,\epsilon}\in \mathcal{P}$ with $|\alpha_{\lambda,\epsilon}|+|\beta_{\lambda,\epsilon}| = n_G$.
Let us now write
\[
\lambda = (\lambda_1 \geq \lambda_2 \geq\ldots \geq \lambda_{\ell(\lambda)}>0)\;,
\]
while harmlessly contrasting the ascending notation of Section \ref{sec:partitions}.
Through the identification of Proposition \ref{prop:comp-comb} we view $\epsilon$ as a boolean function on $S(\lambda)$. 
The character $\epsilon$ now gives rise to the following function
\[
\overline{\epsilon}: \{1,\ldots,\ell(\lambda)\}\;\to \; \{\pm1\}\;,\qquad
\overline{\epsilon}(i) =  (-1)^{\epsilon(\lambda_i) + i-1}\;,
\]
Let us decompose the index set as
\[
\{1,\,\ldots, \ell(\lambda)\} = \{e^1_1 < \ldots < e^1_t\} \sqcup \{e^{-1}_1 < \ldots < e^{-1}_{t'}\}\;,
\]
so that $\overline{\epsilon}(e^u_j) = u$, for $u\in \{\pm1\}$ and all possible indices $j$.

Let $\gamma^{\lambda,\epsilon}_1,\ldots, \gamma^{\lambda,\epsilon}_{\ell(\lambda)}\geq0$ be the unique integers, for which
\[
\alpha_{\lambda,\epsilon} = \left(\gamma^{\lambda,\epsilon}_{e^1_1}\geq \ldots \geq \gamma^{\lambda,\epsilon}_{e^1_t}\right)\,,\quad 
\beta_{\lambda,\epsilon} = \left(\gamma^{\lambda,\epsilon}_{e^{-1}_1}\geq \ldots \geq \gamma^{\lambda,\epsilon}_{e^{-1}_{t'}}\right)
\]
holds.

Indeed, such a parameterization of $(\alpha_{\lambda,\epsilon},\beta_{\lambda,\epsilon})$ is possible and is explicit in the standard algorithms for the combinatorial form of the Springer correspondence, as presented for example in \cite[Section 13.3]{carter}. We explicate it further in Section \ref{sect:zeros}.

\subsubsection{Shoji--Waldspurger--La tableaux}

The following is a restatement, using modified notation, of the construction outlined in \cite[Section 1.4]{waldspurger} and \cite[Section 1.2]{la}.
Our use of tableaux is an iterative implementation of the recursive form of the algorithm that appears in the cited references. 

For an integer $\ell\geq1$, we say that $T = (d_{i,j})$ is an \textit{$\ell$-tableau}, when indices $r_1\geq \ldots \geq r_c\geq 1$ are given so that $(i,j) \mapsto d_{i,j}$ is a bijection between the sets
\[
\{(i,j)\::\: 1\leq i \leq c,\; 1\leq j\leq r_i\} \,\to\, \{1,\ldots,\ell\}\;.
\]
For pairs of integers we assume a lexicographic order, so that $(i,1) < (i,2) < \ldots < (i+1,1)$.
For $\lambda,\epsilon$ as before and each choice $\Delta,\tau$ of positive integers, we construct the set $\mathcal{R}(\lambda,\epsilon,\Delta,\tau)$ of $\ell(\lambda)$-tableaux.
The entries of each $T = (d_{i,j})\in \mathcal{R}(\lambda,\epsilon,\Delta,\tau)$ are produced algorithmically. Assuming $d_{i',j'}$ has been defined for all $(i',j')< (i,j)$, we define $d_{i,j}$ as follows:
For $i\geq 1$ and $j\geq 2$, we write $u = -\overline{\epsilon}( d_{i,j-1})$, and set
\[
d_{i,j} = \min\{ e^u_k \::\: 1\leq k,\mbox{s.t. } d_{i,j-1}< e^u_k \neq d_{i',j'},\mbox{ for any }(i',j')< (i,j)\}\;,
\]
if exists ($r_i = j-1$ is set, if not).
For $i\geq1$ and $j=1$ a choice is made. For each sign $u\in \{\pm1\}$ we write
\[
k^u = \min\{ e^u_k \::\: 1\leq k,\mbox{s.t. } e^u_k \neq d_{i',j'},\mbox{ for any }(i',j')< (i,1)\}\;,
\]
when exists. Then, define $d_{i,1} = k^u$, for any sign $u$, for which either the inequality 
\begin{equation}\label{eq:cond-walds}
\gamma_{k^u}^{\lambda,\epsilon}\geq  -u\left( \Delta -\tau \sum_{m=1}^{i-1} \overline{\epsilon}( d_{m,1})\right)
\end{equation}
is satisfied, or $k^{-u}$ does not exist.

\subsubsection{Distinguished constituents}
For each $\ell(\lambda)$-tableau $T = (d_{i,j})\in \mathcal{R}(\lambda,\epsilon,\Delta,\tau)$, we write the tuple of integers $(s^T_1,\ldots, s^T_c)$
given by 
\[
s^T_i = \sum_{j=1}^{r_i}  \gamma^{\lambda,\epsilon}_{d_{i,j}}\,, \quad i=1,\ldots, c\;.
\]
The partitions 
\[
\alpha_T = \bigcup_{1\leq i\leq c,\; \overline{\epsilon}( d_{i,1})=1} (s^T_i)\;,\qquad \beta_T = \bigcup_{1\leq i\leq c,\; \overline{\epsilon}( d_{i,1})=-1} (s^T_i)\;,
\]
are constructed, so that $|\alpha_T| + |\beta_T| = n_G$ still holds.
For each choice of $\Delta,\tau$, we denote by
\[
P(\lambda,\epsilon, \Delta,\tau) = \{(\alpha_T,\beta_T)\::\: T\in \mathcal{R}(\lambda,\epsilon,\Delta,\tau)\} \subset \irr(W_{n_G})
\]
the set of representations that are produced in this manner.

\begin{theorem}\label{thm:waldsla}
Let $\lambda\in \mathcal{P}^{s_G}_0(N_G)$ be a good parity partition.
Let $\epsilon\in \widehat{A(\mathcal{O}^\vee_{\lambda})}_0$ be a character for which $\sigma(\mathcal{O}_{\lambda}^\vee, \epsilon)$ is an irreducible representation (Springer-type). Let $\Delta,\tau$ be a choice of positive integers.
Then, for each $(\alpha,\beta)\in P(\lambda,\epsilon, \Delta,\tau)$, we have
\[
\left\langle (\alpha,\beta)\,,\, \Sigma(\mathcal{O}_{\lambda}^\vee, \epsilon) \right\rangle = 1\;.
\]

\end{theorem}
\begin{proof}
Follows from \cite[Proposition 4.2]{waldspurger} and \cite[Proposition 3.1(iii)]{waldspurger} in the case of $s_G=-1$. The analogous results for $s_G=1$ are stated in \cite[Propositions 3.9,3.10]{la}.
\end{proof}

\begin{remark}

Our set $P(\lambda,\epsilon, \Delta,\tau)$ is denoted as $P_{\Delta,0,\tau}(\alpha_{\lambda,\epsilon},\beta_{\lambda,\epsilon}, <)$ in \cite{waldspurger} and \cite{la}, where $<$ is the specific order on the parts of $(\alpha_{\lambda,\epsilon},\beta_{\lambda,\epsilon})$ that is obtained through the Springer correspondence.

\end{remark}

\subsubsection{On maximality}
The analysis in \cite{waldspurger} and \cite{la} gave another characterization of the irreducible constituents of Springer representations that occur in the setting of Theorem \ref{thm:waldsla}. In particular, it provided sufficient conditions under which the vanishing of the multiplicity $\left\langle (\alpha,\beta)\,,\, \Sigma(\mathcal{O}_{\lambda}^\vee, \epsilon) \right\rangle$ may be determined by an inclusion $(\alpha,\beta)\in P(\lambda,\epsilon, \Delta,\tau)$.
To that aim we would like to define certain families of partial orders on the set $\irr(W_n)$.

Let us extend the domain of partitions to the set $\mathcal{P}\subset \overline{\mathcal{P}}$, so that each $\lambda\in \overline{\mathcal{P}}$ is given by an infinite sequence
\[
\lambda = (\lambda_1 \geq \lambda_2 \geq \lambda_3 \geq\ldots)
\]
of integers. 
The subset $\mathcal{P}$ of partitions is then naturally viewed as sequences with a constant tail of $0$'s.
We say that $\lambda \leq \mu$ holds, for $\lambda = (\lambda_i)_{i=1}^\infty,\mu=(\mu_i)_{i=1}^\infty\in \overline{\mathcal{P}}$, whenever
\[
\sum_{i=1}^k \lambda_i \leq \sum_{i=1}^k \mu_i
\]
holds, for all $1\leq k$. 
Restricting the resulting partial order to $\mathcal{P}(N) \subset \overline{\mathcal{P}}$, we see the familiar dominance order on partitions.
For $\lambda,\mu\in \overline{\mathcal{P}}$, the operation $\lambda\cup \mu\in \overline{\mathcal{P}}$ is similarly defined as in the case of partitions.
For $\lambda = (\lambda_i)_{i=1}^\infty\in \overline{\mathcal{P}}$ and a choice of non-negative integers $\Delta, \tau$, we define
\[
R_{\Delta,\tau}(\lambda) = (\Delta + \lambda_i - \tau(i-1))_{i=1}^\infty\in \overline{\mathcal{P}}\;.
\]
For a pair $\alpha,\beta\in \overline{\mathcal{P}}$, we write
\[
\Lambda_{\Delta,\tau}(\alpha,\beta): = R_{\Delta,\tau}(\alpha) \cup R_{0,\tau}(\beta)\in \overline{\mathcal{P}}\;.
\]
Now, we impose a partial order $\leq_{\Delta,\tau}$ on $\irr (W_n)$ by setting
\[
(\alpha_1,\beta_1) \leq_{\Delta,\tau} (\alpha_2,\beta_2)\; \Leftrightarrow \; \Lambda_{\Delta,\tau}(\alpha_1,\beta_1) \leq \Lambda_{\Delta,\tau}(\alpha_2,\beta_2)
\]
on each pair of bi-partitions $(\alpha_1,\beta_1),(\alpha_2,\beta_2)\in \irr(W_n)$.

\begin{theorem}\label{thm:maximal}
For $\lambda,\epsilon, \Delta,\tau$ as in Theorem \ref{thm:waldsla}, the set $P(\lambda,\epsilon, \Delta,\tau)$ equals the set of maximal elements in 
\[
\left\{ \sigma \in \irr(W_n)\;:\; \left\langle \sigma\,,\, \Sigma(\mathcal{O}_{\lambda}^\vee, \epsilon) \right\rangle \neq 0 \right\}
\]
with respect to the order $\leq_{\Delta,\tau}$.

\end{theorem}
\begin{proof}
This is \cite[Proposition 3.1(ii)]{waldspurger} combined with \cite[Lemme 1.6]{waldspurger} and the analogous treatment of \cite{la}.
\end{proof}

\subsection{Combinatorics of the canonical character subgroup}
Let $\lambda\in \mathcal{P}^{s_G}_0(N_G)$ (a partition of good parity) and $\epsilon\in\widehat{A(\mathcal O^\vee_\lambda)}_0$ be fixed, along with all additional notation defined in the previous section.
For any $a\in S(\lambda)$, we let $1\leq i(a)\leq \ell(\lambda)$ be the minimal index with $\lambda_{i(a)} = a$. 
We write 
\[
X_{\lambda} = \{i(a)\}_{a\in S(\lambda)}\subset \{1,\ldots,\ell(\lambda)\}
\]
for the resulting set of indices.
We also denote
\[
X_{\lambda,\epsilon} = \{i\in X_{\lambda}\::\: \epsilon(\lambda_i)= 1 -\epsilon(\lambda_{i-1})\}\;,
\]
where $\epsilon(\lambda_0) = 0$ is assumed.
Let us write
\[
X_{\lambda} = X_{\lambda}^1 \sqcup X_{\lambda}^{-1}\;,
\]
with $X^1_{\lambda}$ denoting the even indices and $X^{-1}_{\lambda}$ the odd ones.
Let us denote the subsets
\[
S(\lambda)_{\max} = \{\widetilde\theta_{\max}\}_{\theta\in S^{\dagger}(\lambda)}\setminus\{\infty\}\,,\;S(\lambda)_{\min} = \{\widetilde\theta_{\min}\}_{\theta\in S^{\dagger}(\lambda)}\setminus \{0\}\;\subset\, S(\lambda) \;.
\]


\begin{lemma}\label{lem:walds-dagger}
\begin{enumerate}
    \item\label{it:lem-walds1} 
An equality

\[
X^{s_G}_{\lambda} = \{i\in X_{\lambda}\::\: \lambda_i\in S(\lambda)_{\max} \}
\]

holds.
    
\item\label{it:lem-walds2}
The inclusion $\epsilon\in A^{\dagger}(\mathcal{O}^\vee_{\lambda})$ holds, if and only if, we have $X_{\lambda,\epsilon}\subset X^{s_G}_{\lambda} $.

\end{enumerate}
\end{lemma}

\begin{proof}
    Most of the statement follows directly from the definition of $S^{\dagger}(\lambda)$, of $P^{\dagger}(\lambda)_0$ (Section \ref{sect:special-piece}) and its identification with $A^{\dagger}(\mathcal{O}^\vee_{\lambda})$ (Proposition \ref{prop:quotient}).

One issue that remains to be clarified for the proof of \eqref{it:lem-walds2} is the behavior of the index $1\not\in X^{1}_{\lambda}$ in the case of $s_G=1$.
Indeed, in that case Lemma \ref{lem:head} implies that $\epsilon(\lambda_1)=0$, whenever $\epsilon\in A^{\dagger}(\mathcal{O}^\vee_{\lambda})_0$. Consequently, the exclusion $1\not\in X_{\lambda,\epsilon}$ becomes part of the defining condition for  $\epsilon\in A^{\dagger}(\mathcal{O}^\vee_{\lambda})$.

\end{proof}

\subsubsection{Zero parts}\label{sect:zeros}

We would like to strengthen the sufficient condition for an inclusion $\epsilon\in A^{\dagger}(\mathcal{O}^\vee_{\lambda})$ that was given in Lemma \ref{lem:walds-dagger}\eqref{it:lem-walds2}.
To that aim, we need to explicate some ingredients that provide the translation of the labelling $\sigma(\mathcal{O}^\vee_{\lambda},\epsilon)\in \irr(W_{n_G})$ into the bi-partition parameterization.
In particular, we are now interested in characterizing occurrences of $\gamma_i^{\lambda,\epsilon}=0$ in the sequence $\{\gamma_i^{\lambda,\epsilon}\}_{i=1}^{\ell(\lambda)}$ produced in Section \ref{sect:spring-comb}. 
We recall that $\lambda\in \mathcal{P}^{s_G}_0(N_G)$ is a partition of good parity.
We define the relative defect
\[
D_{\epsilon}(i) = \sum_{\lambda_i>a\in S(\lambda)_{\max}} \epsilon(a) - \sum_{\lambda_i>a\in S(\lambda)_{\min}} \epsilon(a)\;,
\]
for all $0\leq i\leq \ell(\lambda)$, assuming $\lambda_0 > \lambda_1$.
By comparing $D_\epsilon(0)$ with the usual notion of defect, we see by \cite[\S 12-13]{Lusztig1984}, that the non-vanishing of $\sigma(\mathcal{O}^\vee_{\lambda},\epsilon)$ (Springer-type) is equivalent to $D_{\epsilon}(0) = 0$.

We recall the formulas that produce those integers, following \cite[Section 13.3]{carter}.
We first denote
\[
\widetilde{\gamma}^{\lambda}_i =  \left\{ \begin{array}{ll} \lceil \lambda_i /2 \rceil &  i\mbox{ even,} \\
\lfloor \lambda_i /2 \rfloor &  i\mbox{ odd,}
\end{array} \right.
\]
for all $1\leq i \leq \ell(\lambda)$. In these terms, we have
\begin{equation}\label{eq:formula-springer}
\gamma^{\lambda,\epsilon}_i =   \left\{ \begin{array}{ll}   \widetilde{\gamma}^{\lambda}_i -2s_G\overline{\epsilon}(i)D_{\epsilon}(i)+ (-1)^i\epsilon(\lambda_i)m_G &  \lambda_i\not\in S(\lambda)_{\min} \\  
\widetilde{\gamma}^{\lambda}_i -2s_G\overline{\epsilon}(i)D_{\epsilon}(i) + (-1)^i\epsilon(\lambda_i)m'_G  &  \lambda_i\in  S(\lambda)_{\min}
\end{array} \right.\;,
\end{equation}
where 
\[
(m_G,m'_G) = \left\{\begin{array}{ll}
  (-2,0)    & s_G=1 \\
   (1,-1)  & s_G=-1 
\end{array}
\right.\;.
\]

\begin{lemma}\label{lem:zeroos}
    Suppose that $1\leq i\leq \ell(\lambda)$ is an index, for which $D_{\epsilon}(i)\in \{1, 0,-1\}$ and $\gamma_i^{\lambda,\epsilon}=0$ hold.
Then,
\begin{enumerate}
    \item\label{it-zeros:0} We have $\lambda_i\in \{1,2,3,4,5,6,7\}$.

    \item\label{it-zeros:2} When $\lambda_i =2$ and $2\not\in S(\lambda)_{\min}$, $i$ is odd.

    \item\label{it-zeros:3} When $\lambda_i=3$, $i$ is even.
    
    \item\label{it-zeros:4} When $\lambda_i =4$, we have one of:
    \begin{enumerate}
        \item\label{it-zeros:4a} $D_{\epsilon}(i)=1$ and $\epsilon(4)=0$.
        \item\label{it-zeros:4b}  $4\not\in S(\lambda)_{\min}$ and $i$ is odd.
    \end{enumerate}

    \item\label{it-zeros:5} When $\lambda_i =5$, we have one of:
    \begin{enumerate}
        \item\label{it-zeros:5a} $D_{\epsilon}(i)=1$ and $\epsilon(5)=0$.
        \item\label{it-zeros:5b}  $1,3\in S_0(\lambda)$ and $\epsilon(1) =1$, $\epsilon(3) = 0$, $\epsilon(5) = 1$.
    \end{enumerate}
        
\item\label{it-zeros:6} When $\lambda_i =6$, we have one of
\begin{enumerate}
    \item\label{it-zeros:6a} $D_{\epsilon}(i)=1$ and $i$ is odd.
    \item\label{it-zeros:6b} $2,4\in S_0(\lambda)$, $6\in S(\lambda)_{\min}$ and $\epsilon(2) =1$, $\epsilon(4) = 0$, $\epsilon(6) = 1$.

\end{enumerate}
\item\label{it-zeros:7} When $\lambda_i =7$, we have $D_{\epsilon}(i) = 1$ and $7\not\in S(\lambda)_{\min}$.

\end{enumerate}
\end{lemma}

\begin{proof}
\begin{enumerate}
    \item 

In case of $s_G=-1$, according to \eqref{eq:formula-springer} we have $|\gamma_i^{\lambda,\epsilon}-\widetilde{\gamma}_i^{\lambda}|\leq 3$. Thus,  $\gamma_i^{\lambda,\epsilon}=0$ can be reached only when $\lambda_i\in \{2,4,6\}$.
In case of $s_G=1$, according to \eqref{eq:formula-springer} we have $|\gamma_i^{\lambda,\epsilon}-\widetilde{\gamma}_i^{\lambda}|\in \{0,2, 4\}$.
For $\lambda_i =9$, with odd $i$ and $\widetilde{\gamma}_i^{\lambda}= 4$, we would obtain 
\[
\gamma_i^{\lambda,\epsilon}\geq 4 - 2\overline{\epsilon}(i)D_{\epsilon}(i)\geq 2\;.
\]

\item Evident from the formula, since $D_{\epsilon}(i) = 0$.

\item Follows from the same argument as in \eqref{it-zeros:0}.

\item When $\lambda_i = 4$ and $D_{\epsilon}(i) = 0$, we have $\gamma_i^{\lambda,\epsilon}\geq 1$.
If $D_{\epsilon}(i) = -1$ holds, it necessarily implies $2\in S(\lambda)_{\min}$, $\epsilon(4) = 0$ and $\overline{\epsilon}(i) = 1$. The statement follows.
Otherwise, $D_{\epsilon}(i)=1$ implies $\epsilon(4) = 0$. 

\item When $\lambda_i = 5$, as in a previous argument we must have an odd $i$ and $\widetilde{\gamma}_i^{\lambda}= 2$. 
If $D_{\epsilon}(i)=0$, the formula implies that $\gamma_i^{\lambda,\epsilon}\geq \widetilde{\gamma}_i^{\lambda}$, which is a contradiction.
If $D_{\epsilon}(i) = 1$, we see $\overline{\epsilon}(i)= 1$. In turn, that shows that $\epsilon(5) = 0$ ($i$ being odd) holds.
Otherwise, $D_{\epsilon}(i)=-1$ and $\overline{\epsilon}(i)= -1$ are implied. Thus, $\epsilon(5) = 1$. In particular, we must be in the $5\in S(\lambda)_{\min}$ case of \eqref{eq:formula-springer}.
Now, $D_{\epsilon}(i)=-1$ and $\lambda_i = 5\in S(\lambda)_{\min}$ can happen, only when $\theta\in S^{\dagger}(\lambda)$ exists, so that $\theta_{\min} = 1$, $\theta_{\max} = 3$, $\epsilon(1) =1$ and $\epsilon(3) =0$.
\item For $\lambda_i=6$, we must have $\epsilon(6) =1$. Now, either $D_{\epsilon}(i)=1$ and $\overline{\epsilon}(i)=-1$ hold, in which case $i$ must be odd, or $D_{\epsilon}(i)=-1$ and $\overline{\epsilon}(i)=1$ hold.
The latter case implies that $i$ is even. It also implies that we are in the $6\in S(\lambda)_{\min}$ case of the formula.
Now, $D_{\epsilon}(i)=-1$ and $\lambda_i = 6\in S(\lambda)_{\min}$ can happen, only when $\theta\in S^{\dagger}(\lambda)$ exists, so that $\theta_{\min} = 2$, $\theta_{\max} = 4$, $\epsilon(2) =1$ and $\epsilon(4) =0$.
\item 
Suppose that $\lambda_i =7$, with even $i$ and $\widetilde{\gamma}_i^{\lambda}= 4$. 
We must have $7\not\in S(\lambda)_{\min}$ for the formula to reach $\gamma_i^{\lambda, \epsilon} - \widetilde{\gamma}_i^{\lambda}= 4$.
The possibility of $D_{\epsilon}(i)\neq 1$ would force $\overline{\epsilon}(i) = -1$ and $\epsilon(7) = 1$, causing a contradiction.

\end{enumerate}
 
\end{proof}

\begin{proposition}\label{lem:zerotail}
    Let $\epsilon\in \widehat{A(\mathcal{O}^\vee_{\lambda})}_0$ be a Springer-type character, and $i\in  X_{\lambda}$ be an index with $\gamma^{\lambda,\epsilon}_i=0$.
    Suppose that for all $i> a\in X_{\lambda,\epsilon}$, we have $a\in X^{s_G}_{\lambda}$. 
    Then, $\epsilon\in A^{\dagger}(\mathcal{O}^\vee_{\lambda})$ holds.
\end{proposition}

\begin{proof}
By the assumption we know that $D_\epsilon(0) - D_{\epsilon}(i-1) \in \{0,1\}$. 
Since $\epsilon$ is of Springer-type, we have  $D_{\epsilon}(0)=0$, and so $D_{\epsilon}(i-1) \in \{0,-1\}$. 
In case $\lambda_i\in S(\lambda)_{\max}$, we can further deduce that $D_{\epsilon}(i-1)=0$. Thus, in all cases we have $D_{\epsilon}(i)\in \{1,0,-1\}$. This fact situates us in the scope of Lemma \ref{lem:zeroos}.
In particular, we have $1\leq \lambda_i\leq 7$.

    When $\lambda_i\in \{1,2\}$, we have $\lambda_i = \lambda_{\ell(\lambda)}$. If in addition $\lambda_i\in S_0(\lambda)$, the value of $\epsilon(\lambda_i)$ is determined uniquely by $D_{\epsilon}(i-1)$, so that the inclusion $\epsilon\in A^{\dagger}(\mathcal{O}^\vee_{\lambda})$ holds.
    When $\lambda_i=1\not\in S_0(\lambda)$, the value of $\epsilon(1)$ is inconsequential to the inclusion. 
    When $\lambda_i=2\not\in S_0(\lambda)$, we may have $2\in S(\lambda)_{\max}$ and then $\epsilon(2)$ is similarly inconsequential. Otherwise, when $2\not\in S(\lambda)_{\max}$, we must also have $2\not\in S(\lambda)_{\min}$ and by Lemma \ref{lem:zeroos}\eqref{it-zeros:2}, $i$ is odd. Hence, $i\in X_{\lambda}^{s_G}$ contradicts Lemma \ref{lem:walds-dagger}\eqref{it:lem-walds1}.
    The case of $\lambda_i=3$ follows similarly: Either $\{1,3\}\cap S_0(\lambda)$ is empty, or it is a singleton set $\{\mu\}$, in which case $\epsilon(\mu)$ is favorably determined by the value of $D_{\epsilon}(i-1)$. When $\mu=1$, we have $3\not\in S(\lambda)_{\max}$. Yet, by Lemma \ref{lem:zeroos}\eqref{it-zeros:3}, $i\in X_{\lambda}^{s_G}$ gives a similar contradiction.
    The last option could be $(1,3) = (\theta_{\min},\theta_{\max})$, for $\theta\in S^{\dagger}(\lambda)$. Here, $D_{\epsilon}(i-1) = D_{\epsilon}(0) = 0$ implies $\epsilon(1) = \epsilon(3)$, and  $\epsilon\in A^{\dagger}(\mathcal{O}^\vee_{\lambda})$.
Cases \eqref{it-zeros:4a} and \eqref{it-zeros:5a} of Lemma \ref{lem:zeroos} are impossible, since these would imply 
\[
D_{\epsilon}(i-1) = D_{\epsilon}(i) = 1\;.
\]
Suppose now that $\lambda_i\in S_0(\lambda)\cap S(\lambda)_{\min}$.
This assumption directly rules out cases \eqref{it-zeros:4b} and \eqref{it-zeros:7} of the lemma. 
Cases \eqref{it-zeros:5b} and \eqref{it-zeros:6b} are ruled out, since those suggest $D_{\epsilon}(i-1) = -2$.
Finally, the case of Lemma \ref{lem:zeroos}\eqref{it-zeros:6a} has $\lambda_i =6$ with an odd $i$. Here we note that $i\in X_{\lambda}^{s_G}$ holds, and that $S_0(\lambda)\cap S(\lambda)_{\min}, S(\lambda)_{\max}$ are disjoint sets. Hence, Lemma \ref{lem:walds-dagger}\eqref{it:lem-walds1} rules out this case.
We are left with the situation of $\lambda_i\not\in S_0(\lambda)\cap S(\lambda)_{\min}$.
Here, the case of $\lambda_i=4$ is treated identically to the prior case of $\lambda_i=3$.
We now also have $D_{\epsilon}(i)<1$, which takes cases \eqref{it-zeros:6a} and \eqref{it-zeros:7} of the lemma out of our scope.
Thus, we are left with \eqref{it-zeros:5b} or \eqref{it-zeros:6b}. In these cases we see that $\lambda_i\in S(\lambda)_{\max}$. That implies $D_{\epsilon}(i-1) = 0$ on one hand, but also $D_{\epsilon}(i-1) = -1$ follows from the description of those cases.

\end{proof}

\subsection{Proof of Proposition \ref{prop:dagger}}
\label{sec:dagger}

We fix the integers
\[
(\Delta_G,\tau_G) = \left\{ \begin{array}{ll}  (n_G +1 , 1) & s_G = 1 \\
(n_G+1, 2n_G+1) & s_G=-1
\end{array} \right.  \;,
\]
which will provide us with a useful partial order $\leq_{\Delta_G,\tau_G}$ in each of the cases.
We first reduce the property of weak $s$-sphericity into an algorithmic study as in the previous section.

\begin{lemma}\label{lem:algoreduction}
    Let $\lambda\in \mathcal P_0^{s_G}(N_G)$ be a good parity partition, and $\epsilon\in \widehat{A(\mathcal O^\vee_\lambda)}_0$ a character of Springer-type.
    
    The representation $\Sigma(\mathcal O^\vee_\lambda,\epsilon)$ is weakly $s_G$-spherical, if and only if, an index $0\leq i\leq n_G$ exists, so that  
    \[
    E^{s_G}_i\in P(\lambda,\epsilon,\Delta_G,\tau_G)
    \]
holds.
    
\end{lemma}

\begin{proof}
One direction follows immediately from Theorem \ref{thm:waldsla}.

Conversely, let us assume that $E^{s_G}_j$ appears as a sub-representation of $\Sigma(\mathcal O^\vee_\lambda,\epsilon)$.
Writing $E^{s_G}_j = (\alpha,\beta)$, it follows from Theorem \ref{thm:maximal} that $(\alpha',\beta')\in P(\lambda,\epsilon,\Delta_G,\tau_G)$ exists, for which an inequality
\[
(\mu_i)_{i=1}^\infty = \Lambda_{\Delta_G,\tau_G}(\alpha,\beta) \leq \Lambda_{\Delta_G,\tau_G}(\alpha',\beta') = (\mu'_i)_{i=1}^\infty 
\]
holds.
Let us also write $\alpha' = (\alpha'_1 \geq \alpha'_2 \geq \ldots)$ and $\beta' = (\beta'_1 \geq \beta'_2 \geq \ldots)$.

In case $s_G=1$, we have $(\alpha,\beta) = ((n_G-j,j),\emptyset)$. Since $\Delta_G - \tau_G \geq n_G$ holds in that case, we clearly have 
\[
\mu_1 = \Delta_G + n_G- j,\; \mu_2 = \Delta_G + j -\tau_G,\; \mu'_1 = \Delta_G + \alpha'_1,\; \mu'_2 = \Delta_G + \alpha'_2 - \tau_G\;.
\]
The inequality $\mu_1 + \mu_2\leq \mu'_1 + \mu'_2$ then forces $n_G \leq \alpha'_1 + \alpha'_2$. The last inequality evidently implies that $(\alpha',\beta') = E^1_{j'}$, for an index $j'$.

In case $s_G=-1$, we have $(\alpha,\beta) = ((n_G-j),(j))$. Since $n_G < \Delta_G$ and $\Delta_G - \tau_G  + a< 0$ holds for any $0\leq a < n_G$ in that case, we clearly have 
\[
\mu_1 = \Delta_G + n_G- j,\; \mu_2 = j,\; \mu'_1 = \Delta_G + \alpha'_1,\; \mu'_2 = \beta'_1\;.
\]
Same as in the previous case, we obtain $n_G \leq \alpha'_1 + \beta'_1$, which implies that $(\alpha',\beta') = E^{-1}_{j'}$, for an index $j'$.

\end{proof}

Let $\lambda\in \mathcal{P}^{s_G}_0(N_G)$ and $\epsilon\in\widehat{A(\mathcal O^\vee_\lambda)}_0$ of Springer-type now be fixed, along with all additional notation defined in Section \ref{sect:algo-green}.

\begin{proposition}\label{prop:first-row}
For any $\ell(\lambda)$-tableau $T = (d_{i,j})\in \mathcal{R}(\lambda,\epsilon,\Delta_G,\tau_G)$, we have
\[
\{1,\ldots, \ell(\lambda)\}\setminus \{d_{1,1}, d_{1,2},\ldots, d_{1,r_1}\} = X_{\lambda,\epsilon}\;.
\]

\end{proposition}

\begin{proof}
By definition of $\overline{\epsilon}$ it is easily verified that a description
\[
X_{\lambda,\epsilon} = \{1\leq i\leq \ell(\lambda)\::\: \overline{\epsilon}(i)= \overline{\epsilon}(i-1)\}\;
\]
holds, where $\overline{\epsilon}(0) = -1$ is assumed.
Since $\Delta_G> n_G$, equation \eqref{eq:cond-walds} forces $\overline{\epsilon}(d_{1,1})=1$. 
In particular, $d_{1,1}$ is defined to be the minimal index in $\{1,\ldots,\ell(\lambda)\}\setminus X_{\lambda,\epsilon}$. Similarly, $d_{1,2}$ is defined to be the minimal index in $\{1,\ldots,\ell(\lambda)\}\setminus X_{\lambda,\epsilon}$ that is larger than $d_{1,1}$, and so on.

\end{proof}

\begin{proposition}\label{prop:final}
An index $0\leq i\leq n_G$ exists, so that  
    \[
    E^{s_G}_i\in P(\lambda,\epsilon,\Delta_G,\tau_G)\;,
    \]
if and only if, an inclusion $\epsilon\in A^{\dagger}(\mathcal{O}^\vee_{\lambda})$ holds.
\end{proposition}

\begin{proof}

Let us denote the set of indices $X_{\lambda,\epsilon} = \{i_1 < \ldots < i_f\}$.
Thus, $(-1)^{\epsilon(\lambda_{i_j})} = (-1)^j$ holds, for all $1\leq j \leq f$. In other words, we have
\[
\overline{\epsilon}(i_j) = (-1)^{j + i_j-1}\;,
\]
for all $1\leq j \leq f$.
By Lemma \ref{lem:walds-dagger}, we see that the inclusion $\epsilon\in A^{\dagger}(\mathcal{O}^\vee_{\lambda})$ becomes  equivalent to the validity of
\begin{equation}\label{eq:proof-dagger}
\overline{\epsilon}(i_j) = s_G(-1)^{j -1}\;,
\end{equation}
for all $1\leq j \leq f$.
One consequence of Theorem \ref{thm:maximal} is that an $\ell(\lambda)$-tableau $T = (d_{i,j})\in \mathcal{R}(\lambda,\epsilon,\Delta_G,\tau_G)$ exists, with $c\geq1$ rows.
By Proposition \ref{prop:first-row}, we know that 
\[
\{1,\ldots,\ell(\lambda)\}\setminus\{d_{1,1},\ldots,d_{1,r_1}\} = \{i_1, \ldots , i_f\}\;.
\]
Because of $\Delta_G> n_G$, equation \eqref{eq:cond-walds} implies $\overline{\epsilon}(d_{1,1})=1$. 
Also, $s_G(\Delta_G- \tau_G)\geq n_G$ implies by the same reasoning that when $c>1$, $\overline{\epsilon}(d_{2,1})=s_G$ holds, unless $\overline{\epsilon}(i_j)= -s_G$ is valid for all $j$. 
Suppose first that $\epsilon\in A^{\dagger}(\mathcal{O}^\vee_{\lambda})$. Then, from \eqref{eq:proof-dagger} and the algorithm that produces $T$, we see that $c\leq 2$ and the bi-partition $(\alpha_T,\beta_T)$ is of the form $E^{s_G}_i$.
Conversely, suppose that $(\alpha_T,\beta_T)$ is of the form $E^{s_G}_i$.
Let $1\leq j_0\leq f$ be the minimal index for which $\overline{\epsilon}(i_{j_0}) = s_G(-1)^{j_0}$ holds, if exists.
When $j_0$ does not exist, \eqref{eq:proof-dagger} is valid and $\epsilon\in A^{\dagger}(\mathcal{O}^\vee_{\lambda})$. Otherwise, for all $i_{j_0} > a\in X_{\lambda,\epsilon}$, we have $a\in X^{s_G}_{\lambda}$. 
If $\gamma:= \gamma^{\lambda,\epsilon}_{i_{j_0}}= 0$ holds, from Proposition \ref{lem:zerotail} we still have $\epsilon\in A^{\dagger}(\mathcal{O}^\vee_{\lambda})$.
Thus, we are left to treat the case of $\gamma> 0$.
By construction of $T$, we have $d_{2,j} = i_j$, for all $1\leq j <j_0$, and $i_{j_0} = d_{c',1}$, for an index $2< c'$. 
It follows that $s^T_{c'} \geq \gamma$ contributes a non-zero part to $\alpha_T$, when $\overline{\epsilon}(i_{j_0})=1$, or to $\beta_T$, when $\overline{\epsilon}(i_{j_0})=-1$.

Suppose first that $s_G=1$.
Then, $|\beta_T|=0$ and $\overline{\epsilon}(i_{j_0})=1$ must hold. Also, $\overline{\epsilon}(d_{1,1})=\overline{\epsilon}(d_{2,1})=1$. This implies 
\[
\gamma^{\lambda,\epsilon}_{d_{1,1}}\geq \gamma^{\lambda,\epsilon}_{d_{2,1}} \geq \gamma^{\lambda,\epsilon}_{d_{c',1}} = \gamma >0\;,
\]
exhibiting three non-zero parts $s^T_1, s^T_2, s^T_{c'}$ for $\alpha_T$. This is a contradiction.

Now suppose that $s_G=-1$.
We have either $\overline{\epsilon}(i_{j_0})=1$, and consequently $\gamma^{\lambda,\epsilon}_{d_{1,1}}\geq \gamma^{\lambda,\epsilon}_{d_{c',1}}>0$ holds, or $\overline{\epsilon}(i_{j_0})=-1$, which similarly implies $\gamma^{\lambda,\epsilon}_{d_{2,1}}\geq \gamma^{\lambda,\epsilon}_{d_{c',1}}>0$.
In both cases, one of $\alpha_T, \beta_T$ will have two non-zero parts, which is a contradiction.

\end{proof}

Proposition \ref{prop:dagger} now follows from Lemma \ref{lem:algoreduction} and Proposition \ref{prop:final}, Remark \ref{rem:shoji} and the fact that all characters in $A^{\dagger}(\mathcal{O}^\vee_{\lambda})$ are of Springer-type.




\begin{sloppypar} 
\printbibliography[title={References}] 
\end{sloppypar}

\end{document}